\documentclass[a4paper,12pt,reqno]{amsart}  

\usepackage{amsmath}
\usepackage{tikz-cd}
\usepackage{amssymb}
\usepackage{amsthm}
\usepackage{amsfonts}
\usepackage{mathtools}
\usepackage{comment}
\usepackage{mathrsfs}
\usepackage{pifont}
\usepackage[noadjust]{cite} 
\usepackage{enumerate}
\usepackage{here}
\usepackage{autobreak}
\usepackage{wrapfig}
\usepackage{graphicx}
\usepackage{lscape}
\usepackage[colorlinks]{hyperref}  
\usepackage{xcolor}
\hypersetup{
bookmarksnumbered=true,
    colorlinks=true,
    citecolor=blue,
    linkcolor=purple,
    urlcolor=orange,
}
\usepackage{mleftright} 
\usepackage{extpfeil} 
\usepackage{tikz}
\usetikzlibrary{arrows} 
\usetikzlibrary{cd, decorations.pathmorphing} 

\usepackage{pgfplots}
\pgfplotsset{compat=1.18}

\usepackage[capitalize, nameinlink]{cleveref}
\creflabelformat{enumi}{#2(#1)#3}

\everymath{\displaystyle}
\numberwithin{equation}{section}

\theoremstyle{plain}

\newtheorem{thm}{Theorem}[section]
\crefname{thm}{Theorem}{Theorems}
\newtheorem{conj}[thm]{Conjecture}
\crefname{conj}{Conjecture}{Conjectures}

\newtheorem{lem}[thm]{Lemma}
\crefname{lem}{Lemma}{Lemmas}
\newtheorem{prop}[thm]{Proposition}
\crefname{prop}{Proposition}{Propositions}
\newtheorem{cor}[thm]{Corollary}
\crefname{cor}{Corollary}{Corollaries}

\newtheorem*{claim*}{Claim}
\newtheorem*{thm*}{Theorem}

\newtheorem{introthm}{Theorem}[section]

\crefname{introthm}{Theorem}{Theorems}
\crefname{introprop}{Proposition}{Propositions}
\crefname{introcor}{Corollary}{Corollaries}

\theoremstyle{definition}

\newtheorem{dfn}[thm]{Definition}
\crefname{dfn}{Definition}{Definitions}

\newtheorem{setup}[thm]{Setup}
\crefname{setup}{Setup}{Setups}

\crefname{eg}{Example}{Examples}
\newtheorem*{Ack}{Acknowledgement}

\newtheorem*{Out}{Outline of the paper}

\newtheorem*{Rels}{Related Works}
\newtheorem*{AI}{AI Usage}

\theoremstyle{remark}
\newtheorem{rem}[thm]{Remark}
\crefname{rem}{Remark}{Remarks}

\usepackage{enumitem} 

\DeclareMathOperator{\Ext}{Ext}

\DeclareMathOperator{\Pic}{Pic}
\DeclareMathOperator{\Aut}{Aut}

\DeclareMathOperator{\Stab}{Stab}

\DeclareMathOperator{\Spec}{Spec}
\DeclareMathOperator{\Bl}{Bl}

\DeclareMathOperator{\Ker}{Ker}
\DeclareMathOperator{\ch}{ch}
\DeclareMathOperator{\td}{td}
\DeclareMathOperator{\rk}{rk}

\DeclareMathOperator{\id}{id}

\DeclareMathOperator{\GL}{GL}
\DeclareMathOperator{\PGL}{PGL}

\DeclareMathOperator{\Cone}{Cone}
\DeclareMathOperator{\Image}{Im}

\DeclareMathOperator{\Perf}{Perf}

\DeclareMathOperator{\End}{End}
\DeclareMathOperator{\Supp}{Supp}

\DeclareMathOperator{\Alb}{Alb}
\DeclareMathOperator{\Gal}{Gal}

\DeclareMathOperator{\Gr}{Gr}

\DeclareMathOperator{\ev}{ev}

\newcommand\Hom{\mathop{\mathrm{Hom}}\nolimits}
\newcommand\RHom{\mathop{\mathbf{R}\mathrm{Hom}}\nolimits}
\newcommand{\PP}[1]{\mathbb{P}^{#1}}

\newcommand{\abs}[1]{\lvert{#1}\rvert}
\newcommand{\norm}[1]{\lVert{#1}\rVert}
\newcommand{\lderived}{\mathbf{L}}
\newcommand{\rderived}{\mathbf{R}}
\newcommand{\GLt}{\widetilde{\GL}^{+}(2,\RB)}

\newcommand{\twoheadlongrightarrow}{\relbar\joinrel\twoheadrightarrow}

\newcommand\dual{\raise0.9ex\hbox{$\scriptscriptstyle\vee$}}

\newcommand{\Spbdl}{S}
\newcommand{\mgrpsch}{\mathbf{G}_m}

\newcommand{\Knum}{K_{\mathrm{num}}}

\newcommand{\Stabnum}{\Stab}
\newcommand{\Stabgeo}{\Stab^{\mathrm{geo}}}

\newcommand{\res}{\operatorname{res}}

\DeclareMathOperator{\Coh}{Coh}

\newcommand{\CB}{\mathbb{C}}
\newcommand{\DB}{\mathbb{D}}

\newcommand{\HB}{\mathbb{H}}

\newcommand{\RB}{\mathbb{R}}
\newcommand{\PB}{\mathbb{P}}
\newcommand{\QB}{\mathbb{Q}}
\newcommand{\ZZ}{\mathbb{Z}}

\newcommand{\AC}{\mathcal{A}}
\newcommand{\BC}{\mathcal{B}}
\newcommand{\CC}{\mathcal{C}}
\newcommand{\DC}{\mathcal{D}}
\newcommand{\EC}{\mathcal{E}}
\newcommand{\FC}{\mathcal{F}}

\newcommand{\HC}{\mathcal{H}}
\newcommand{\IC}{\mathcal{I}}

\newcommand{\LC}{\mathcal{L}}

\newcommand{\NC}{\mathcal{N}}
\newcommand{\OC}{\mathcal{O}}
\newcommand{\PC}{\mathcal{P}}
\newcommand{\QC}{\mathcal{Q}}
\newcommand{\RC}{\mathcal{R}}
\newcommand{\SC}{\mathcal{S}}
\newcommand{\TC}{\mathcal{T}}
\newcommand{\UC}{\mathcal{U}}

\newcommand{\XC}{\mathcal{X}}
\newcommand{\YC}{\mathcal{Y}}
\newcommand{\ZC}{\mathcal{Z}}

\newcommand{\OO}{\mathscr{O}}

\author{Genki Ouchi}
\address[G.Ouchi]{Department~of~Mathematics, Faculty~of~Science, Hokkaido~University, Kita 10, Nishi 8, Kita-Ku, Sapporo, Hokkaido, 060-0810, Japan}
\email{\href{mailto:genki.ouchi@math.sci.hokudai.ac.jp}{genki.ouchi@math.sci.hokudai.ac.jp}}

\author{Tomoki Yoshida}
\address[T.Yoshida]{Department~of~Mathematics, School~of~Science~and~Engineering, Waseda~University, Ohkubo~3-4-1, Shinjuku, Tokyo~169-8555, Japan}
\email{\href{mailto:tomoki_y@asagi.waseda.jp}{tomoki\_y@asagi.waseda.jp}}

\title[Disconnectedness of stability manifolds]{Disconnectedness of stability manifolds and their geometric loci, and separation of braid orbits of semiorthogonal decompositions}
\date{October 30, 2026}
\keywords{Bridgeland stability condition, geometric stability condition, spherical twist, quintic threefold, semiorthogonal decomposition, braid group action, spinor bundle.}
\subjclass[2020]{14F08 (primary), 14D05, 14J32, 14J45, 18G80 (secondary).}

\begin{document}

\begin{abstract}
This paper proves that the Bridgeland stability manifold $\mathrm{Stab}(X)$ has infinitely many connected components for general quintic threefolds, other general complete intersection Calabi--Yau threefolds, certain odd-dimensional Calabi--Yau varieties, and certain rational weak Fano varieties.
Moreover, we show that the geometric locus $\mathrm{Stab}^{\mathrm{geo}}(X)$ meets distinct connected components of $\mathrm{Stab}(X)$ for general quintic threefolds and their products with varieties with finite Albanese morphisms, which gives examples in any dimension at least three.
For the weak Fano examples, we also construct semiorthogonal decompositions in infinitely many braid group orbits.

The proof is an extension of Nordskova's argument to families, using semilinear autoequivalences and absolute Galois actions on geometric generic fibers to permute mutually orthogonal spherical objects.
Then, we specialize the separation to closed fibers.
In our examples, the spherical objects are constructed from rigid rational curves or spinor bundles on quadrics.
\end{abstract}

\maketitle

\setcounter{tocdepth}{2} 
\tableofcontents
\setcounter{section}{-1} 
\section{Introduction}\label{sec:introduction}

Let $X$ be a smooth projective variety over $\CB$, and let $D^b(X)$ denote the bounded derived category of coherent sheaves on $X$.
We write $\Stab(X)$ for the space of numerical Bridgeland stability conditions on $D^b(X)$ satisfying the support property with respect to the numerical Grothendieck group $\Knum(X)$.
Bridgeland introduced stability conditions in \cite{bridgeland_2007_stability_conditions_on_triangulated_categories}, motivated by Douglas's work on $\Pi$-stability in string theory.
The deformation theorem states that the map 
\[
\Stab(X)\longrightarrow\Hom(\Knum(X),\CB),
\quad (Z,\PC)\longmapsto Z
\]
is a local homeomorphism. 
In particular, $\Stab(X)$ carries a natural complex manifold structure 
(see also \cite{bayer_2019_a_short_proof_of_the_deformation_property_of_bridgeland_stability_conditions} for the formulation using the support property).

The general existence of stability conditions remained open for many years.
Recently, Li \cite{preprint_chunyi_2026_a_remark_on_stability_conditions_on_smooth_projective_varieties} proved the existence of stability conditions on every smooth projective variety over $\CB$.
Li--Liu--Liu--Macr\`i--Perry--Stellari--Zhao \cite{preprint_chunyi_zhiyu_ziqi_emanuele_alexander_paolo_xiaolei_2026_stability_conditions_and_moduli_spaces_on_projective_families} extended the construction to projective families, and Cheng \cite{preprint_yiran_2026_a_remark_on_the_full_support_property} proved that the resulting stability conditions satisfy the support property with respect to the full numerical Grothendieck group.
It is then natural to study the global topology of $\Stab(X)$, and in particular its connectedness.

In the context of mirror symmetry, stability conditions on Calabi--Yau threefolds are of particular interest.
For instance, mirror symmetry predicts an embedding of the stringy K\"ahler moduli space into the quotient of $\Stab(X)$ by autoequivalences and the natural $\CB$-action \cite[Remark~3.9]{bridgeland_2009_spaces_of_stability_conditions}.
\emph{Quintic threefolds} are one of the most well-studied examples in this context.
Therefore, the spaces of stability conditions on quintic threefolds have been particularly believed to have some ``good'' properties.

However, Nordskova \cite[Theorem~6.3]{preprint_anya_2026_a_smooth_projective_counterexample_to_bondalpolishchuks_conjecture} proved that $\Stab(X)$ is disconnected for a special one-parameter family of smooth quintic threefolds admitting an involution exchanging two disjoint $(-1,-1)$-lines (she has informed us that, independently of the present work, she has extended her result to very general quintic threefolds, and that this extension will appear in a forthcoming revision of her paper).
Her argument uses the following data \cite[Lemma~6.1]{preprint_anya_2026_a_smooth_projective_counterexample_to_bondalpolishchuks_conjecture} and proves that the spherical twist of $E_1$ sends $\sigma_0$ to a different component:
\begin{itemize}
\item two orthogonal spherical objects $E_1,E_2\in D^b(X)$;
\item an automorphism $g\colon X\xrightarrow{\sim}X$ whose pullback exchanges $E_1$ and $E_2$ and acts trivially on $\Knum(X)$;
\item a stability condition $\sigma_0\in\Stab(X)$ fixed by $g^*$.
\end{itemize}
In this symmetric quintic example, the spherical objects are the sheaves $\OC_{\ell_i}(-1)$ on the two lines $\ell_1,\ell_2$.

A stability condition is called \emph{geometric} if all skyscraper sheaves are stable of the same phase.
We write $\Stabgeo(X)\subset\Stab(X)$ for the locus of geometric stability conditions.
The locus $\Stabgeo(X)$ is contractible when $X$ is a smooth projective curve or surface
\cite{macri_2007_stability_conditions_on_curves,preprint_nick_2024_contractibility_of_the_geometric_stability_manifold_of_a_surface}.
Therefore, it is natural to ask whether $\Stabgeo(X)$ is also contractible in higher dimensions.

The key observation of this paper is that the automorphism in Nordskova's argument can be replaced by the absolute Galois action on the geometric generic fiber of a family.
We extend her argument to semilinear autoequivalences and apply it to the geometric generic fiber. We then use the theory of stability conditions in families to show that the corresponding spherical twists also give rise to distinct connected components on closed fibers.
This yields infinitely many connected components of $\Stab(X)$ for a broad range of smooth projective varieties, including general quintic threefolds, without requiring automorphisms of the individual fibers exchanging the spherical objects.
We also study geometric stability conditions in these components and obtain analogous separation results for braid group orbits of semiorthogonal decompositions.

\subsection{Results}\label{subsection: results in introduction}

Our main result gives the following examples of varieties with disconnected stability manifolds:

\begin{introthm}[See \cref{theorem: disconnectedness for general quintics,theorem: other complete intersection Calabi--Yau threefolds,theorem: disconnectedness for the Dwork pencil,theorem: product criterion for distinguished components,proposition: disconnectedness for anticanonical divisors in quadrics,proposition: disconnectedness for weak Fano blowups,proposition: disconnectedness for point blowups of del Pezzo threefolds}]
\label{theorem in intro: disconnectedness examples}
The stability manifold $\Stab(X)$ has infinitely many connected components for each of the following smooth projective varieties $X$:
    \begin{enumerate}
        \item\label{item in intro: general quintics}
        a general quintic threefold;
        \item\label{item in intro: other complete intersection threefolds}
        a general complete intersection Calabi--Yau threefold in projective space;
        \item\label{item in intro: Dwork quintics}
        a smooth member of the Dwork pencil (see \cref{definition: Dwork pencil});
        \item\label{item in intro: quadric Calabi--Yau varieties}
        a smooth complete intersection of type $(2,n+1)$ in $\PP{n+2}$ whose defining quadric is smooth, where $n\geq3$ is odd;
        \item\label{item in intro: weak Fano blowups}
        the blow-up $\Bl_C\PP{n}$, where $n\geq3$ is odd and $C$ is a smooth complete intersection of type $(2,n+1)$ whose defining quadric is smooth;
        \item\label{item in intro: point blowups}
        the blow-up $\Bl_pV_5$ at a general point $p\in V_5$, where $V_5$ is the smooth Fano threefold of Picard rank one, index two, and degree five;
        \item\label{item in intro: products of quintics}
        the product of one of the above types and a variety with finite Albanese morphism.
    \end{enumerate}
    In each case, there are stability conditions with the same central charge in infinitely many distinct connected components.
\end{introthm}

We first give a criterion for spherical twists to separate connected components of stability manifolds, and then extend it to families.
We use the following setting. 
Let $B$ be a smooth quasi-projective variety over $\CB$, and $f\colon\XC\to B$ a smooth projective morphism of relative dimension $n\geq2$ with geometrically connected fibers, $T\to B$ a connected finite \'etale cover of degree $d\geq2$, and $E\in\Perf(\XC_T)$. 
Fix $\sigma_\bullet\in\Stab(\XC/B)$, where $\Stab(\XC/B)$ is the space of stability conditions over $B$ with respect to $\Knum(\XC/B)$ (\cref{definition:stability_condition_over_B}), and let $\Sigma_b$ be the connected component of $\Stab(\XC_b)$ containing $\sigma_b$ for $b\in B(\CB)$. 
Assume that 
there exist $A_1,\ldots,A_r\in\Perf(\XC)$ whose numerical classes form a basis of $\Knum(\XC_{\bar b})_\QB$ on every geometric fiber $\XC_{\bar b}$ and, for every geometric point $\bar b$ of $B$, 
the objects $E_t\coloneqq E|_{\XC_{\bar b}\times\{t\}}$ for $t\in T_{\bar b}$ are mutually orthogonal $n$-spherical objects.

\begin{introthm}[See \cref{theorem: disconnectedness in families}]
\label{theorem in intro: disconnectedness in families}
In the setting above, choose an ordering $T_b=\{t_1,\ldots,t_d\}$ for $b\in B(\CB)$ and put
$\Phi_{\mathbf a,b}\coloneqq T_{E_{t_1}}^{a_1}\cdots T_{E_{t_d}}^{a_d}$ for $\mathbf a=(a_1,\ldots,a_d)\in\ZZ^d$.
Then, for $b\in B(\CB)$ and $\mathbf a,\mathbf a'\in\ZZ^d$, we have
\[
\Phi_{\mathbf a,b}\Sigma_b=\Phi_{\mathbf a',b}\Sigma_b
\quad\Longrightarrow\quad
\mathbf a-\mathbf a'\in\ZZ(1,\ldots,1).
\]
In particular, $\Stab(\XC_b)$ has infinitely many connected components for every $b\in B(\CB)$.
\end{introthm}

\begin{rem}
As we stated above, the proof is an extension of Nordskova's argument to semilinear equivalences and absolute Galois action on the geometric generic fiber.
The numerical hypothesis and the connectedness of $T$ make this action trivial on the numerical Grothendieck group and transitive on the spherical objects.
Using stability conditions over discrete valuation rings, we then deduce the same separation of connected components on closed fibers (\cref{lemma: specialization of connected components}).
Note that the argument in \cite{preprint_anya_2026_a_smooth_projective_counterexample_to_bondalpolishchuks_conjecture} uses an automorphism of the variety exchanging orthogonal spherical objects.
By contrast, in \cref{theorem in intro: disconnectedness in families}, the Galois action arising from the connected cover $T\to B$ replaces this automorphism, so no such automorphism of the closed fibers is required.
\end{rem}

We construct the spherical objects using two classical geometric constructions.
On threefolds, pushforwards of line bundles on disjoint rational curves with normal bundle $\OC_{\PP{1}}(-1)^{\oplus2}$ give mutually orthogonal spherical objects.
For the examples involving quadrics, we use the two spinor bundles on a smooth even-dimensional quadric $Q$.
Their restrictions to a smooth anticanonical divisor, or their pushforwards from a divisor $Q\subset X$ with $N_{Q/X}\cong\omega_Q$, are mutually orthogonal spherical objects (\cref{lemma: exceptional spherical constructions}).

Note that the generality assumption on the quintic is used to obtain a family of disjoint rigid lines.
More precisely, there is a nonempty open subset $U$ of the parameter space of smooth quintic threefolds such that the relative Hilbert scheme of lines over $U$ is a connected finite \'etale cover of degree $2875$, and distinct lines on each fiber are disjoint (\cref{remark: lines on general quintic Calabi--Yau threefolds}).
The smooth members of the Dwork pencil lie outside $U$, since their Hilbert schemes of lines are positive-dimensional \cite{mustac_2013_degree_1_curves_in_the_dwork_pencil_and_the_mirror_quintic}.
We treat these members separately, using rigid lines when $\psi\ne0$ and rigid conics on the Fermat quintic.
Thus, $U$ does not exhaust the locus of smooth quintics with disconnected stability manifolds.

We next show that distinct connected components can contain geometric stabilities. 

\begin{introthm}[See \cref{setup: geometric wall,theorem: geometric components}]
\label{theorem in intro: geometric disconn general case}
Fix $\sigma_0=(Z_0,\PC_0)\in\Stabgeo(X)$ such that $\phi_{\sigma_0}(\OC_x) = 1$.
Let $\Sigma$ be the connected component containing $\sigma_0$, and $\Phi\in\Aut D^b(X)$ satisfy $\Phi^{-1}\Sigma\ne\Sigma$.
Assume that there exist a $\sigma_0$-stable object $S$ of phase one and an integer $r\geq1$ such that, for $x\in X(\CB)$, there is an exact sequence
\[
0\longrightarrow\OC_x\longrightarrow\Phi(\OC_x)
\longrightarrow S^{\oplus r}\longrightarrow0
\]
in $\PC_0(1)$.
Then $\Phi^{-1}\Sigma$ contains a geometric stability condition.
In particular, $\Stabgeo(X)$ is disconnected.
\end{introthm}

We apply this theorem to quintic threefolds.

\begin{introthm}[See \cref{theorem: geometric disconnectedness for general quintics,corollary: disconnectedness of products}]
\label{theorem in intro: geometric disconnectedness}
Let $X$ be a quintic threefold represented by a point of $U$ or a product of such a quintic and a variety with finite Albanese morphism.
    Then, $\Stabgeo(X)$ meets at least $2^{2875}-1$ distinct connected components of $\Stab(X)$.
    In particular, $\Stabgeo(X)$ is disconnected.
\end{introthm}

Let $\ell$ be a line on a quintic threefold $Y\subset\PP{4}$ represented by a point of $U$, and put
\[
E_\ell\coloneqq
\Ker\bigl(
H^0(Y,\IC_\ell(1))\otimes\OC_Y
\xrightarrow{\ev}\IC_\ell(1)
\bigr).
\]
We find a stability condition in Li's family \cite{li_2019_on_stability_conditions_for_the_quintic_threefold} at which $E_\ell[1]$ and all skyscraper sheaves are stable of phase one.
We then apply \cref{theorem in intro: geometric disconn general case} with $\Phi=T_{E_\ell}$ and $S=E_\ell[1]$.

\begin{rem}
In \cref{theorem in intro: disconnectedness in families}, orthogonality and equality of the numerical classes of the spherical objects force $n$ to be odd (\cref{remark: parity obstruction}).
The product construction gives examples in every dimension at least three for which $\Stab(X)$ has infinitely many connected components and $\Stabgeo(X)$ meets distinct connected components (\cref{corollary: disconnectedness of products}).
Taking the second factor to be an abelian variety gives such examples with $\omega_X\cong\OC_X$.
In particular, one cannot in general choose a single connected component containing all geometric stability conditions in higher dimensions.
\end{rem}

Finally, we apply a similar argument to the braid group action on semiorthogonal decompositions by mutations.
Bondal--Polishchuk proposed the following conjecture.

\begin{conj}[Bondal--Polishchuk {\cite[Conjecture~2.2]{bondal_polishchuk_1993_homological_properties_of_associative_algebras_the_method_of_helices}}]
\label{conjecture in intro: Bondal Polishchuk}
    Let $\DC$ be an Ext-finite $\CB$-linear triangulated category admitting a full exceptional collection of length $r$.
    Then the braid group $B_r$ acts transitively by mutations on full exceptional collections in $\DC$, up to shifts of the individual objects.
\end{conj}

Nordskova \cite{preprint_anya_2026_a_smooth_projective_counterexample_to_bondalpolishchuks_conjecture} constructed the first counterexample in the setting of a smooth projective variety, which is a rational weak Fano threefold with full exceptional collections in infinitely many braid group orbits.
We extend Nordskova's argument to semiorthogonal decompositions with an exceptional component and to semilinear autoequivalences.

\begin{introthm}[See \cref{theorem: separation of braid orbits}]
\label{theorem in intro: separation of braid orbits}
    Let $X$ be a smooth projective variety of dimension $n\geq2$ over an algebraically closed field $k$ of characteristic zero and $\DB=(\AC_1,\ldots,\AC_r)$ be a non-trivial semiorthogonal decomposition of $D^b(X)$, $F_1,\ldots,F_d$ be mutually orthogonal $n$-spherical objects, where $d\geq2$. 
    Put $\DB_{\mathbf a}\coloneqq (T_{F_1}^{a_1}\cdots T_{F_d}^{a_d})(\DB)$, for $\mathbf a=(a_1,\ldots,a_d)\in\ZZ^d$.
    Assume that the following hold:
    \begin{enumerate}
        \item a group $G$ of $k$-semilinear autoequivalences preserves each $\AC_j$ and acts transitively on $\{F_1,\ldots,F_d\}$ up to isomorphism;
        \item $\AC_j=\langle E\rangle$ for some exceptional object $E$ and some $j$, and $\RHom(E,F_1)\ne0$.
    \end{enumerate}
    Then, if $\DB_{\mathbf a}$ and $\DB_{\mathbf a'}$ belong to the same $B_r$-orbit, $\mathbf a-\mathbf a'\in\ZZ(1,\ldots,1)$.
\end{introthm}

We also prove a relative version of \cref{theorem in intro: separation of braid orbits}.
Here, the role of stability conditions in families is replaced by the moduli space of semiorthogonal decompositions of Belmans--Okawa--Ricolfi \cite{preprint_pieter_shinnosuke_andrea_2025_moduli_spaces_of_semiorthogonal_decompositions_in_families}.
Using the absolute Galois action on the geometric generic fiber, we obtain the same separation of braid group orbits on every geometric fiber (\cref{theorem: separation of braid orbits in families}).
Since this moduli space is \'etale over the base and mutations commute with base change, the locus where two relative semiorthogonal decompositions belong to the same braid group orbit is open.

\begin{rem}
The example $X=\Bl_pV_5$ also shows that the converse of \cref{theorem in intro: separation of braid orbits} does not hold (see \cref{remark: counter example of the converse of separation theorem}).
\end{rem}

As applications of the absolute and relative criteria, we obtain the following result.

\begin{introthm}[See \cref{proposition: disconnectedness for weak Fano blowups,proposition: disconnectedness for point blowups of del Pezzo threefolds}]
\label{theorem in intro: weak Fano braid orbits}
    Let $X$ be a variety as in \labelcref{item in intro: weak Fano blowups} or \labelcref{item in intro: point blowups} in \cref{theorem in intro: disconnectedness examples}, and $r=n+2$ in the first case and $r=6$ in the second case.
    Then, for $m\in\ZZ$, there are semiorthogonal decompositions of $D^b(X)$
    \[
        \DB_m=(\AC_{1,m},\ldots,\AC_{r,m}),
    \]
    in pairwise distinct $B_r$-orbits such that, for each $j$, the image of $K_0(\AC_{j,m})$ in $\Knum(X)$ is independent of $m$.

\end{introthm}

\begin{rem}
In the latter case, the decompositions are induced by full exceptional collections.
In particular, these collections give a counterexample to \cref{conjecture in intro: Bondal Polishchuk}.
Moreover, because these decompositions are obtained by autoequivalences acting trivially on $\Knum(X)$, their braid group orbits cannot be distinguished by numerical classes.
\end{rem}

\begin{Rels}
The stability manifolds of smooth projective curves \cite{okada_2006_stability_manifold_of_mathbb_p1,bridgeland_2007_stability_conditions_on_triangulated_categories,macri_2007_stability_conditions_on_curves}, abelian surfaces \cite{bridgeland_2008_stability_conditions_on_k3_surfaces,huybrechts_macri_stellari_2008_stability_conditions_for_generic_k3_categories}, and abelian threefolds of Picard rank one \cite{fu_li_zhao_2022_stability_manifolds_of_varieties_with_finite_albanese_morphisms} are known to be contractible.
For K3 surfaces, Bridgeland constructed a distinguished connected component and conjectured that it is simply connected, which is known in the Picard rank one case \cite{bayer_bridgeland_2017_derived_automorphism_groups_of_k3_surfaces_of_picard_rank_1}.
Similarly, the contractibility of the connected component of $\Stab(\PP{2})$ containing geometric stability conditions is shown in \cite{li_2017_the_space_of_stability_conditions_on_the_projective_plane}.

For geometric loci of curves, the explicit description $\Stabgeo(C)\cong \HB\times\CB$ is proven in \cite{kikuta_koseki_ouchi_2024_thurston_compactifications_of_spaces_of_stability_conditions_on_curves}.
In \cite{fu_li_zhao_2022_stability_manifolds_of_varieties_with_finite_albanese_morphisms}, the geometricity of stability conditions on a smooth projective variety with finite Albanese morphism is proven.
Moreover, the geometric locus of a smooth projective surface is described by using a generalized Le Potier function in \cite{dell_2025_stability_conditions_on_free_abelian_quotients}.
Rekuski's contractibility theorem recalled above uses this description.

Similar results are known for derived categories of quivers: the contractibility for stability manifolds of derived categories of Dynkin quivers \cite{qiu_woolf_2018_contractible_stability_spaces_and_faithful_braid_group_actions} and connectedness for the stability manifolds of finite connected acyclic quivers \cite{preprint_takumi_dongjian_2025_stability_conditions_and_algebraic_hearts_for_acyclic_quivers}.

On the other hand, examples of disconnected stability spaces were already known
for quotient categories \cite{meinhardt_partsch_2014_quotient_categories_stability_conditions_and_birational_geometry,barbieri_moller_qiu_so_2024_quadratic_differentials_as_stability_conditions_collapsing_subsurfaces}.
More recently, \cite{preprint_dawei_jingyin_yu_fei_2026_nonasphericity_of_strata_of_genusone_differentials_and_stability_spaces} proved the existence of noncontractible connected components.

The Bondal--Polishchuk conjecture is known for del Pezzo surfaces \cite{kuleshov_orlov_1994_exceptional_sheaves_on_del_pezzo_surfaces}.
However, counterexamples to \cref{conjecture in intro: Bondal Polishchuk} are constructed using derived categories of graded gentle algebras in \cite{chang_haiden_schroll_2025_braid_group_actions_on_branched_coverings_and_full_exceptional_sequences}. 
A purely algebraic proof of this non-transitivity result was given in \cite{preprint_atsuki_atsushi_2025_a_simple_algebraic_proof_of_the_nontransitivity_of_the_braid_group_action_on_full_exceptional_sequences}.
Then, a counterexample in the smooth projective setting is constructed in  \cite{preprint_anya_2026_a_smooth_projective_counterexample_to_bondalpolishchuks_conjecture}.
In addition, the blow-up of $\PP{2}$ at ten general points is known to have both a full exceptional collection and a non-full exceptional collection of maximal length \cite{krah_2024_a_phantom_on_a_rational_surface}.
\end{Rels}

\begin{Out}
In \cref{section: preliminaries}, we recall numerical stability conditions, their behavior under base change, stability conditions over a base, and relative Fourier--Mukai transforms.
In \cref{section: disconnectedness in families}, we establish criteria for spherical twists to separate connected components of stability manifolds.
We extend Nordskova's argument to semilinear autoequivalences in \cref{subsection:Spherical twists and connected components}, and prove \cref{theorem in intro: disconnectedness in families} in \cref{subsection: disconnectedness in families}.
In \cref{section: A local criterion for geometricity}, we prove \cref{theorem in intro: geometric disconn general case} and its family version.
The separation criterion for braid group orbits (\cref{theorem in intro: separation of braid orbits}) is proved in \cref{subsection: separation of braid orbits by semilinear autoequivalences}, and its family version is established in \cref{subsection: separation of braid orbits in families}.

\cref{section: Calabi--Yau threefolds via rigid rational curves} gives applications using rigid rational curves.
In \cref{subsection: rational curves on quintic threefolds}, we prove the assertions of \cref{theorem in intro: disconnectedness examples,theorem in intro: geometric disconnectedness} for general quintic threefolds.
The other complete intersection Calabi--Yau threefolds and the Dwork pencil in \cref{theorem in intro: disconnectedness examples} are treated in \cref{subsection: further calabi-yau 3-folds}.
In \cref{subsection: products and finite Albanese morphisms}, we prove the product cases of \cref{theorem in intro: disconnectedness examples,theorem in intro: geometric disconnectedness}.
Finally, \cref{section: examples from Spinor bundles} treats the remaining cases of \cref{theorem in intro: disconnectedness examples} and proves \cref{theorem in intro: weak Fano braid orbits}.
We first construct spherical objects from spinor bundles and apply the family criterion to Calabi--Yau complete intersections in smooth quadrics.
The blow-ups of projective space along complete intersections are treated in \cref{subsection: rational weak Fano varieties}, and the blow-up of $V_5$ at a general point is treated in \cref{subsection: further examples}.
\end{Out}

\begin{AI}
The authors used GPT-6 Astra throughout the preparation of this work.
The authors independently checked and reconstructed all mathematical arguments and proofs developed with AI assistance. The authors take full responsibility for the mathematical content of this paper, including any errors. 
\end{AI}

\begin{Ack}
We would like to thank Kohei Kikuta for useful conversations and for his support during the preparation of this manuscript. 
The first author was supported by JSPS KAKENHI Grant Numbers
JP25K00906 and JP25K06932.
    The second author acknowledges support from OpenAI through the ChatGPT for Academic Researchers program.
\end{Ack}
\section{Preliminaries on stability conditions}\label{section: preliminaries}

\subsection{Stability conditions}\label{subsection: Bridgeland stability conditions}

Let $\TC$ be a triangulated category.
Fix a finitely generated free abelian group $\Lambda$ and a group homomorphism $v \colon K_0(\TC) \to \Lambda$ with finite cokernel. 
For a group homomorphism $Z\colon\Lambda\to\CB$,
we write $Z(E)\coloneqq Z(v(E))$.

\begin{dfn}
[{\cite[Definition~5.1]{bridgeland_2007_stability_conditions_on_triangulated_categories}}]
A \emph{pre-stability condition} on $\TC$ with respect to $(\Lambda, v)$ is a pair
$\sigma=(Z,\PC)$ consisting of a group homomorphism
$Z\colon \Lambda \to \CB$, called the \emph{central charge}, 
and a slicing $\PC$ of $\TC$
such that, for any $\phi \in \mathbb{R}$ and $0\neq E\in\mathcal{P}(\phi)$, we have
$Z(E) \in \mathbb{R}_{>0}e^{\sqrt{-1}\pi\phi}$.
\end{dfn}

Let $\sigma=(Z,\PC)$ be a pre-stability condition on $\TC$.
By \cite[Lemma~5.2]{bridgeland_2007_stability_conditions_on_triangulated_categories},
the category $\mathcal{P}(\phi)$ is abelian for $\phi\in\mathbb{R}$.
A nonzero object $E \in \mathcal{P}(\phi)$ is called \emph{$\sigma$-semistable of phase $\phi$}. 
A $\sigma$-semistable object $E \in \mathcal{P}(\phi)$ is called
\emph{$\sigma$-stable} if it is simple in the abelian category
$\mathcal{P}(\phi)$.
For an interval $I\subset\RB$,
we denote by $\PC(I)$ the extension-closed subcategory
of $\TC$ generated by $\PC(\phi)$ for $\phi\in I$.
Every nonzero object $E\in \TC$ admits the \emph{Harder--Narasimhan filtration}:
\[
\begin{tikzcd}[column sep=1.2em, row sep=1.4em]
0=E_0 \arrow[rr]
&& E_1 \arrow[rr] \arrow[dl]
&& E_2 \arrow[rr] \arrow[dl]
&& \cdots \arrow[rr]
&& E_{n-1} \arrow[rr]
&& E_n=E \arrow[dl] \\
& A_1 \arrow[ul,dashed]
&& A_2 \arrow[ul,dashed]
&&&&&& A_n \arrow[ul,dashed]
\end{tikzcd},
\]
where 
$A_i\in\mathcal{P}(\phi_i)$ and $\phi_1>\phi_2 > \cdots> \phi_n$.
The objects $A_1,\ldots,A_n$ are called the
\emph{Harder--Narasimhan factors} of $E$.
We define
$\phi_\sigma^+(E)\coloneqq \phi_1,
\phi_\sigma^-(E)\coloneqq \phi_n$,
and the \emph{mass} of $E$ with respect to $\sigma$ by
\[m_\sigma(E)\coloneqq \sum_{i=1}^n |Z(A_i)|.\]

\begin{dfn}[{\cite{preprint_maxim_yan_2008_stability_structures_motivic_donaldsonthomas_invariants_and_cluster_transformations}}, {\cite[Lemma~11.4]{bayer_macri_stellari_2016_the_space_of_stability_conditions_on_abelian_threefolds_and_on_some_calabiyau_threefolds}}]
Fix a norm $\|-\|$ on $\Lambda_\RB$.
A pre-stability condition $\sigma=(Z,\PC)$ on $\TC$ is said to satisfy the \emph{support property} if the following equivalent conditions hold:
\begin{enumerate}
\item
There exists a constant $C>0$ such that
\[\|v(E)\| \leq C|Z(E)|\]
for any $\sigma$-semistable object $E \in \TC$. 
\item
There exists a quadratic form
$Q$ on $\Lambda_{\mathbb{R}}$ such that
$Q|_{\ker Z}$ is negative definite and
\[
  Q(v(E))\geq 0
\]
for any $\sigma$-semistable object $E\in\TC$.
\end{enumerate}
A pre-stability condition $\sigma$ on $\TC$ is called a \emph{stability condition} if $\sigma$ satisfies the support property. We denote by $\operatorname{Stab}_{(\Lambda,v)}(\TC)$ the set of stability conditions on $\TC$ with respect to $(\Lambda,v)$.
For a smooth projective variety $X$, we write $\Stab(X)$ for the set of stability conditions on $D^b(X)$ with respect to $\Lambda=\Knum(X)$.
\end{dfn}

Let $\sigma$ be a stability condition on $\TC$.
By \cite[Appendix~B]{bayer_macri_2011_the_space_of_stability_conditions_on_the_local_projective_plane},
$\sigma$ is locally finite.
Thus a $\sigma$-semistable object $E\in\mathcal{P}(\phi)$ admits a
Jordan--H\"older filtration in $\mathcal{P}(\phi)$ and
its $\sigma$-stable factors
are called the \emph{Jordan--H\"older factors} of $E$.

For $\sigma,\tau\in\operatorname{Stab}_{(\Lambda,v)}(\TC)$, we define
\[
d(\sigma,\tau)
\coloneqq
\sup_{0 \neq E \in \TC}
\left\{
\left|\phi_\sigma^+(E)-\phi_\tau^+(E)\right|,
\left|\phi_\sigma^-(E)-\phi_\tau^-(E)\right|,
\left|\log\frac{m_\sigma(E)}{m_\tau(E)}\right|
\right\} \in [0,\infty].
\]
By \cite[Proposition~8.1]{bridgeland_2007_stability_conditions_on_triangulated_categories},
the function $d$ defines a generalized metric on $\operatorname{Stab}_{(\Lambda,v)}(\TC)$.

We recall the deformation theorem for stability conditions.

\begin{thm}[{\cite[Theorem~1.2]{bridgeland_2007_stability_conditions_on_triangulated_categories}}]
\label{theorem: Bridgeland's deformation theorem}
The map 
\[ \Stab_{(\Lambda,v)}(\TC) \to \Hom(\Lambda, \CB),~(Z,\PC) \mapsto Z \]
is a local homeomorphism. In particular, $\Stab_{(\Lambda,v)}(\TC)$ is a complex manifold.
\end{thm}

Let $G$ be a subgroup of $\Aut(\TC)$, and
let $\alpha \colon G \to \GL(\Lambda)$ be a homomorphism satisfying 
$v \circ [\Phi]=\alpha(\Phi) \circ v$
for all $\Phi \in G$.
For $\Phi\in G$ and
$\sigma=(Z,\mathcal{P})\in\operatorname{Stab}_{(\Lambda,v)}(\TC)$, we set
\[\Phi(\sigma)
\coloneqq 
\bigl(Z\circ \alpha(\Phi)^{-1},\Phi(\mathcal{P})\bigr),\]
where 
$\Phi(\mathcal{P})(\phi)\coloneqq \Phi(\mathcal{P}(\phi))$
for $\phi\in\mathbb{R}$.

We write $\GLt$ for the universal cover of $\GL^+(2,\RB)$.
It acts continuously on the space of stability conditions;
see \cite[Lemma~8.2]{bridgeland_2007_stability_conditions_on_triangulated_categories}.

\subsection{Base change}\label{subsection:Base change}
Let $f\colon\XC\to B$ be a smooth projective morphism,
where $B$ is either a smooth quasi-projective variety
over $\CB$ or the spectrum of a stalk of the structure
sheaf of such a variety.
For a $B$-scheme $T$ and a point $t\in T$, we write
\[
\XC_T\coloneqq\XC\times_B T,
\quad
\XC_t\coloneqq\XC_T\times_T\Spec\kappa(t).
\]
For a $T$-perfect complex $E$ on $\XC_T$,
we denote by $E_t$ its derived restriction to $\XC_t$.
We assume that there exist
$A_1,\ldots,A_r\in\Perf(\XC)$
such that, for every geometric point $\bar b$ of $B$,
the classes
$[(A_1)_{\bar b}],\ldots,[(A_r)_{\bar b}]$
form a basis of $\Knum(\XC_{\bar b})_\QB$.
Let $\Lambda\coloneqq\Knum(\XC/B)$
be the relative numerical Grothendieck group defined in
\cite[Definition~21.1]{bayer_lahoz_macri_perry_stellari_2021_stability_conditions_in_families}.
For a $B$-scheme $T$ and $t\in T$,
we denote by
$[-]_t\colon K_0(\XC_t)\to\Lambda$
the natural homomorphism defined in
\cite[(8.2)]{preprint_chunyi_zhiyu_ziqi_emanuele_alexander_paolo_xiaolei_2026_stability_conditions_and_moduli_spaces_on_projective_families},
which factors through $\Knum(\XC_t)$.
We also write
$[-]_t\colon\Knum(\XC_t)\to\Lambda$
for the induced homomorphism.
By
\cite[Proposition and Definition~21.5]{bayer_lahoz_macri_perry_stellari_2021_stability_conditions_in_families},
the relative Euler pairing
$\chi\colon K_0(\Perf(\XC))\times\Lambda\to\ZZ$
satisfies
\[
\chi([F],[E]_t)=\chi(F_t,E)
\]
for any $B$-scheme $T$, $t\in T$,
$F\in\Perf(\XC)$, and $E\in D^b(\XC_t)$.

\begin{lem}\label{lemma: numerical identification}
For any geometric point $\bar b$ of $B$,
the homomorphism
\[
[-]_{\bar b}\colon
\Knum(\XC_{\bar b})_\QB
\to \Lambda_\QB
\]
is an isomorphism.
\end{lem}

\begin{proof}
Fix a geometric point $\bar b$ of $B$.
By
\cite[Definition~21.1]{bayer_lahoz_macri_perry_stellari_2021_stability_conditions_in_families}
and the assumption, $[-]_{\bar b} \colon \Knum(\XC_{\bar b})_\QB
\to \Lambda_\QB$ is surjective.
To prove injectivity, take
$w\in\ker([-]_{\bar b})$.
For each $1\leq i\leq r$, we have
$\chi((A_i)_{\bar b},w)
=\chi([A_i],[w]_{\bar b})=0$.
Since the classes
$[(A_1)_{\bar b}],\ldots,[(A_r)_{\bar b}]$
form a basis of $\Knum(\XC_{\bar b})_\QB$
and the Euler pairing is nondegenerate,
we obtain $w=0$.
\end{proof}

By \cref{lemma: numerical identification},
the nondegeneracy of the Euler pairing on the fibers,
and the torsion-freeness of $\Lambda$,
the relative Euler pairing gives an injection
\begin{equation}\label{equation: relative Euler embedding}
\Lambda\hookrightarrow\ZZ^r,\quad
\lambda\mapsto(\chi([A_i],\lambda))_{i=1}^r.
\end{equation}
Hence $\Lambda$ is a finitely generated free abelian group.
Moreover, by
\cite[Remark~8.8]{preprint_chunyi_zhiyu_ziqi_emanuele_alexander_paolo_xiaolei_2026_stability_conditions_and_moduli_spaces_on_projective_families},
for any connected $B$-scheme $T$
and any $T$-perfect complex $E$ on $\XC_T$,
the class $[E_t]_t\in\Lambda$
is independent of $t\in T$.
Since $\Knum(\XC_{\bar b})$ is torsion-free,
\cref{lemma: numerical identification} implies that
$[-]_{\bar b}$ is injective with finite cokernel.
We use the homeomorphism
\begin{equation}\label{equation: identification of stability spaces}
\Stab_{(\Lambda,[-]_{\bar b})}(D^b(\XC_{\bar b}))
\xrightarrow{\sim}
\Stab(\XC_{\bar b}),
\qquad
(Z,\PC)\longmapsto (Z\circ[-]_{\bar b},\PC)
\end{equation}
to identify these two spaces.

Let $b\in B$ and let $k'/\kappa(b)$ be a field extension.
Write
$\XC_{b,k'}\coloneqq
\XC_b\times_{\kappa(b)}\Spec k'$,
and let $\pi\colon\XC_{b,k'}\to\XC_b$
be the projection.
We regard $\Spec k'$ as a $B$-scheme. 
We write
\[
[-]_{b,k'}\colon K_0(\XC_{b,k'})\to\Lambda
\]
for the homomorphism $[-]_t$ defined above,
where $t$ is the unique point of $\Spec k'$.

We recall the base change theorem in the above setting.

\begin{thm}
[{\cite[(2.6), (8.2),Theorem~2.8 and Lemma~2.10]{preprint_chunyi_zhiyu_ziqi_emanuele_alexander_paolo_xiaolei_2026_stability_conditions_and_moduli_spaces_on_projective_families}}]
\label{thm:base_change}
Let $\sigma=(Z,\PC)$ be a stability condition
on $D^b(\XC_b)$ with respect to $(\Lambda,[-]_b)$.
There exists a canonical slicing $\PC_{k'}$
of $D^b(\XC_{b,k'})$ such that
$\sigma_{k'}\coloneqq(Z,\PC_{k'})$
is a stability condition with respect to
$(\Lambda,[-]_{b,k'})$.
For $E\in D^b(\XC_b)$ and $\phi\in\RB$,
the object $E$ is $\sigma$-semistable of phase $\phi$
if and only if $\pi^*E$ is
$\sigma_{k'}$-semistable of phase $\phi$.
Moreover, the map
\[
\Stab_{(\Lambda,[-]_b)}(D^b(\XC_b))
\to
\Stab_{(\Lambda,[-]_{b,k'})}(D^b(\XC_{b,k'})),
\quad
\sigma\mapsto\sigma_{k'}
\]
is an open embedding.
\end{thm}

\begin{dfn}
[{\cite[Definition~12.18]{bayer_lahoz_macri_perry_stellari_2021_stability_conditions_in_families},
\cite[Definition~2.9]{preprint_chunyi_zhiyu_ziqi_emanuele_alexander_paolo_xiaolei_2026_stability_conditions_and_moduli_spaces_on_projective_families}}]
Let $b\in B$, and let $k'$ be an algebraic closure
of $\kappa(b)$.
For $\sigma \in \Stab_{(\Lambda,[-]_b)}(D^b(\XC_b))$,
an object $E\in D^b(\XC_b)$ is called
\emph{geometrically $\sigma$-stable}
if the base change $E_{k'}$ is $\sigma_{k'}$-stable.
\end{dfn}

We recall the notion of semilinearity.

\begin{dfn}
Let $k$ be a field and $\varphi\in\Aut(k)$.
A map $f\colon V\to W$ between $k$-vector spaces
is called \emph{$\varphi$-semilinear}
if $f$ is additive and
$f(av)=\varphi(a)f(v)$
for all $a\in k$ and $v\in V$.
Let $\mathcal C,\mathcal D$ be $k$-linear categories.
An additive functor
$\Phi\colon\mathcal C\to\mathcal D$
is called \emph{$k$-semilinear}
if there exists $\varphi\in\Aut(k)$ such that,
for any $E,F\in\mathcal C$, the map
\[
\Phi \colon \Hom_{\mathcal C}(E,F)
\to
\Hom_{\mathcal D}(\Phi(E),\Phi(F))
\]
is $\varphi$-semilinear.
\end{dfn}

\begin{lem}\label{lem:Galois invariance}
Let $b\in B$, and let $k'$ be an algebraic closure
of $\kappa(b)$.
For $g\in\Gal(k'/\kappa(b))$,
we also denote by $g$ the induced automorphism
of $\XC_{b,k'}$ over $\XC_b$.
Then the $k'$-semilinear autoequivalence
$g^* \colon D^b(\XC_{b,k'})\xrightarrow{\sim} D^b(\XC_{b,k'})$
acts trivially on $\Knum(\XC_{b,k'})$.
Moreover, for any
$\sigma\in\Stab_{(\Lambda,[-]_b)}(D^b(\XC_b))$,
we have $g^*\sigma_{k'}=\sigma_{k'}$.
\end{lem}
\begin{proof}
Fix $g\in\Gal(k'/\kappa(b))$.
Since $\pi\circ g=\pi$, we have
$g^*((A_i)_{b,k'})\simeq(A_i)_{b,k'}$
for $1\leq i\leq r$, and
$R\pi_*\circ g^*\simeq R\pi_*$
as functors to $D_{\mathrm{qc}}(\XC_b)$.
Since the classes
$[(A_1)_{b,k'}],\ldots,[(A_r)_{b,k'}]$
form a basis of $\Knum(\XC_{b,k'})_\QB$,
we obtain $[g^*]=\id$ on $\Knum(\XC_{b,k'})$.
By
\cite[Theorem~2.8(1)]{preprint_chunyi_zhiyu_ziqi_emanuele_alexander_paolo_xiaolei_2026_stability_conditions_and_moduli_spaces_on_projective_families},
we also obtain $g^*\PC_{k'}=\PC_{k'}$.
Hence $g^*\sigma_{k'}=\sigma_{k'}$.
\end{proof}

\subsection{Stability conditions over a base}\label{subsection:Stability conditions over a base}
We recall some basic definitions and results on stability conditions over a base. For details, see \cite{bayer_lahoz_macri_perry_stellari_2021_stability_conditions_in_families} and \cite{preprint_chunyi_zhiyu_ziqi_emanuele_alexander_paolo_xiaolei_2026_stability_conditions_and_moduli_spaces_on_projective_families}.

We continue with the notation and assumptions of
\cref{subsection:Base change}.
Let $\sigma_\bullet=(\sigma_b)_{b\in B}$ be a collection
of stability conditions $\sigma_b=(Z,\PC_b)$
on $D^b(\XC_b)$ with respect to $(\Lambda,[-]_b)$,
with common central charge $Z\colon\Lambda\to\CB$. 
For a $B$-scheme $T$ and $t\in T$,
let $b\in B$ be the image of $t$.
By \cref{thm:base_change}, the base change
$\sigma_t\coloneqq(\sigma_b)_{\kappa(t)}$
is a stability condition on $D^b(\XC_t)$
with respect to $(\Lambda,[-]_t)$.
For a morphism $T\to B$ essentially of finite type,
we denote by
\[
[-]_T\colon\Knum(\XC_T/T)\to\Lambda
\]
the homomorphism constructed in
\cite[Definitions~8.7 and~8.17(iii)]{preprint_chunyi_zhiyu_ziqi_emanuele_alexander_paolo_xiaolei_2026_stability_conditions_and_moduli_spaces_on_projective_families}.

\begin{dfn}
[{\cite[Definition~8.17]{preprint_chunyi_zhiyu_ziqi_emanuele_alexander_paolo_xiaolei_2026_stability_conditions_and_moduli_spaces_on_projective_families}}]
\label{definition:stability_condition_over_B}
The collection $\sigma_\bullet=(\sigma_b)_{b\in B}$
is called a \emph{stability condition over $B$}
if the following conditions hold:
\begin{enumerate}
\item
There exists a quadratic form $Q$ on $\Lambda_\RB$
such that $Q|_{\ker Z}$ is negative definite and
\[
Q([E]_b)\geq 0
\]
for any $b\in B$ and any $\sigma_b$-semistable
object $E\in D^b(\XC_b)$.
This condition is also called the \emph{support property}.
\item
For any $B$-scheme $T$ and any $T$-perfect complex
$E$ on $\XC_T$, the subset
\[
\{t\in T\mid
E_t\text{ is geometrically $\sigma_t$-stable}\}
\]
is open in $T$.

\item
For any morphism $T\to B$ essentially of finite type
from a Dedekind scheme $T$,
the stability conditions $(\sigma_t)_{t\in T}$
are induced by a Harder--Narasimhan structure
on $D^b(\XC_T)$ over $T$
in the sense of
\cite[Definition~8.13]{preprint_chunyi_zhiyu_ziqi_emanuele_alexander_paolo_xiaolei_2026_stability_conditions_and_moduli_spaces_on_projective_families},
with respect to the homomorphism $[-]_T$.

\item
For any $w\in\Lambda$ and $\phi\in\RB$ satisfying
$Z(w)\in\mathbb{R}_{>0}e^{\sqrt{-1}\pi\phi}$,
the moduli functor
$\mathcal M^{\mathrm{st}}_{\sigma_\bullet}(w,\phi)$
defined in
\cite[Definition~8.16(1)]{preprint_chunyi_zhiyu_ziqi_emanuele_alexander_paolo_xiaolei_2026_stability_conditions_and_moduli_spaces_on_projective_families}
is bounded in the sense of
\cite[Definition~8.6]{preprint_chunyi_zhiyu_ziqi_emanuele_alexander_paolo_xiaolei_2026_stability_conditions_and_moduli_spaces_on_projective_families}.
\end{enumerate}
\end{dfn}

We denote by $\Stab(\XC/B)$
the set of stability conditions over $B$,
endowed with the coarsest topology making
all fiber restriction maps
\[
\Stab(\XC/B)
\to
\Stab_{(\Lambda,[-]_b)}(D^b(\XC_b)),
\quad
\sigma_\bullet\mapsto\sigma_b
\]
continuous.

For a morphism $T\to B$ essentially of finite type,
by
\cite[Remark~20.6(4) and Definition~21.15]{bayer_lahoz_macri_perry_stellari_2021_stability_conditions_in_families}
(see also
\cite[Definition~8.17]{preprint_chunyi_zhiyu_ziqi_emanuele_alexander_paolo_xiaolei_2026_stability_conditions_and_moduli_spaces_on_projective_families}),
the collection
$\sigma_\bullet|_T\coloneqq(\sigma_t)_{t\in T}$
is a stability condition over $T$
with respect to $(\Lambda,[-]_T)$.
Thus, base change induces a continuous map
\[
\Stab(\XC/B)
\to
\Stab_{(\Lambda,[-]_T)}(\XC_T/T),
\quad
\sigma_\bullet\mapsto\sigma_\bullet|_T.
\]

We use the following two properties: openness and semistable reduction.

\begin{prop}
[{\cite[Proposition~20.8]{bayer_lahoz_macri_perry_stellari_2021_stability_conditions_in_families}}]\label{proposition: openness_PI}
Let $\sigma_\bullet\in\Stab(\XC/B)$.
For any $B$-scheme $T$,
any $T$-perfect complex $E$ on $\XC_T$,
and any interval $I\subset\RB$, the subset
\[
\{t\in T\mid E_t\in\PC_t(I)\}
\]
is open in $T$.
\end{prop}

\begin{prop}
[{\cite[Lemma~21.22]{bayer_lahoz_macri_perry_stellari_2021_stability_conditions_in_families}}]
\label{proposition: semistable reduction}
Let $\sigma_\bullet\in\Stab(\XC/B)$
and let $T\to B$ be a morphism
essentially of finite type,
where $T=\Spec R$ for a discrete valuation ring $R$.
Write $\eta$ and $0$ for the generic and closed
points of $T$, respectively.
For any $\sigma_\eta$-semistable object
$F\in D^b(\XC_\eta)$ of phase $\phi$,
there exists $E\in D^b(\XC_T)$
such that $E_\eta\simeq F$ and
$E_0$ is $\sigma_0$-semistable of phase $\phi$.
\end{prop}

We recall the deformation theorem for stability
conditions over a base.

\begin{thm}
[{\cite[Theorem~22.2]{bayer_lahoz_macri_perry_stellari_2021_stability_conditions_in_families},
\cite[Theorem~8.19 and Remark~8.20]{preprint_chunyi_zhiyu_ziqi_emanuele_alexander_paolo_xiaolei_2026_stability_conditions_and_moduli_spaces_on_projective_families}}]
\label{thm:relative_deformation}
The map
\[
\ZC\colon\Stab(\XC/B)\to\Hom(\Lambda,\CB),
\quad
\sigma_\bullet\mapsto Z
\]
is a local homeomorphism.
More precisely, fix $\sigma_\bullet\in\Stab(\XC/B)$,
where $\sigma_b=(Z,\PC_b)$ for $b\in B$.
Let $Q$ be a quadratic form giving the uniform
support property for $\sigma_\bullet$.
Let $U_Q(Z)$ denote the connected component
containing $Z$ of the open subset
\[
\left\{
W\in\Hom(\Lambda,\CB)
\;\middle|\;
Q|_{\ker W}\text{ is negative definite}
\right\}.
\]
Then there exists an open neighborhood
$U$ of $\sigma_\bullet$ in $\Stab(\XC/B)$
such that
\[
\ZC|_U\colon U\to U_Q(Z)
\]
is a covering map.
\end{thm}

Following
\cite[Section~10]{preprint_chunyi_zhiyu_ziqi_emanuele_alexander_paolo_xiaolei_2026_stability_conditions_and_moduli_spaces_on_projective_families},
we recall the construction of stability conditions
over $B$ induced by projective embeddings.
We first fix notation for closed immersions.
Let $i\colon Y\hookrightarrow X$ be a closed immersion
of smooth projective varieties.
Following
\cite[Definition~3.1]{preprint_chunyi_zhiyu_ziqi_emanuele_alexander_paolo_xiaolei_2026_stability_conditions_and_moduli_spaces_on_projective_families},
for a slicing $\QC$ of $D^b(X)$, we set
\[
(i^\sharp\QC)(\phi)
\coloneqq
\{E\in D^b(Y)\mid i_*E\in\QC(\phi)\}
\]
for $\phi\in\RB$.
In general, $i^\sharp\QC$ need not be a slicing of $D^b(Y)$;
see
\cite[Remark~3.2]{preprint_chunyi_zhiyu_ziqi_emanuele_alexander_paolo_xiaolei_2026_stability_conditions_and_moduli_spaces_on_projective_families}.
For $\tau=(W,\QC)\in\Stab(X)$, we write
\[
i^\sharp\tau\coloneqq(W\circ i_*,i^\sharp\QC)
\]
if it is a stability condition in $\Stab(Y)$.

Returning to the family $\XC\to B$ and the lattice
$\Lambda=\Knum(\XC/B)$, fix a closed immersion
$j\colon\XC\hookrightarrow\PB^N_B$ over $B$.
For $b\in B$, let
$j_b\colon\XC_b\hookrightarrow\PB^N_{\kappa(b)}$
be the base change of $j$, and set
$H_b\coloneqq c_1(\OC_{\PB^N_{\kappa(b)}}(1))$.
For $\xi>0$ and $\theta\in\RB$, let
$\tau_{b,\xi,\theta}
=(W_{b,\xi,\theta},\QC_{b,\xi,\theta})
\in\Stab(\PB^N_{\kappa(b)})$
be the stability condition constructed in
\cite[(5.4), (5.5)]{preprint_chunyi_zhiyu_ziqi_emanuele_alexander_paolo_xiaolei_2026_stability_conditions_and_moduli_spaces_on_projective_families}.
By
\cite[Propositions~8.10 and~10.2]{preprint_chunyi_zhiyu_ziqi_emanuele_alexander_paolo_xiaolei_2026_stability_conditions_and_moduli_spaces_on_projective_families},
there is a homomorphism
$Z_{\xi,\theta}\colon\Lambda\to\CB$
such that
\[
Z_{\xi,\theta}([E]_b)
=
W_{b,\xi,\theta}((j_b)_*E)
=
-\int_{\PB^N_{\kappa(b)}}
e^{-(\theta+\sqrt{-1}\xi)H_b}
\operatorname{ch}((j_b)_*E)
\]
for $b\in B$ and $E\in D^b(\XC_b)$.

\begin{thm}
[{\cite[Theorems~6.2 and~10.3, Proposition~10.2]{preprint_chunyi_zhiyu_ziqi_emanuele_alexander_paolo_xiaolei_2026_stability_conditions_and_moduli_spaces_on_projective_families}}]
\label{thm:projective_family_stability}
For $\xi>0$, $\theta\in\RB$, and $b\in B$, set
$\sigma_{b,\xi,\theta}
\coloneqq
(Z_{\xi,\theta},j_b^\sharp\QC_{b,\xi,\theta})$.
There exists $\xi_0>0$ such that, for $\xi>\xi_0$
and $\theta\in\RB$, we have a stability condition
$\sigma_{\bullet,\xi,\theta}
\coloneqq
(\sigma_{b,\xi,\theta})_{b\in B}
\in\Stab(\XC/B)$.
Moreover, $\sigma_{\bullet,\xi,\theta}$
depends continuously on $(\xi,\theta)$.
For any geometric point $\bar b$ of $B$
and any closed point $x\in\XC_{\bar b}$,
the skyscraper sheaf $\mathcal O_x$
is $\sigma_{\bar b,\xi,\theta}$-stable of phase $1$.
\end{thm}

\begin{proof}
Let $v_H\colon\Lambda\to\Lambda_H$
be the homomorphism defined in
\cite[(10.1)]{preprint_chunyi_zhiyu_ziqi_emanuele_alexander_paolo_xiaolei_2026_stability_conditions_and_moduli_spaces_on_projective_families}.
By
\cite[Theorem~10.3]{preprint_chunyi_zhiyu_ziqi_emanuele_alexander_paolo_xiaolei_2026_stability_conditions_and_moduli_spaces_on_projective_families},
$\sigma_{\bullet,\xi,\theta}$ is a stability condition
over $B$ with respect to $(\Lambda_H,v_H)$.
By
\cite[Proposition~3.4]{preprint_yiran_2026_a_remark_on_the_full_support_property},
$\sigma_{\bullet,\xi,\theta}$ admits a uniform
mass--Hom bound.

Let $\NC(\XC/B)$ be the uniformly numerical
relative Grothendieck group defined in
\cite[Proposition and Definition~21.5]{bayer_lahoz_macri_perry_stellari_2021_stability_conditions_in_families}.
By \labelcref{equation: relative Euler embedding},
the canonical map
$\Lambda\to\NC(\XC/B)$ is an isomorphism.
Hence, by
\cite[Theorem~3.2]{preprint_yiran_2026_a_remark_on_the_full_support_property},
$\sigma_{\bullet,\xi,\theta}$ satisfies
the support property.
Thus, we have
$\sigma_{\bullet,\xi,\theta}\in\Stab(\XC/B)$.

Continuity follows from
\cite[Proposition~10.2 and Lemma~3.5(3)]{preprint_chunyi_zhiyu_ziqi_emanuele_alexander_paolo_xiaolei_2026_stability_conditions_and_moduli_spaces_on_projective_families},
and the assertion about skyscraper sheaves follows from
\cite[Theorem~6.2]{preprint_chunyi_zhiyu_ziqi_emanuele_alexander_paolo_xiaolei_2026_stability_conditions_and_moduli_spaces_on_projective_families}.
\end{proof}

For $\xi>\xi_0$ and $\theta\in\RB$, under the identification
\labelcref{equation: identification of stability spaces},
we have
$\sigma_{\bar b,\xi,\theta}
=
j_{\bar b}^\sharp\tau_{\bar b,\xi,\theta}$
for any geometric point $\bar b$ of $B$.

Let $X$ be a smooth projective variety over $\CB$, and let $H$ be
the numerical class of an ample line bundle on $X$.
Choose a closed immersion $j\colon X\hookrightarrow\PP{N}$ such that
$c_1(j^*\OC_{\PP{N}}(1))=mH$ for some integer $m>0$.
We denote by $\Stab^\dagger_H(X)$ the connected component of
$\Stab(X)$ containing the stability conditions induced by $j$
in \cref{thm:projective_family_stability} with $B=\Spec\CB$.
For an ample line bundle $L$ on $X$, we also write
$\Stab^\dagger_L(X)\coloneqq\Stab^\dagger_{c_1(L)}(X)$.

By
\cite[Theorem~6.8]{preprint_chunyi_zhiyu_ziqi_emanuele_alexander_paolo_xiaolei_2026_stability_conditions_and_moduli_spaces_on_projective_families}
and
\cite[Corollary~2.3]{preprint_yiran_2026_a_remark_on_the_full_support_property},
the component $\Stab^\dagger_H(X)$ is independent of $j$ and $m$.
It is also independent of $H$ by
\cite[Theorem~4.2]{preprint_yiran_2026_a_remark_on_the_full_support_property}.

\subsection{Relative Fourier--Mukai transforms}
\label{subsection:relative Fourier--Mukai transforms}

Let $f\colon\XC\to B$ be a smooth projective morphism,
with $B$ as in \cref{subsection:Base change}.
We use the notation of that section.
For \cref{proposition: relative FM equivalence,proposition: relative spherical twist},
we do not require the existence of the perfect complexes
$A_1,\ldots,A_r$ specified there.
Let
$p_1,p_2\colon\XC\times_B\XC\to\XC$
be the projections.
For $\mathcal K\in\Perf(\XC\times_B\XC)$,
the \emph{relative Fourier--Mukai transform} is
\[
\Phi_{\mathcal K}\colon D^b(\XC)\to D^b(\XC),
\quad
F\mapsto Rp_{2*}(p_1^*F\otimes^L\mathcal K).
\]
For a $B$-scheme $T$, let
$q\colon\XC_T\to\XC$ and
$j\colon\XC_T\times_T\XC_T
\to\XC\times_B\XC$
be the canonical morphisms.
Put $\mathcal K_T\coloneqq j^*\mathcal K$.
Then we have the relative
Fourier--Mukai transform
\[
\Phi_{\mathcal K_T}\colon
\Perf(\XC_T)\to\Perf(\XC_T),
\]
which is $T$-linear.

By
\cite[Theorem~4.16 and Lemma~3.4]{preprint_elías_pat_kabeer_nebojsa_2026_fiberwise_criteria_for_fouriermukai_equivalences}
and
\cite[Proposition/Definition~2.3 and Remark~2.5]{preprint_riku_2026_fouriermukai_partners_of_abelian_varieties_and_k3_surfaces_in_positive_and_mixed_characteristics},
we obtain the following.

\begin{prop}\label{proposition: relative FM equivalence}
Assume that $\Phi_{\mathcal K_{\bar b}}$
is an autoequivalence of $D^b(\XC_{\bar b})$
for any geometric point $\bar b$ of $B$.
Then $\Phi_{\mathcal K}$ is a $B$-linear
autoequivalence of $D^b(\XC)$.

For any noetherian $B$-scheme $T$,
the relative Fourier--Mukai transform
$\Phi_{\mathcal K_T}$
is an autoequivalence of $D^b(\XC_T)$,
and there are natural isomorphisms
\[
q^*\circ\Phi_{\mathcal K}
\simeq
\Phi_{\mathcal K_T}\circ q^*,
\quad
q^*\circ\Phi_{\mathcal K}^{-1}
\simeq
\Phi_{\mathcal K_T}^{-1}\circ q^*
\]
of functors $\Perf(\XC)\to\Perf(\XC_T)$.
\end{prop}

We apply \cref{proposition: relative FM equivalence}
to spherical twists.
Let
$\Delta\colon\XC\hookrightarrow\XC\times_B\XC$
be the relative diagonal.

\begin{prop}\label{proposition: relative spherical twist}
Let $E\in\Perf(\XC)$ be an object such that
$E_{\bar b}$ is spherical
for every geometric point $\bar b$ of $B$.
Put
\[
\mathcal K_E\coloneqq
\Cone\left(
p_1^*E^\vee\otimes^L p_2^*E
\xrightarrow{\mathrm{ev}}
\Delta_*\OC_\XC
\right).
\]
Then $\mathcal K_E$ is perfect, and
\[
T_E\coloneqq\Phi_{\mathcal K_E}\colon
D^b(\XC)\xrightarrow{\sim}D^b(\XC)
\]
is a $B$-linear autoequivalence.
For any geometric point $\bar b$ of $B$,
we have
\[
\Phi_{(\mathcal K_E)_{\bar b}}
\simeq T_{E_{\bar b}}.
\]
Moreover, $T_E$ and $T_E^{-1}$
commute with base change as in
\cref{proposition: relative FM equivalence}.
\end{prop}

\begin{proof}
Since $f$ is smooth, the relative diagonal
$\Delta$ is a regular immersion.
Hence $\mathcal K_E$ is perfect.
By derived base change, the derived dual,
evaluation morphism, and diagonal kernel
are compatible with restriction to fibers.
Hence $(\mathcal K_E)_{\bar b}$ is the kernel
of the spherical twist $T_{E_{\bar b}}$,
which is an autoequivalence by
\cite[Theorem~1.2]
{seidel_thomas_2001_braid_group_actions_on_derived_categories_of_coherent_sheaves}.
The assertions follow from
\cref{proposition: relative FM equivalence}.
\end{proof}

For the following proposition, we additionally impose
the numerical assumption of \cref{subsection:Base change}:
there exist $A_1,\ldots,A_r\in\Perf(\XC)$
such that, for every geometric point $\bar b$ of $B$,
the classes
$[(A_1)_{\bar b}],\ldots,[(A_r)_{\bar b}]$
form a basis of $\Knum(\XC_{\bar b})_\QB$.
Put $\Lambda\coloneqq\Knum(\XC/B)$.

\begin{prop}\label{proposition: relative FM action}
Let $\mathcal K\in\Perf(\XC\times_B\XC)$
satisfy the assumptions of
\cref{proposition: relative FM equivalence}.
Then there exists a unique automorphism
$\alpha_{\mathcal K}\in\GL(\Lambda)$
such that
\[
[\Phi_{\mathcal K_t}(F)]_t
=
\alpha_{\mathcal K}([F]_t)
\]
for a $B$-scheme $T$,  $t\in T$,
and $F\in D^b(\XC_t)$.

The relative Fourier--Mukai transform
induces a homeomorphism
\[
\Phi_{\mathcal K}\colon
\Stab(\XC/B)\xrightarrow{\sim}\Stab(\XC/B)
\]
given by
\[
\Phi_{\mathcal K}(\sigma_\bullet)
=
\left(
Z\circ\alpha_{\mathcal K}^{-1},
(\Phi_{\mathcal K_b}\PC_b)_{b\in B}
\right)
\]
for $\sigma_\bullet=(Z,\PC_\bullet)$.
For any morphism $T\to B$
essentially of finite type, we have
\[
(\Phi_{\mathcal K}\sigma_\bullet)|_T
=
\Phi_{\mathcal K_T}(\sigma_\bullet|_T).
\]
\end{prop}

\begin{proof}
Let $\mathcal K'\in\Perf(\XC\times_B\XC)$
be a Fourier--Mukai kernel of
$\Phi_{\mathcal K}^{-1}$.
For any noetherian $B$-scheme $T$,
both $\Phi_{\mathcal K_T}$ and
$\Phi_{\mathcal K'_T}$ preserve
$T$-perfect complexes.
Indeed, if $F$ is $T$-perfect, then
$p_1^*F\otimes^L\mathcal K_T$ is $T$-perfect.
By
\cite[Lemma~8.11]
{preprint_chunyi_zhiyu_ziqi_emanuele_alexander_paolo_xiaolei_2026_stability_conditions_and_moduli_spaces_on_projective_families},
its derived pushforward under $p_2$
is also $T$-perfect.
The same argument applies to $\mathcal K'_T$.

Both $\Phi_{\mathcal K}$ and $\Phi_{\mathcal K'}$
preserve relatively perfect complexes and,
by \cref{proposition: relative FM equivalence},
commute with base change.
Therefore they preserve the relations defining
$\Lambda$ in
\cite[Definition~21.1]
{bayer_lahoz_macri_perry_stellari_2021_stability_conditions_in_families}.
Hence $\Phi_{\mathcal K}$ induces an
automorphism $\alpha_{\mathcal K}$ of $\Lambda$
such that
$[\Phi_{\mathcal K_t}(F)]_t
=
\alpha_{\mathcal K}([F]_t)$
for any $B$-scheme $T$, $t\in T$,
and $F\in D^b(\XC_t)$.

Fix $\sigma_\bullet=(Z,\PC_\bullet)\in\Stab(\XC/B)$
and put
$\tau_\bullet\coloneqq
\left(
Z\circ\alpha_{\mathcal K}^{-1},
(\Phi_{\mathcal K_b}\PC_b)_{b\in B}
\right)$.
Let $Q$ be a quadratic form giving the 
support property for $\sigma_\bullet$.
Then $Q\circ\alpha_{\mathcal K}^{-1}$ gives
the support property for $\tau_\bullet$.

Let $T$ be an affine $B$-scheme of finite type,
and let $F$ be a $T$-perfect complex on $\XC_T$.
By base change, $F_t$ is geometrically
$\tau_t$-stable if and only if
$(\Phi_{\mathcal K'_T}(F))_t$
is geometrically $\sigma_t$-stable.
Since $\Phi_{\mathcal K'_T}(F)$ is
$T$-perfect, condition~(2) for
$\sigma_\bullet$ shows that
$\{t\in T\mid
F_t\text{ is geometrically $\tau_t$-stable}\}$
is open in $T$.
By
\cite[Lemma~20.3]
{bayer_lahoz_macri_perry_stellari_2021_stability_conditions_in_families},
this proves condition~(2) for $\tau_\bullet$.

Let $T\to B$ be a morphism essentially
of finite type from a Dedekind scheme.
Let $(Z,\PC_T)$ be a Harder--Narasimhan
structure on $D^b(\XC_T)$ over $T$
inducing $\sigma_\bullet|_T$.
Since $\Phi_{\mathcal K_T}$ is $T$-linear
and commutes with base change,
$(Z\circ\alpha_{\mathcal K}^{-1},
\Phi_{\mathcal K_T}\PC_T)$
is a Harder--Narasimhan structure
inducing $\tau_\bullet|_T$.
Thus condition~(3) holds.

For $w\in\Lambda$ and $\phi\in\RB$,
$\Phi_{\mathcal K}$ induces an isomorphism
$\mathcal M^{\mathrm{st}}_{\sigma_\bullet}
(\alpha_{\mathcal K}^{-1}(w),\phi)
\xrightarrow{\sim}
\mathcal M^{\mathrm{st}}_{\tau_\bullet}(w,\phi)$.
Thus condition~(4) for $\tau_\bullet$
follows from that for $\sigma_\bullet$.

Since $\Phi_{\mathcal K'}$ induces the inverse
map and both maps are continuous by the
definition of the topology on $\Stab(\XC/B)$,
$\Phi_{\mathcal K}$ is a homeomorphism.
\end{proof}

\section{Criteria for disconnectedness}\label{section: disconnectedness in families}
\subsection{Nordskova's argument}\label{subsection:Spherical twists and connected components}
Let $k$ be an algebraically closed field of characteristic zero
and $X$ a smooth projective variety over $k$ with $\dim X \geq 2$.
We extend Nordskova's argument
\cite[Lemma~6.1]{preprint_anya_2026_a_smooth_projective_counterexample_to_bondalpolishchuks_conjecture}
to $k$-semilinear autoequivalences.

\begin{lem}\label{lemma: fixed component}
Let $g$ be a $k$-semilinear autoequivalence
of $D^b(X)$ acting trivially on $\Knum(X)$.
If $g(\sigma_0)=\sigma_0$ for some $\sigma_0\in\Stab(X)$,
then $g$ acts as the identity on the connected component
of $\Stab(X)$ containing $\sigma_0$.
\end{lem}
\begin{proof}
The proof of
\cite[Lemma~6.1(1)]{preprint_anya_2026_a_smooth_projective_counterexample_to_bondalpolishchuks_conjecture}
also applies to $k$-semilinear autoequivalences.
It uses only the triviality of the numerical action and the local injectivity of the map in
\cref{theorem: Bridgeland's deformation theorem}.
\end{proof}

For a spherical object $E\in D^b(X)$, let $T_E$ denote its spherical twist. By \cite[(1.2)]{seidel_thomas_2001_braid_group_actions_on_derived_categories_of_coherent_sheaves},
for any $v\in\Knum(X)$, we have
\begin{equation}\label{equation: relative twist numerical action}
[T_E](v)=v-\chi([E],v)[E].
\end{equation}
Two objects $E_1,E_2\in D^b(X)$
are called \emph{orthogonal} if they satisfy
\[\RHom(E_1,E_2)=\RHom(E_2,E_1)=0.\]

\begin{prop}\label{proposition: phase separation}
Let $E_1,\ldots,E_d\in D^b(X)$ be mutually
orthogonal spherical objects, where $d\geq 2$.
For $\mathbf a=(a_1,\ldots,a_d)\in\ZZ^d$, put
$\Phi_{\mathbf a}
\coloneqq T_{E_1}^{a_1}\cdots T_{E_d}^{a_d}$.
Let $\Sigma\subset\Stab(X)$ be a connected component
such that
$\phi_\sigma^+(E_1)=\cdots=\phi_\sigma^+(E_d)$
for any $\sigma\in\Sigma$.
For $\mathbf a,\mathbf a'\in\ZZ^d$,
if $\Phi_{\mathbf a}\Sigma=\Phi_{\mathbf a'}\Sigma$,
then
$\mathbf a-\mathbf a'\in\ZZ(1,\ldots,1)$.
In particular, the connected components
$\Sigma$ and $T_{E_i}^{-1}\Sigma$
for $1\leq i\leq d$ are pairwise distinct.
Moreover, if $[E_1]=\cdots=[E_d]$ in $\Knum(X)$,
then $\Phi_{\mathbf a}$ acts trivially on $\Knum(X)$
whenever $\sum_{i=1}^d a_i=0$.
For such $\mathbf a$ and
$\sigma=(Z,\PC)\in\Sigma$, we have
$\Phi_{\mathbf a}\sigma
=(Z,\Phi_{\mathbf a}\PC)
\in\Phi_{\mathbf a}\Sigma$.
\end{prop}

\begin{proof}
Fix $\sigma\in\Sigma$ and put
$\phi\coloneqq\phi_\sigma^+(E_1)
=\cdots=\phi_\sigma^+(E_d)$.
Suppose that
$\Phi_{\mathbf a}\Sigma=\Phi_{\mathbf a'}\Sigma$
for some $\mathbf a,\mathbf a'\in\ZZ^d$.
Since $T_{E_i}(E_i)\simeq E_i[1-\dim X]$
and $T_{E_j}(E_i)\simeq E_i$ for $i\ne j$,
the stability condition
$\tau\coloneqq\Phi_{\mathbf a'}^{-1}
\Phi_{\mathbf a}\sigma\in\Sigma$ satisfies
\[
\phi_\tau^+(E_i)
=\phi+(a_i-a_i')(\dim X-1)
\]
for $1\leq i\leq d$.
By the assumption on $\Sigma$, this implies
\[
(a_i-a_i')(\dim X-1)
=(a_j-a_j')(\dim X-1)
\]
for all $i,j$.
Since $\dim X\geq 2$, we obtain
$\mathbf a-\mathbf a'\in\ZZ(1,\ldots,1)$.
In particular, the connected components
$\Sigma$ and $T_{E_i}^{-1}\Sigma$
for $1\leq i\leq d$ are pairwise distinct.
Finally, suppose that
$[E_1]=\cdots=[E_d]$ in $\Knum(X)$.
By \labelcref{equation: relative twist numerical action},
we have
$[T_{E_1}]=\cdots=[T_{E_d}]$.
Hence $[\Phi_{\mathbf a}]=\mathrm{id}$
on $\Knum(X)$ whenever
$\sum_{i=1}^d a_i=0$.
Therefore, for such $\mathbf a$
and $\sigma=(Z,\PC)\in\Sigma$, we obtain
$\Phi_{\mathbf a}\sigma
=(Z,\Phi_{\mathbf a}\PC)
\in\Phi_{\mathbf a}\Sigma$.
\end{proof}

\begin{cor}\label{corollary: disconnectedness from a semilinear equivalence}
Let $E_1,\ldots,E_d\in D^b(X)$ be mutually
orthogonal spherical objects, where $d\geq 2$.
Let $g_2,\ldots,g_d$ be $k$-semilinear
autoequivalences of $D^b(X)$ acting trivially
on $\Knum(X)$ such that
$g_i(E_1)\simeq E_i$ for $2\leq i\leq d$.
Suppose that there exists
$\sigma_0\in\Stab(X)$ such that
$g_i(\sigma_0)=\sigma_0$ for $2\leq i\leq d$.
Let $\Sigma$ be the connected component
containing $\sigma_0$.
For $\mathbf a=(a_1,\ldots,a_d)\in\ZZ^d$, put
$\Phi_{\mathbf a}
\coloneqq T_{E_1}^{a_1}\cdots T_{E_d}^{a_d}$.
For $\mathbf a,\mathbf a'\in\ZZ^d$,
if $\Phi_{\mathbf a}\Sigma=\Phi_{\mathbf a'}\Sigma$,
then
$\mathbf a-\mathbf a'\in\ZZ(1,\ldots,1)$.
In particular, the connected components
$\Sigma$ and $T_{E_i}^{-1}\Sigma$
for $1\leq i\leq d$ are pairwise distinct.
Moreover, $\Phi_{\mathbf a}$ acts trivially
on $\Knum(X)$ whenever $\sum_{i=1}^d a_i=0$.
For such $\mathbf a$ and
$\sigma=(Z,\PC)\in\Sigma$, we have
$\Phi_{\mathbf a}\sigma
=(Z,\Phi_{\mathbf a}\PC)
\in\Phi_{\mathbf a}\Sigma$.
\end{cor}

\begin{proof}
By \cref{lemma: fixed component},
we have $g_i(\sigma)=\sigma$
for $\sigma\in\Sigma$ and $2\leq i\leq d$.
Since $g_i(E_1)\simeq E_i$, it follows that
$\phi_\sigma^+(E_1)=\cdots=\phi_\sigma^+(E_d)$
for $\sigma\in\Sigma$.
The triviality of the actions of $g_i$ on $\Knum(X)$
implies $[E_1]=\cdots=[E_d]$.
The assertions follow from
\cref{proposition: phase separation}.
\end{proof}

\begin{rem}\label{remark: parity obstruction}
If orthogonal spherical objects $E_1,E_2$ satisfy $[E_1]=[E_2]$ in $\Knum(X)$, then
\[
0=\chi(E_1,E_2)=\chi(E_1,E_1)=1+(-1)^{\dim X}.\]
Thus $\dim X$ is odd.
In particular, this holds under the assumptions of
\cref{corollary: disconnectedness from a semilinear equivalence}.
\end{rem}

\subsection{\'Etale covers and Galois actions}\label{subsection: disconnectedness in families}

We apply
\cref{corollary: disconnectedness from a semilinear equivalence}
to the Galois action on the geometric generic fiber
and use specialization to prove disconnectedness
for closed fibers.

\begin{setup}\label{setup: disconnectedness in families}
Let $f\colon\XC\to B$ be a smooth projective morphism
of relative dimension $n\geq 2$
with geometrically connected fibers,
where $B$ is a smooth quasi-projective variety over $\CB$.
Assume that there exist
$A_1,\ldots,A_r\in\Perf(\XC)$
such that, for every geometric point $\bar b$ of $B$,
the classes
$[(A_1)_{\bar b}],\ldots,[(A_r)_{\bar b}]$
form a basis of $\Knum(\XC_{\bar b})_\QB$.

Let $\Lambda\coloneqq\Knum(\XC/B)$,
and fix a stability condition
$\sigma_\bullet=(\sigma_b)_{b\in B}\in\Stab(\XC/B)$.
For $b\in B(\CB)$, let $\Sigma_b$ be
the connected component of $\Stab(\XC_b)$
containing $\sigma_b$,
where we regard $\sigma_b$ as an element of
$\Stab(\XC_b)$ via the identification \labelcref{equation: identification of stability spaces}.

Let $\pi\colon T\to B$ be a connected finite étale cover
of degree $d\geq 2$, and let $E\in\Perf(\XC_T)$.
Assume that, for any geometric point $\bar b$ of $B$,
the objects $E_t$ for $t\in T_{\bar b}$
are mutually orthogonal $n$-spherical objects
in $D^b(\XC_{\bar b})$.

For $b\in B(\CB)$, choose an ordering
$T_b=\{t_1,\ldots,t_d\}$.
For $\mathbf a=(a_1,\ldots,a_d)\in\ZZ^d$, put
$\Phi_{\mathbf a,b}
\coloneqq T_{E_{t_1}}^{a_1}\cdots T_{E_{t_d}}^{a_d}$.
\end{setup}

\begin{thm}\label{theorem: disconnectedness in families}
In \cref{setup: disconnectedness in families},
fix $b\in B(\CB)$.
For $\mathbf a,\mathbf a'\in\ZZ^d$,
if $\Phi_{\mathbf a,b}\Sigma_b
=\Phi_{\mathbf a',b}\Sigma_b$,
then
$\mathbf a-\mathbf a'\in\ZZ(1,\ldots,1)$.
In particular, the connected components
$\Sigma_b$ and $T_{E_t}^{-1}\Sigma_b$
for $t\in T_b$ are pairwise distinct.
Moreover, $\Phi_{\mathbf a,b}$ acts trivially
on $\Knum(\XC_b)$ whenever
$\sum_{i=1}^d a_i=0$.
The connected components
$\Phi_{\mathbf a,b}\Sigma_b$
for such $\mathbf a$ are pairwise distinct.
In particular, every central charge occurring in
$\Sigma_b$ has lifts in infinitely many distinct
connected components of $\Stab(X_b)$.
\end{thm}

We first recall a basic fact about absolute Galois groups.
Let $K$ be a field of characteristic zero, and fix an algebraic
closure $\bar K$.
The absolute Galois group $\Gal(\bar K/K)$ is the group of
field automorphisms of $\bar K$ fixing $K$ pointwise.

\begin{lem}\label{lemma: transitive Galois action}
Let $L/K$ be a finite extension.
Then the action of $\Gal(\bar K/K)$ on $(\Spec L)(\bar K)$ is transitive.
\end{lem}

\begin{proof}
The $\bar K$-valued points of $\Spec L$ correspond to
$K$-embeddings $L\hookrightarrow\bar K$.
For any two such embeddings $\iota_1$ and $\iota_2$,
the uniqueness of algebraic closures up to isomorphism
gives $g\in\Gal(\bar K/K)$ such that
$g\circ\iota_1=\iota_2$.
Therefore, $\Gal(\bar K/K)$ acts transitively on $(\Spec L)(\bar K)$.
\end{proof}

We use \cref{lemma: transitive Galois action} to show that
certain stability conditions on the geometric generic fiber
lie in distinct connected components.

\begin{lem}\label{lemma: separation on geometric generic fiber}
In \cref{setup: disconnectedness in families},
let $K\coloneqq\CB(B)$ and fix an algebraic closure $\bar K$ of $K$.
Let $\bar\eta\colon\Spec\bar K\to B$ be the corresponding
geometric generic point, and let $\Sigma_{\bar\eta}$ be
the connected component of $\Stab(\XC_{\bar\eta})$
containing $\sigma_{\bar\eta}$.
Write $T_{\bar\eta}=\{u_1,\ldots,u_d\}$.
For $\mathbf a=(a_1,\ldots,a_d)\in\ZZ^d$, put
$\Phi_{\mathbf a,\bar\eta}
\coloneqq T_{E_{u_1}}^{a_1}\cdots T_{E_{u_d}}^{a_d}$.
For $\mathbf a,\mathbf a'\in\ZZ^d$,
if $\Phi_{\mathbf a,\bar\eta}\Sigma_{\bar\eta}
=\Phi_{\mathbf a',\bar\eta}\Sigma_{\bar\eta}$,
then $\mathbf a-\mathbf a'\in\ZZ(1,\ldots,1)$.
\end{lem}
\begin{proof}
Let $\eta$ be the generic point of $B$.
Since $T$ is connected and finite \'etale over the smooth variety $B$, it is integral.
Put $L\coloneqq\CB(T)$.
Then $T_\eta=\Spec L$, and $L/K$ is a finite extension.
Under the natural identification $T_\eta(\bar K)\simeq T_{\bar\eta}(\bar K)$, \cref{lemma: transitive Galois action} shows that $\Gal(\bar K/K)$ acts transitively on $T_{\bar\eta}(\bar K)$.
Since $E$ is defined over $\XC_T$, the pullbacks $g^*$ for $g\in\Gal(\bar K/K)$ permute $E_{u_1},\ldots,E_{u_d}$ transitively up to isomorphism.

By \cref{lem:Galois invariance}, the pullbacks $g^*$ act trivially on $\Knum(\XC_{\bar\eta})$ and fix $\sigma_{\bar\eta}$.
In particular, they preserve $\Sigma_{\bar\eta}$.
By \cref{corollary: disconnectedness from a semilinear equivalence},
the equality $\Phi_{\mathbf a,\bar\eta}\Sigma_{\bar\eta}
=\Phi_{\mathbf a',\bar\eta}\Sigma_{\bar\eta}$
implies $\mathbf a-\mathbf a'\in\ZZ(1,\ldots,1)$.
\end{proof}

We establish a specialization result
for connected components of the spaces of stability conditions.

\begin{lem}\label{lemma: specialization of connected components}
Let $C$ be a smooth curve over $\CB$,
let $c\in C(\CB)$, and put
$S\coloneqq\Spec\mathcal O_{C,c}$.
Let $f\colon\XC\to S$ be a smooth projective morphism
with geometrically connected fibers.
Assume that there exist $A_1,\ldots,A_r\in\Perf(\XC)$
as in \cref{subsection:Base change}.
Write $0$ and $\eta$ for the closed and generic
points of $S$, respectively, and fix a geometric
generic point
$\bar\eta\colon\Spec\overline{\kappa(\eta)}\to S$.
Under the identification
\eqref{equation: identification of stability spaces},
the map
\[
\res_0 \colon \Stab(\XC/S)\to\Stab(\XC_0),
\quad
\sigma_\bullet\mapsto\sigma_0
\]
is a homeomorphism onto a union of connected
components of $\Stab(\XC_0)$.
In particular, for any
$\sigma_\bullet,\tau_\bullet\in\Stab(\XC/S)$,
if $\sigma_{\bar\eta}$ and $\tau_{\bar\eta}$
lie in distinct connected components of
$\Stab(\XC_{\bar\eta})$,
then $\sigma_0$ and $\tau_0$ lie in distinct
connected components of $\Stab(\XC_0)$.
\end{lem}

\begin{proof}
Fix $\sigma_\bullet=(\sigma_s)_{s \in S} \in\Stab(\XC/S)$, where $\sigma_s=(Z,\PC_s)$ for $s \in S$.
Let $Q$ be a quadratic form giving the support
property for $\sigma_0$.
We first show that the same $Q$ gives
the support property for $\sigma_\bullet$.
Let $F\in D^b(\XC_\eta)$ be
$\sigma_\eta$-semistable.
By \cref{proposition: semistable reduction},
there exists $\mathcal F\in\Perf(\XC)$
such that $\mathcal F_\eta\simeq F$
and $\mathcal F_0$ is $\sigma_0$-semistable.
Since $S$ is connected, by
\cite[Remark~8.8]{preprint_chunyi_zhiyu_ziqi_emanuele_alexander_paolo_xiaolei_2026_stability_conditions_and_moduli_spaces_on_projective_families},
we have
$[F]_\eta=[\mathcal F_0]_0$ in $\Knum(\XC/S)$.
By the support property for $\sigma_0$, we obtain
\[
Q([F]_\eta)=Q([\mathcal F_0]_0)\geq 0.
\]
Thus $Q$ gives the support property
for $\sigma_\bullet$.

We next prove injectivity.
Take $\tau_\bullet\in\Stab(\XC/S)$
with $\tau_0=\sigma_0$.
Then we may write $\tau_s=(Z,\QC_s)$ for $s \in S$.
For $\phi\in\RB$, let
$F\in\PC_\eta(\phi)$.
By \cref{proposition: semistable reduction},
there exists $\mathcal F\in\Perf(\XC)$
such that $\mathcal F_\eta\simeq F$
and $\mathcal F_0$ is
$\sigma_0$-semistable of phase $\phi$.
Since $\tau_0=\sigma_0$, we have
\[
\mathcal F_0\in
\QC_0((\phi-\varepsilon,\phi+\varepsilon))
\]
for any $\varepsilon>0$.
By \cref{proposition: openness_PI},
the subset
$\{s\in S\mid
\mathcal F_s\in
\QC_s((\phi-\varepsilon,\phi+\varepsilon))\}$
is open in $S$.
This open subset contains the closed point $0$
and hence equals $S$.
Thus $F\in\QC_\eta((\phi-\varepsilon,\phi+\varepsilon))$
for every $\varepsilon>0$,
which implies $F\in\QC_\eta(\phi)$.
Interchanging $\sigma_\bullet$ and
$\tau_\bullet$, we obtain
$\PC_\eta=\QC_\eta$.
Therefore $\sigma_\eta=\tau_\eta$
and $\sigma_\bullet=\tau_\bullet$.

By
\cref{theorem: Bridgeland's deformation theorem,thm:relative_deformation},
the map
$\res_0$
is a local homeomorphism, and hence a homeomorphism onto its image
by injectivity.
In particular, $\res_0(\Stab(\XC/S))$ is open.
We show that $\res_0(\Stab(\XC/S))$ is also closed,
and therefore a union of connected components
of $\Stab(\XC_0)$.
Let
\[
\ZC_0\colon\Stab(\XC_0)
\to\Hom(\Knum(\XC/S),\CB),
\quad
(W,\PC)\mapsto W.
\]
Let $V_{\sigma_0}$ be the connected component of
$\ZC_0^{-1}(U_Q(Z))$ containing $\sigma_0$.
We claim that
$V_{\sigma_0}\subset \res_0(\Stab(\XC/S))$.
By
\cite[Proposition~11.5]{bayer_macri_stellari_2016_the_space_of_stability_conditions_on_abelian_threefolds_and_on_some_calabiyau_threefolds},
the restriction
$\ZC_0|_{V_{\sigma_0}}\colon V_{\sigma_0}\to U_Q(Z)$
is a covering map, and every stability condition in $V_{\sigma_0}$
satisfies the support property with respect to $Q$.
By \cref{thm:relative_deformation},
there exists a connected open neighborhood
$V$ of $\sigma_\bullet$ in $\Stab(\XC/S)$
such that
$\ZC|_V\colon V\to U_Q(Z)$
is a covering map.
Since $\res_0(V)$ is a connected subset of
$\ZC_0^{-1}(U_Q(Z))$ containing $\sigma_0$,
we have $\res_0(V)\subset V_{\sigma_0}$.
Take $\tau_0\in V_{\sigma_0}$ and choose a path
$\gamma\colon[0,1]\to V_{\sigma_0}$
with $\gamma(0)=\sigma_0$ and
$\gamma(1)=\tau_0$.
The covering map $\ZC|_V$ gives a lift
$\widetilde\gamma\colon[0,1]\to V$
of $\ZC_0\circ\gamma$
with $\widetilde\gamma(0)=\sigma_\bullet$.
Since $\ZC_0\circ\res_0=\ZC$,
both $\gamma$ and
$\res_0\circ\widetilde\gamma$
are lifts of $\ZC_0\circ\gamma$
with respect to $\ZC_0|_{V_{\sigma_0}}$,
starting at $\sigma_0$.
By uniqueness of path lifting, we obtain
$\res_0\circ\widetilde\gamma=\gamma$.
In particular,
$\tau_0=\res_0(\widetilde\gamma(1))$.
Since $\tau_0\in V_{\sigma_0}$ is arbitrary,
we obtain $V_{\sigma_0}\subset \res_0(\Stab(\XC/S))$.

Take
$\tau_0\in\Stab(\XC_0)\setminus \res_0(\Stab(\XC/S))$.
Put $W\coloneqq\ZC_0(\tau_0)$,
and let $Q'$ be a quadratic form
as in the support property for $\tau_0$.
Let $V_{\tau_0}$ be the connected component of
$\ZC_0^{-1}(U_{Q'}(W))$ containing $\tau_0$.
By
\cite[Proposition~11.5]{bayer_macri_stellari_2016_the_space_of_stability_conditions_on_abelian_threefolds_and_on_some_calabiyau_threefolds},
every stability condition in $V_{\tau_0}$
satisfies the support property with respect to $Q'$.
Suppose that
$V_{\tau_0}\cap\res_0(\Stab(\XC/S)) \ne\varnothing$.
Choose $\sigma'_\bullet\in\Stab(\XC/S)$
with $\sigma'_0\in V_{\tau_0}$.
Then $\sigma'_0$ satisfies the support property
with respect to $Q'$.
Applying the preceding argument to
$\sigma'_\bullet$ and $Q'$, we obtain
$V_{\tau_0}\subset \res_0(\Stab(\XC/S))$,
contrary to the choice of $\tau_0$.
Therefore, we have
$V_{\tau_0}\cap\res_0(\Stab(\XC/S))=\varnothing$.
In particular, it follows that
\[\tau_0\in V_{\tau_0}
\subset
\Stab(\XC_0)\setminus\res_0(\Stab(\XC/S)).\]
Since $V_{\tau_0}$ is open,
$\res_0(\Stab(\XC/S))$ is closed.
Thus $\res_0(\Stab(\XC/S))$ is open and closed,
and hence a union of connected components of
$\Stab(\XC_0)$.
The final assertion now follows by restricting to the geometric generic fiber.
\end{proof}

\begin{proof}[Proof of \cref{theorem: disconnectedness in families}]
Fix $b\in B(\CB)$ and write $T_b=\{t_1,\ldots,t_d\}$.
Since $T\to B$ is a connected finite \'etale cover of degree $d\geq2$, we have $\dim B>0$.
If $\dim B=1$, put $C=B$ and $c=b$.
Otherwise, repeated application of
\cite[Lemma~6.4]{rulling_schroer_2025_loops_on_schemes_and_the_algebraic_fundamental_group}
gives an integral curve through $b$ whose pullback to $T$ is irreducible.
Taking its normalization, we obtain a smooth connected curve $C$, a morphism $C\to B$, and a point $c\in C(\CB)$ over $b$ such that $T_C$ is connected.

By
\cite[Proposition~1.1]{rulling_schroer_2025_loops_on_schemes_and_the_algebraic_fundamental_group},
choose a connected finite \'etale cover $C'\to C$ and a point $c'\in C'(\CB)$ over $c$ such that $T_{C'}\simeq\coprod_{i=1}^d C'$.
Put $S\coloneqq\Spec\OC_{C',c'}$, and write $0$ for its closed point.
Put $K\coloneqq\CB(C)$, choose an algebraic closure $\bar K$ of $K$, and fix a $K$-embedding $\CB(C')\hookrightarrow\bar K$.
Let $\bar\eta$ be the resulting geometric generic point of $S$, also regarded as a geometric generic point of $C$.

For $1\leq i\leq d$, let $s_i\colon S\to T_S$ be the sections corresponding to the splitting $T_S\simeq\coprod_{i=1}^d S$.
After renumbering, we may assume that $s_i(0)=t_i$ under the identification $T_0=T_b$.
Let $u_i\in T_{\bar\eta}(\bar K)$ be the restriction of $s_i$ to $\bar\eta$.
Let $E_i\in\Perf(\XC_S)$ be the pullback of $E$ along the morphism $\XC_S\to\XC_T$ induced by $S\xrightarrow{s_i}T_S\to T$.
Then $(E_i)_0\simeq E_{t_i}$ and $(E_i)_{\bar\eta}\simeq E_{u_i}$.

By \cref{lemma: numerical identification}, the canonical homomorphisms
$[-]_C\colon\Knum(\XC_C/C)\to\Lambda$ and
$[-]_S\colon\Knum(\XC_S/S)\to\Lambda$
are injective with finite cokernel.
Composing the central charges with these homomorphisms, we regard $\sigma_\bullet|_C$ and $\sigma_\bullet|_S$ as elements of $\Stab(\XC_C/C)$ and $\Stab(\XC_S/S)$, respectively.

For $\mathbf a=(a_1,\ldots,a_d)\in\ZZ^d$, put
\[
\Phi_{\mathbf a,S}\coloneqq T_{E_1}^{a_1}\cdots T_{E_d}^{a_d},
\quad
\Phi_{\mathbf a,\bar\eta}\coloneqq T_{E_{u_1}}^{a_1}\cdots T_{E_{u_d}}^{a_d}.
\]
By \cref{proposition: relative spherical twist,proposition: relative FM action},
the stability condition $\Phi_{\mathbf a,S}(\sigma_\bullet|_S)\in\Stab(\XC_S/S)$
restricts to $\Phi_{\mathbf a,b}\sigma_b$ at $0$ and to
$\Phi_{\mathbf a,\bar\eta}\sigma_{\bar\eta}$ at $\bar\eta$.

Suppose that $\mathbf a-\mathbf a'\notin\ZZ(1,\ldots,1)$.
Since $T_C$ is connected, \cref{lemma: separation on geometric generic fiber}, applied to the family $\XC_C\to C$ and the cover $T_C\to C$, shows that
$\Phi_{\mathbf a,\bar\eta}\sigma_{\bar\eta}$ and
$\Phi_{\mathbf a',\bar\eta}\sigma_{\bar\eta}$
lie in distinct connected components of $\Stab(\XC_{\bar\eta})$.
By \cref{lemma: specialization of connected components}, applied to
$\Phi_{\mathbf a,S}(\sigma_\bullet|_S)$ and
$\Phi_{\mathbf a',S}(\sigma_\bullet|_S)$,
the stability conditions $\Phi_{\mathbf a,b}\sigma_b$ and
$\Phi_{\mathbf a',b}\sigma_b$ lie in distinct connected components of $\Stab(\XC_b)$.
Hence $\Phi_{\mathbf a,b}\Sigma_b\neq\Phi_{\mathbf a',b}\Sigma_b$.
This proves the first assertion.
In particular, $\Sigma_b$ and $T_{E_t}^{-1}\Sigma_b$ for $t\in T_b$ are pairwise distinct.

Finally, since $T$ is connected, the relative numerical class $[E_t]_t\in\Lambda$ is independent of $t\in T$.
By \cref{lemma: numerical identification}, we obtain $[E_{t_1}]=\cdots=[E_{t_d}]$ in $\Knum(\XC_b)$.
Hence \eqref{equation: relative twist numerical action} gives
$[\Phi_{\mathbf a,b}]=\id$ whenever $\sum_i a_i=0$.
If $\sum_i a_i=\sum_i a_i'=0$, the first assertion shows that
$\Phi_{\mathbf a,b}\Sigma_b=\Phi_{\mathbf a',b}\Sigma_b$
implies $\mathbf a=\mathbf a'$.
Thus the connected components $\Phi_{\mathbf a,b}\Sigma_b$ for such $\mathbf a$ are pairwise distinct.
For such $\mathbf a$ and $\sigma=(Z,\PC)\in\Sigma_b$, we have
$\Phi_{\mathbf a,b}\sigma=(Z,\Phi_{\mathbf a,b}\PC)$.
Since $d\geq2$, there are infinitely many such vectors.
Therefore, every central charge occurring in $\Sigma_b$ has lifts in infinitely many distinct connected components of $\Stab(\XC_b)$.
\end{proof}

\section{A local criterion for geometricity}\label{section: A local criterion for geometricity}
In this section, we establish a criterion for the existence of geometric stability conditions in the connected components constructed in \cref{section: disconnectedness in families}.
We first show that, for objects in \(\PC_\sigma(\phi)\) with bounded \(\sigma\)-mass, Harder–Narasimhan filtrations and their Jordan–H\"older filtrations remain in \(\PC_\sigma(\phi)\) under small deformations of \(\sigma\).

\begin{lem}\label{lemma: uniform localization}
Let $\sigma=(Z_\sigma,\PC_\sigma)\in\Stab(X)$, and fix $\phi\in\RB$ and $M>0$.
Then there exists $0<\delta<1/4$ such that 
for $\tau\in\Stab(X)$ with $d(\sigma,\tau)<\delta$ and nonzero object $E\in\PC_{\sigma}(\phi)$ with $m_\sigma(E)\leq M$, 
\begin{itemize}
\item the $\tau$-Harder--Narasimhan filtration of $E$ and 
\item $\tau$-Jordan--H\"older filtration of each $\tau$-Harder--Narasimhan factor of $E$
\end{itemize}
consist of short exact sequences in the abelian category $\PC_\sigma(\phi)$.
In particular, all factors in these filtrations are $\sigma$-semistable of phase $\phi$.
\end{lem}
\begin{proof}
We may assume that $\phi=1$.
By the support property, there is a constant $C>0$ such that $\norm{[G]}\leq C\abs{Z_\sigma(G)}$ for every $\sigma$-semistable object $G$.
Since $\Knum(X)$ is a lattice, the set $S\coloneqq \{v\in\Knum(X)\mid\norm{v}\leq2CM\}$ is finite.
Therefore, there exists a constant $\eta>0$ such that $\Im Z_\sigma(v)=0$ or $\abs{\Im Z_\sigma(v)}\geq\eta$ for $v\in S$.
We take $0<\delta<1/4$ such that $Me^{2\delta}\sin(2\pi\delta)<\eta$.

Let $\tau\in\Stab(X)$ and $E\in\PC_\sigma(1)$ satisfy the conditions in the statement.
Let $F$ be a $\tau$-Harder--Narasimhan factor of $E$ or a $\tau$-Jordan--H\"older factor of such a factor.
In either case, $F$ is $\tau$-semistable and satisfies $\phi_\tau^-(E)\leq \phi_\tau(F) \leq\phi_\tau^+(E)$.
Since $E\in\PC_\sigma(1)$ and $d(\sigma,\tau)<\delta$, we have $\abs{\phi_\tau(F)-1}<\delta$.
Applying the distance bound to the $\tau$-semistable object $F$, we obtain
\begin{equation}\label{equation: uniform localization phase}
1-2\delta<\phi_\tau^-(F) -\delta < \phi_\sigma^-(F)
\leq\phi_\sigma^+(F)< \phi_\tau^+(F) +\delta <1+2\delta.
\end{equation}

Let $G$ be a $\sigma$-Harder--Narasimhan factor of $F$. Then, $\abs{\phi_\sigma(G)-1}<2\delta$ by \labelcref{equation: uniform localization phase}.
On the other hand, as $m_\tau(F)\leq m_\tau(E)$ holds by the additivity of mass function along Harder--Narasimhan and Jordan--H\"older filtrations, we obtain 
\begin{equation}\label{equation: uniform localization mass}
\begin{aligned}
\abs{Z_\sigma(G)}
&\leq m_\sigma(F)
\leq e^\delta m_\tau(F)\\
&\leq e^\delta m_\tau(E)
\leq e^{2\delta}m_\sigma(E)
\leq e^{2\delta}M.
\end{aligned}
\end{equation}
Therefore, we have $\norm{[G]}\leq Ce^{2\delta}M<2CM$, i.e. $[G]\in S$ and 
\[
\abs{\Im Z_\sigma(G)}
=\abs{Z_\sigma(G)}\abs{\sin\bigl(\pi(\phi_\sigma(G)-1)\bigr)}
\leq Me^{2\delta}\sin(2\pi\delta)<\eta.
\]
Thus, the choice of $\eta$ gives $\Im Z_\sigma(G)=0$, which means that the nonzero $\sigma$-semistable object $G$ satisfies $\phi_\sigma(G)\in\ZZ$.
Now, \labelcref{equation: uniform localization phase} gives $\phi_\sigma(G)=1$.
Thus, every $\sigma$-Harder--Narasimhan factor of $F$ has phase one, so $F\in\PC_\sigma(1)$.
\end{proof}

Before proceeding with the criterion, we show an elementary lemma about subobjects in an abelian category.

\begin{lem}\label{lemma: wall subobjects}
Let $P$ and $S$ be nonisomorphic simple objects of an abelian category $\CC$, and
\[
0\longrightarrow P\longrightarrow B\longrightarrow S^{\oplus r}\longrightarrow0
\]
be an exact sequence with $r\geq1$ and $\Hom(S,B)=0$.
Then every nonzero proper subobject $A\subset B$ fits into an exact sequence
\[
0\longrightarrow P\longrightarrow A\longrightarrow S^{\oplus k}\longrightarrow0
\]
for some integer $0\leq k<r$.
\end{lem}
\begin{proof}
Let $A\subset B$ be a nonzero proper subobject.
First, we show $P\subset A$. As $P$ is simple, it suffices to show that $A\cap P\neq0$.
Suppose that $A\cap P=0$.
Then there is an injection $A\hookrightarrow B\twoheadrightarrow S^{\oplus r}$ from a nonzero object to a semisimple object. Therefore an injection $S\hookrightarrow A$ exists, which contradicts $\Hom(S,B)=0$.

Consider the quotient $A/P$, which is a subobject of $S^{\oplus r}$ by the given short exact sequence.
Thus, $A/P\cong S^{\oplus k}$ for some $0\leq k< r$ since $A\subsetneq B$.
\end{proof}
We now give a local criterion for geometricity after applying an autoequivalence.

\begin{setup}\label{setup: geometric wall}
Fix $\sigma_0=(Z_0,\PC_0)\in\Stabgeo(X)$ such that $\phi_{\sigma_0}(\OC_x)=1$.
Let $\Sigma$ be the connected component of $\sigma_0$, $S$ be a $\sigma_0$-stable object of phase one, and fix $\Phi\in\Aut D^b(X)$ such that $\Phi^{-1}\Sigma\ne\Sigma$.
Assume that there is an integer $r\geq1$ such that, for $x\in X$, there is an exact sequence in $\PC_0(1)$:
\begin{equation}\label{equation: extension at wall}
0\longrightarrow\OC_x\longrightarrow\Phi(\OC_x)
\longrightarrow S^{\oplus r}\longrightarrow0.
\end{equation}

For $\tau=(Z_\tau,\PC_\tau)\in\Sigma$, denote $\Delta(\tau)\coloneqq \Im\bigl(Z_\tau(S)\overline{Z_\tau(\OC_x)}\bigr)$.
Note that $\Delta(\tau)$ is independent of $x$, since the numerical class of $\OC_x$ is independent of $x$.
\end{setup}

\begin{rem}
Note that $[\OC_x]$ and $[S]$ are linearly independent in $\Knum(X)_{\RB}$.
Indeed, suppose that $[S]=c[\OC_x]$ for some $c\in\RB$.
Since $\chi(\OC_X,\OC_x)=1$, the class $[\OC_x]$ is primitive in $\Knum(X)$, and $c=\chi(\OC_X,S)\in\ZZ$.
Moreover, $c>0$ since $Z_0(S)$ and $Z_0(\OC_x)$ are negative real numbers.
By \eqref{equation: extension at wall}, we have
\[
[\Phi(\OC_x)]=(1+rc)[\OC_x].
\]
Since $\Phi$ induces an automorphism of $\Knum(X)$, the class $[\Phi(\OC_x)]$ is primitive.
This contradicts $1+rc>1$.

We also have $\Hom(S,\Phi(\OC_x))=0$ for $x\in X$.
Indeed, let $q_x\colon\Phi(\OC_x)\twoheadrightarrow S$ be the composition of the quotient map in \eqref{equation: extension at wall} with a projection $S^{\oplus r}\twoheadrightarrow S$.
For $f\colon S\to\Phi(\OC_x)$, the restriction of $fq_x$ to the subobject $\OC_x$ vanishes.
Since $\End(\Phi(\OC_x))\cong\End(\OC_x)\cong\CB$, we obtain $fq_x=0$.
As $q_x$ is an epimorphism in $\PC_0(1)$, we have $f=0$.
\end{rem}

\begin{prop}\label{proposition: extension stability}
\label{proposition: spherical twist geometricity}
In \cref{setup: geometric wall}, there is an open neighborhood $U\subset\Sigma$ of $\sigma_0$ consisting of geometric stability conditions such that, for $\tau\in U$ and $x\in X$, the following hold:
\begin{enumerate}
\item\label{item: positive extension determinant}
If $\Delta(\tau)>0$, then $\Phi(\OC_x)$ is $\tau$-stable.
\item\label{item: zero extension determinant}
If $\Delta(\tau)=0$, then $\Phi(\OC_x)$ is a strictly $\tau$-semistable object whose Jordan--H\"older factors are $\OC_x$ and $r$ copies of $S$.
\item\label{item: negative extension determinant}
If $\Delta(\tau)<0$, then $\Phi(\OC_x)$ is a $\tau$-unstable object whose Harder--Narasimhan factors are $\OC_x$ and $S^{\oplus r}$.
\end{enumerate}
In particular, for $\tau\in U$, the stability condition $\Phi^{-1}\tau$ is geometric if and only if $\Delta(\tau)>0$.
\end{prop}

\begin{proof}
Applying \cref{lemma: uniform localization} with $\sigma=\sigma_0$, $\phi=1$, and $M\coloneqq m_{\sigma_0}(\Phi(\OC_x))=\abs{Z_0(\OC_x)}+r\abs{Z_0(S)}$, we obtain a constant $0<\delta<1/4$ independent of $x$.
Put $U\coloneqq\{\tau\in\Sigma\mid d(\sigma_0,\tau)<\delta\}$, and fix $\tau\in U$.
For $x\in X$, the $\tau$-Harder--Narasimhan filtration of $\Phi(\OC_x)$ and all its Jordan--H\"older refinements consist of short exact sequences in $\PC_0(1)$.
Note that $S$ and all $\OC_x$ are $\tau$-stable since they are simple in $\PC_0(1)$.
Thus the distance bound gives
\begin{equation}\label{equation: spherical twist phase window}
\begin{aligned}
1-\delta&<\phi_\tau(\OC_x),\  \phi_\tau(S)<1+\delta,\\
1-\delta&<\phi_\tau^-(\Phi(\OC_x))
\leq\phi_\tau^+(\Phi(\OC_x))<1+\delta.
\end{aligned}
\end{equation}
In particular, we obtain $\phi_\tau(\OC_x)- \phi_\tau(S)\in (-2\delta,2\delta)\subset(-1,1)$.

Let $A\subset\Phi(\OC_x)$ be a nonzero proper subobject in $\PC_0(1)$. 
By \cref{lemma: wall subobjects}, there exists an integer $0\leq k<r$ such that $A$ fits into an exact sequence
\begin{equation}\label{equation: spherical twist subobject}
0\longrightarrow\OC_x\longrightarrow A
\longrightarrow S^{\oplus k}\longrightarrow0.
\end{equation}
Therefore, \labelcref{equation: extension at wall}, \labelcref{equation: spherical twist subobject}, and $\tau$-stability of $\OC_x$ and $S$ give
\begin{gather}
\Im\bigl(Z_\tau(A)\overline{Z_\tau(\Phi(\OC_x))}\bigr)
=-(r-k)\Delta(\tau), \text{ and}
\label{equation: spherical twist phase comparison}
\\
\Delta(\tau)
=\abs{Z_\tau(S)}\abs{Z_\tau(\OC_x)}
\sin\bigl(\pi(\phi_\tau(S)-\phi_\tau(\OC_x))\bigr).
\label{equation: extension factor phases}
\end{gather}

\labelcref{item: positive extension determinant}:
Assume that $\Delta(\tau)>0$ and $\Phi(\OC_x)$ is not $\tau$-stable.
Let $F_1,\ldots,F_m$ be the $\tau$-stable factors of $\Phi(\OC_x)$, with phases $\phi_1\geq\cdots\geq\phi_m$ and $m\geq2$. 
Then, these satisfy $0\leq\phi_1-\phi_j<2\delta<1$ for $j>1$ by \labelcref{equation: spherical twist phase window}.

By \cref{lemma: uniform localization}, $F_1$ is a nonzero proper subobject of $\Phi(\OC_x)$ in $\PC_0(1)$.
Thus,
\[
\Im\bigl(Z_\tau(F_1)\overline{Z_\tau(\Phi(\OC_x))}\bigr)
=\sum_{j=2}^m
\abs{Z_\tau(F_1)}\abs{Z_\tau(F_j)}
\sin\bigl(\pi(\phi_1-\phi_j)\bigr)
\geq0.
\]
This contradicts the assumption and \labelcref{equation: spherical twist phase comparison} applied to $A=F_1$.

\labelcref{item: zero extension determinant,item: negative extension determinant}: 
Assume that $\Delta(\tau)=0$ (resp. $\Delta(\tau)<0$).
By \labelcref{equation: extension factor phases,equation: spherical twist phase window}, $\phi_\tau(\OC_x) = \phi_\tau(S)$ (resp. $\phi_\tau(\OC_x) > \phi_\tau(S)$). Therefore, \labelcref{equation: extension at wall} shows that $\Phi(\OC_x)$ is strictly $\tau$-semistable (resp. $\tau$-unstable).
\end{proof}

\begin{thm}\label{theorem: geometric components}
\label{corollary: geometric components}
In \cref{setup: geometric wall}, $\Stabgeo(X)$ is disconnected.
\end{thm}
\begin{proof}
By \cref{theorem: Bridgeland's deformation theorem} and the linear independence of $[\OC_x]$ and $[S]$, we can choose $\tau$ in the neighborhood $U$ of \cref{proposition: extension stability} such that $Z_\tau(\OC_x)=Z_0(\OC_x)<0$ and $\Im Z_\tau(S)<0$.
Then $\Delta(\tau)>0$.
By \cref{proposition: extension stability}, both $\tau$ and $\Phi^{-1}\tau$ are geometric.
They belong to distinct components $\Sigma$ and $\Phi^{-1}\Sigma$, respectively.
This completes the proof.
\end{proof}

\subsection{Geometric stability conditions in families}\label{subsection: geometric stability conditions in families}

We apply \cref{proposition: extension stability} to spherical twists in the setting of \cref{setup: disconnectedness in families}.
\begin{thm}\label{theorem: family geometricity}
\label{corollary: family geometricity}
Under the assumptions of \cref{theorem: disconnectedness in families}, fix $b\in B(\CB)$.
Assume that there exist a nonempty subset $I\subset T_b$ and $\sigma_0=(Z_0,\PC_0)\in\Sigma_b$ satisfying the following conditions:
\begin{itemize}
\item $E_t$ is a vector bundle of positive rank for $t\in I$.
\item All objects in $\{E_t[1]\mid t\in I\}$ and all skyscraper sheaves on $\XC_b$ are $\sigma_0$-stable of phase one.
\end{itemize}
Then, there exists a geometric stability condition $\tau\in\Sigma_b$ such that $T_{E_t}^{-1}\tau$ is geometric for $t\in I$.
In particular, $\Stabgeo(\XC_b)$ meets at least $2^{\abs{I}}-1$ connected components of $\Stab(\XC_b)$.
\end{thm}

\begin{proof}
We first verify \cref{setup: geometric wall} for $\Phi=T_{E_t}$ and $S=E_t[1]$, for each $t\in I$.
By connectedness of $T$ and \cref{lemma: numerical identification}, the numerical classes of the $E_t$ coincide. Let $r>0$ denote their common rank.
Since $E_t$ is locally free of rank $r$, the defining triangle of the spherical twist gives an exact sequence in $\PC_0(1)$
\[
0\longrightarrow\OC_x\longrightarrow T_{E_t}(\OC_x)
\longrightarrow E_t[1]^{\oplus r}\longrightarrow0.
\]
By \cref{theorem: disconnectedness in families}, the components $\Sigma_b$ and $T_{E_t}^{-1}\Sigma_b$ for $t\in I$ are pairwise distinct.
Thus, \cref{setup: geometric wall} applies with $\Phi=T_{E_t}$ and $S=E_t[1]$.

For each $t\in I$, \cref{proposition: extension stability} gives an open neighborhood $U_t\subset\Sigma_b$ of $\sigma_0$.
Put $U\coloneqq\bigcap_{t\in I}U_t$.
Since $I$ is finite, $U$ is an open neighborhood of $\sigma_0$.
For $\tau=(Z_\tau,\PC_\tau)\in U$,
\[
\Delta(\tau)
=\Im\bigl(Z_\tau(E_t[1])\overline{Z_\tau(\OC_x)}\bigr)
=\Im\bigl(Z_\tau(\OC_x)\overline{Z_\tau(E_t)}\bigr)
\]
is independent of $t$ and $x$.

Choose a real basis of $\Knum(\XC_b)_{\RB}$ containing $[E_t[1]]$ and $[\OC_x]$.
For $\epsilon>0$, define $Z_\epsilon\in\Hom(\Knum(\XC_b),\CB)$ by $Z_\epsilon = Z_0$ on the basis except for $[E_t[1]]$ and
\[
Z_\epsilon(E_t[1])=Z_0(E_t[1])-\sqrt{-1}\epsilon.
\]
By \cref{theorem: Bridgeland's deformation theorem}, for sufficiently small $\epsilon>0$, there exists $\tau\in U$ with $Z_\tau=Z_\epsilon$.
Since $Z_0(\OC_x)$ and $Z_0(E_t[1])$ are negative real numbers, we obtain
\[
\Delta(\tau)
=\Im\bigl((Z_0(E_t[1])-\sqrt{-1}\epsilon)Z_0(\OC_x)\bigr)
=-\epsilon Z_0(\OC_x)>0.
\]
Since $\Delta(\tau)$ is independent of $t$ and $x$, \cref{proposition: extension stability} shows that $\tau$ and $T_{E_t}^{-1}\tau$ are geometric for $t\in I$.

Write $T_b=\{t_1,\ldots,t_d\}$.
For $J\subseteq I$, define $\mathbf a_J=(a_{J,1},\ldots,a_{J,d})\in\ZZ^d$ by $a_{J,i}=1$ if $t_i\in J$ and $a_{J,i}=0$ otherwise, and put $\Phi_J\coloneqq\Phi_{\mathbf a_J,b}$.
By orthogonality, for nonempty $J$ and $x\in\XC_b$, we have an exact sequence
\[
0\longrightarrow\OC_x\longrightarrow\Phi_J(\OC_x)
\longrightarrow\bigoplus_{t\in J}E_t[1]^{\oplus r}
\longrightarrow0
\]
in $\PC_0(1)$, and $\Hom(E_t[1],\Phi_J(\OC_x))=0$ for $t\in J$.
Since the $E_t[1]$ have the same numerical class, the proof of \cref{lemma: wall subobjects} shows that every nonzero proper subobject $A\subset\Phi_J(\OC_x)$ in $\PC_0(1)$ satisfies $[A]=[\OC_x]+k[E_t[1]]$ for any $t\in J$ and some integer $0\leq k<r\abs{J}$.
The proof of \labelcref{item: positive extension determinant} of \cref{proposition: extension stability} therefore applies with $r$ replaced by $r\abs{J}$.
Since $I$ is finite, taking $\epsilon$ sufficiently small makes $\Phi_J^{-1}\tau$ geometric for every $J\subseteq I$.

By \cref{theorem: disconnectedness in families}, if $\Phi_J^{-1}\Sigma_b=\Phi_K^{-1}\Sigma_b$, then $\mathbf a_J-\mathbf a_K\in\ZZ(1,\ldots,1)$.
For $J,K\subsetneq I$, this forces $J=K$.
The $2^{\abs{I}}-1$ proper subsets of $I$ therefore give distinct connected components meeting $\Stabgeo(\XC_b)$.
\end{proof}

\section{Separating braid group orbits of semiorthogonal decompositions}\label{section: separating braid group orbits of semiorthogonal decompositions}

In this section, we extend the spherical twist argument of \cite[Lemma~2.2 and Theorem~4.3]{preprint_anya_2026_a_smooth_projective_counterexample_to_bondalpolishchuks_conjecture} to semiorthogonal decompositions with an exceptional component.
As in \cref{section: disconnectedness in families}, we allow semilinear autoequivalences.
For families, the deformation theory of semiorthogonal decompositions in \cite{preprint_pieter_shinnosuke_andrea_2025_moduli_spaces_of_semiorthogonal_decompositions_in_families} replaces the use of relative stability conditions.

Recall that the braid group $B_m$ on $m$ strands acts on the set of semiorthogonal decompositions $(\AC_1,\ldots,\AC_m)$ with admissible components by mutations of adjacent components. See \cite{seidel_thomas_2001_braid_group_actions_on_derived_categories_of_coherent_sheaves,bondal_kapranov_1989_representable_functors_serre_functors_and_reconstructions} for more details on mutations and this action.

\subsection{Separation by semilinear autoequivalences}\label{subsection: separation of braid orbits by semilinear autoequivalences}

We first record the compatibility of semilinear autoequivalences with mutations
(cf. \cite[Lemma~2.2(ii)]{kuznetsov_2019_calabiyau_and_fractional_calabiyau_categories}).

\begin{lem}\label{lemma: semilinear commutes with mutations}
Let $\TC$ be a $k$-linear triangulated category, and $\DB=(\AC_1,\ldots,\AC_m)$ a non-trivial semiorthogonal decomposition of $\TC$ with admissible components.
For a $k$-semilinear exact autoequivalence $g$ of $\TC$ and $\beta\in B_m$, we have $g(\beta\DB)=\beta(g(\DB))$.
\end{lem}
\begin{proof}
We may assume that $m=2$. 
The left mutation of $(\AC_1,\AC_2)$ is $(\AC_1^\perp,\AC_1)$.
Since $g$ is an exact equivalence, it preserves the vanishing of Hom.
Therefore, $g(\AC_1^\perp) = g(\AC_1)^\perp$.
\end{proof}

\begin{setup}\label{setup: separation of braid orbits}
Let $X$ be a smooth projective variety of dimension $n\geq2$ over an algebraically closed field $k$ of characteristic zero and $\DB=(\AC_1,\ldots,\AC_r)$ be a non-trivial semiorthogonal decomposition of $D^b(X)$.
Assume that the following data are given:
\begin{enumerate}
\item mutually orthogonal $n$-spherical objects $F_1,\ldots,F_d\in D^b(X)$, where $d\geq2$;
\item\label{item: G preserve sod components and acts trans on spherical objs} a group $G$ consisting of $k$-semilinear autoequivalences preserving each $\AC_j$ and acting transitively on $\{F_1,\ldots,F_d\}$ up to isomorphism;
\item\label{item: non-rothogonality condition of exceptional object} an exceptional object $E$ such that $\AC_j=\langle E\rangle$ for some $j$ and $\RHom(E,F_1)\ne0$.
\end{enumerate}
For $\mathbf a=(a_1,\ldots,a_d)\in\ZZ^d$, put
$\Phi_{\mathbf a}=T_{F_1}^{a_1}\cdots T_{F_d}^{a_d}$ and
$\DB_{\mathbf a}=\Phi_{\mathbf a}(\DB)$.
Note that the twists in $\Phi_{\mathbf a}$ commute because the $F_i$ are mutually orthogonal.

Moreover, we define an action of $G$ on $\ZZ^d$ as follows: for $g\in G$ and $\mathbf{a}\in\ZZ^d$, define $(g\mathbf a)_i=a_{s_g^{-1}(i)}$, where $s_g$ is the permutation determined by $g(F_i)\cong F_{s_g(i)}$.
This is well-defined by \labelcref{item: G preserve sod components and acts trans on spherical objs}.
\end{setup}

\begin{thm}\label{theorem: separation of braid orbits}
In \cref{setup: separation of braid orbits}, let $\mathbf a,\mathbf b\in\ZZ^d$.
If $\DB_{\mathbf a}$ and $\DB_{\mathbf b}$ belong to the same $B_r$-orbit, then
\[
\mathbf a-\mathbf b\in\ZZ(1,\ldots,1).
\]
\end{thm}
\begin{proof}

Suppose that $\beta\DB_{\mathbf a}=\DB_{\mathbf b}$ for some $\beta\in B_r$, and put $\mathbf c\coloneqq\mathbf b-\mathbf a$.
Applying $\Phi_{\mathbf a}^{-1}$ and using \cref{lemma: semilinear commutes with mutations}, we obtain
\begin{equation}\label{equation: beta D is Phi_c D}
\beta\DB=\Phi_{\mathbf c}(\DB).
\end{equation}

Fix $g\in G$, and put $\mathbf e\coloneqq\mathbf c-g\mathbf c$.
Since $g(\DB)=\DB$, \cref{lemma: semilinear commutes with mutations}, \labelcref{equation: beta D is Phi_c D}, and the identity $g\Phi_{\mathbf c}g^{-1}\cong\Phi_{g\mathbf c}$ give
\[
\beta\DB
=g(\beta\DB)
=g\bigl(\Phi_{\mathbf c}(\DB)\bigr)
=(g\Phi_{\mathbf c}g^{-1})\bigl(g(\DB)\bigr)
=\Phi_{g\mathbf c}(\DB).
\]
Comparing this with \eqref{equation: beta D is Phi_c D} and applying $\Phi_{g\mathbf c}^{-1}$, we obtain $\Phi_{\mathbf e}(\DB)=\DB$.

Since $\Phi_{\mathbf e}$ preserves each component of $\DB$, $\Phi_{\mathbf e}(E)\cong E[s]$ holds for some $s\in\ZZ$.
The orthogonality of the $F_i$ gives $\Phi_{\mathbf e}^{-1}(F_i)\cong F_i[(n-1)e_i]$.
Therefore, it follows that
\[
\begin{aligned}
\RHom(E,F_i)[-s]
&\cong\RHom(\Phi_{\mathbf e}(E),F_i)\\
&\cong\RHom(E,\Phi_{\mathbf e}^{-1}(F_i))\\
&\cong\RHom(E,F_i)[(n-1)e_i].
\end{aligned}
\]

Since $G$ also preserves $\langle E\rangle$, $g(E)$ is isomorphic to $E$ up to shift for $g\in G$.
Additionally, since semilinear autoequivalences preserve the vanishing of graded Hom,
transitivity and the assumption $\RHom(E,F_1)\ne0$ give
\begin{equation*}
\RHom(E,F_i)\ne0
\quad\text{for }1\leq i\leq d.
\end{equation*}
Therefore, $\RHom(E,F_i)$ is a nonzero bounded complex of finite-dimensional $k$-vector spaces, and thus we obtain $-s=(n-1)e_i$ for every $i$.
In particular, $e_1 = \cdots = e_d$ since $n\geq2$.
Since $\mathbf{e}$ satisfies $e_1+ \cdots +e_d = 0$ by definition, we have $\mathbf e=0$.
This holds for $g\in G$, so $\mathbf c$ is $G$-invariant.
Transitivity implies that all entries of $\mathbf c$ are equal.
\end{proof}

\begin{rem}
The converse of \cref{theorem: separation of braid orbits} does not hold in general; see \cref{remark: counter example of the converse of separation theorem}.
\end{rem}

\begin{rem}\label{remark: numerical images of the separated decompositions}
In \cref{setup: separation of braid orbits}, suppose in addition that $[F_1]=\cdots=[F_d]$ in $\Knum(X)$, and put 
\[
L=\{(a_1,\ldots,a_d)\in\ZZ^d\mid a_1+\cdots+a_d=0\}.
\]
Since all $T_{F_i}$ induce the same automorphism of $\Knum(X)$, $\Phi_{\mathbf a}$ acts trivially on $\Knum(X)$ for every $\mathbf a\in L$.
Consequently, for each $1\leq j\leq r$ and $\mathbf a\in L$,
the images of $K_0(\AC_j)$ and $K_0(\Phi_{\mathbf a}(\AC_j))$ in $\Knum(X)$ coincide. 
Moreover, the decompositions $\DB_{\mathbf a}$ for $\mathbf a\in L$ belong to pairwise distinct $B_r$-orbits.
\end{rem}

\subsection{Separation in families}\label{subsection: separation of braid orbits in families}

We apply \cref{theorem: separation of braid orbits} on the geometric generic fiber and use the deformation theory of semiorthogonal decompositions to obtain separation on every geometric fiber.
Recall that, for a morphism $f\colon\XC\to B$, a semiorthogonal decomposition $(\AC_1,\ldots,\AC_r)$ of $\Perf(\XC)$ is called \emph{$B$-linear} if each $\AC_j$ is closed under tensoring with $\lderived f^*P$ for $P\in\Perf(B)$.

\begin{setup}\label{setup: separation of braid orbits in families}
Let $B$ be a smooth quasi-projective variety over $\CB$, and let $f\colon\XC\to B$ be a smooth projective morphism with geometrically connected fibers of dimension $n\geq2$.
Let $\bar\eta$ be a geometric generic point of $B$.
Assume that the following data are given:
\begin{enumerate}
\item a non-trivial $B$-linear semiorthogonal decomposition
$\DB=(\AC_1,\ldots,\AC_r)$ of $\Perf(\XC)$ whose induced decomposition
$\DB_{\bar b}=(\AC_{1,\bar b},\ldots,\AC_{r,\bar b})$
has nonzero components for every geometric point $\bar b$ of $B$;
\item a connected finite \'etale cover $\pi\colon T\to B$ of degree $d\geq2$ and a perfect complex $\FC\in\Perf(\XC_T)$ such that, for every geometric point $\bar b$ of $B$, the restrictions $F_t\coloneqq\FC_t$ for $t\in T_{\bar b}$ are mutually orthogonal $n$-spherical objects of $D^b(\XC_{\bar b})$;
\item an exceptional object $E\in D^b(\XC_{\bar\eta})$ such that
$\AC_{j,\bar\eta}=\langle E\rangle$ for some $j$ and
$\RHom(E,F_t)\ne0$ for some $t\in T_{\bar\eta}$.
\end{enumerate}
All restrictions of perfect complexes are derived pullbacks.

Fix a geometric point $\bar b$ of $B$ and an ordering $T_{\bar b}=\{t_1,\ldots,t_d\}$.
For $\mathbf a=(a_1,\ldots,a_d)\in\ZZ^d$, put
$\Phi_{\mathbf a,\bar b}\coloneqq T_{F_{t_1}}^{a_1}\cdots T_{F_{t_d}}^{a_d}$ and 
$\DB_{\mathbf a,\bar b}\coloneqq\Phi_{\mathbf a,\bar b}(\DB_{\bar b})$.
\end{setup}

\begin{thm}\label{theorem: separation of braid orbits in families}
In \cref{setup: separation of braid orbits in families}, let $\mathbf a,\mathbf a'\in\ZZ^d$.
If $\DB_{\mathbf a,\bar b}$ and $\DB_{\mathbf a',\bar b}$ belong to the same $B_r$-orbit, then
$\mathbf a-\mathbf a'\in\ZZ(1,\ldots,1)$.
\end{thm}
\begin{proof}
Suppose that $\beta\DB_{\mathbf a,\bar b}=\DB_{\mathbf a',\bar b}$ for some $\beta\in B_r$.
Choose a connected affine \'etale neighborhood $(S,\bar s)\to(B,\bar b)$ such that
$T_S\simeq\coprod_{i=1}^d S$.
Let $s_1,\ldots,s_d$ be the corresponding sections, numbered so that $s_i(\bar s)=t_i$, and let $\FC_i\in\Perf(\XC_S)$ be the pullback of $\FC$ along $S\xrightarrow{s_i}T_S\to T$.
Since $S$ is smooth and connected, it is integral.

By \cref{proposition: relative spherical twist}, the twists $T_{\FC_i}$ are $S$-linear autoequivalences, and $T_{\FC_i}^{\pm1}$ commute with base change.
For $\mathbf c=(c_1,\ldots,c_d)\in\ZZ^d$, put $\Phi_{\mathbf c,S}\coloneqq T_{\FC_1}^{c_1}\cdots T_{\FC_d}^{c_d}$, and $\DB_{\mathbf c,S}\coloneqq\Phi_{\mathbf c,S}(\DB_S)$.
Then $\DB_{\mathbf c,S}$ is an $S$-linear semiorthogonal decomposition whose restriction at $\bar s$ is $\DB_{\mathbf c,\bar b}$.

By \cite[Theorem~A]{preprint_pieter_shinnosuke_andrea_2025_moduli_spaces_of_semiorthogonal_decompositions_in_families}, the functor of semiorthogonal decompositions of length $r$ for $\XC_S\to S$ is represented by an algebraic space $\mathcal M$ \'etale over $S$.
Since mutations commute with base change by \cite[Proposition~9.18]{preprint_pieter_shinnosuke_andrea_2025_moduli_spaces_of_semiorthogonal_decompositions_in_families}, the decompositions $\beta\DB_{\mathbf a,S}$ and $\DB_{\mathbf a',S}$ define sections $u,v\colon S\to\mathcal M$ with $u(\bar s)=v(\bar s)$.
Put $U\coloneqq\{s\in S\mid u(s)=v(s)\}$.
Since the diagonal of $\mathcal M\to S$ is an open immersion, $U$ is open in $S$.
As $S$ is integral and $\bar s\in U$, the generic point of $S$ belongs to $U$.

Since $S\to B$ is dominant and \'etale, we may choose a lift of $\bar\eta$ to the generic point of $S$.
Restricting $u$ and $v$ to this lift gives $\beta\DB_{\mathbf a,\bar\eta}=\DB_{\mathbf a',\bar\eta}$,
where the ordering of $T_{\bar\eta}$ is induced by $s_1,\ldots,s_d$.

Put $K=\CB(B)$ and choose an algebraic closure $\overline K$ defining $\bar\eta$.
Since $\DB$ is defined over $B$, the semilinear pullbacks $g^*$ for $g\in\Gal(\overline K/K)$ preserve each $\AC_{j,\bar\eta}$.
Since $\FC$ is defined over $T$ and $T$ is connected, these pullbacks act transitively on $\{F_t\mid t\in T_{\bar\eta}\}$ up to isomorphism.
By the assumption on $E$, \cref{theorem: separation of braid orbits} applies after simultaneously reordering the objects $F_t$ and the coordinates of $\mathbf a,\mathbf a'$.
Therefore, $\mathbf a-\mathbf a'\in\ZZ(1,\ldots,1)$.
\end{proof}
\section{Calabi--Yau threefolds via rigid rational curves}\label{section: Calabi--Yau threefolds via rigid rational curves}

We first apply the separation criteria in \cref{section: disconnectedness in families} and \cref{section: A local criterion for geometricity} to general quintic threefolds.
Also, we treat the cases of complete intersection Calabi--Yau threefolds and the Dwork pencil.
We then consider products with varieties whose Albanese morphisms are finite and show that these separation phenomena occur in any dimension at least three. 

\subsection{Quintic threefolds}\label{subsection: rational curves on quintic threefolds}
Let $B\subset\PB H^0(\PP{4},\OC_{\PP{4}}(5))$ be the open subset parametrizing smooth quintic threefolds.
Let $f\colon\XC\to B$ be the universal family, equipped with the relatively very ample line bundle
$\HC\coloneqq(\OC_{\PP{4}}(1)\boxtimes\OC_B)|_\XC$.
Put $H_b\coloneqq c_1(\HC|_{\XC_b})$.
Consider the incidence variety
\[
\mathcal I\coloneqq
\{(\ell,b)\in\Gr(2,5)\times B\mid\ell\subset\XC_b\},
\]
and let $\pi\colon\mathcal I\to B$ be the second projection.
The fiber $\pi^{-1}(b)$ is the scheme of lines on $\XC_b$.
The variety $\mathcal I$ is a nonempty open subset of the projective bundle over $\Gr(2,5)$ whose fiber over $[\ell]$ is
$\PB H^0(\PP{4},\IC_\ell(5))$.
Thus $\mathcal I$ is irreducible.

\begin{rem}\label{remark: lines on general quintic Calabi--Yau threefolds}
By \cite[Theorem~2.1]{debarre_manivel_1998_sur_la_variete_des_espaces_lineaires_contenus_dans_une_intersection_complete}, the fiber $\pi^{-1}(b)$ is nonempty, reduced, and zero-dimensional for general $b\in B$.
Its length is $2875$ by \cite[Theorem~4.3]{debarre_manivel_1998_sur_la_variete_des_espaces_lineaires_contenus_dans_une_intersection_complete}.
Consequently, $\pi$ is finite and \'etale of degree greater than one over a nonempty open subset of $B$.

By \cite[Theorem~4.1]{johnsen_kleiman_1996_rational_curves_of_degree_at_most_9_on_a_general_quintic_threefold}, distinct lines on $\XC_b$ are disjoint for general $b\in B$.
Thus, for the family $\XC\to B$ above, we may choose a nonempty open subset $U\subset B$ such that the restriction $\pi|_T\colon T\to U$, where $T\coloneqq\pi^{-1}(U)$, is finite and \'etale of degree greater than one, and distinct lines on every geometric fiber of $\XC_U\to U$ are disjoint.
Since $T$ is a nonempty open subset of $\mathcal I$, it is irreducible and hence connected.
\end{rem}

\begin{thm}\label{theorem: disconnectedness for general quintics}
Let $b\in U(\CB)$, and put $X=\XC_b$ and $H=H_b$.
Then $\Stab(X)$ has infinitely many connected components, and every central charge occurring in $\Stab^\dagger_H(X)$ has lifts in infinitely many distinct connected components.
\end{thm}
\begin{proof}
For a geometric point $\bar b$ of $U$ and a line $\ell\subset\XC_{\bar b}$,
$N_{\ell/\XC_{\bar b}}\cong\OC_{\PP{1}}(-1)^{\oplus2}$ holds since $\deg N_{\ell/\XC_{\bar b}}=-2$ by adjunction and the tangent space to $T_{\bar b}$ at $[\ell]$ is $H^0(\ell,N_{\ell/\XC_{\bar b}})=0$.
Therefore, \cite[Lemma~3.4]{seidel_thomas_2001_braid_group_actions_on_derived_categories_of_coherent_sheaves} shows that $\OC_\ell(-1)$ is spherical.
These are mutually orthogonal because their supports are disjoint.

Let $j\colon\LC\hookrightarrow\XC_T$ be the universal line, whose fiber over $(\ell,b)\in T$ is $\ell\subset\XC_b$.
The complex $j_*\OC_\LC(-1)$ is perfect, and its fiber restrictions are the spherical sheaves.
The classes of $\HC^{\otimes i}|_{\XC_{\bar b}}$ for $0\leq i\leq3$ form a basis of $\Knum(\XC_{\bar b})_\QB$ for every geometric point $\bar b$ of $U$.
By \cref{thm:projective_family_stability}, choose $\sigma_\bullet\in\Stab(\XC_U/U)$ induced by the embedding $\XC_U\hookrightarrow\PP{4}\times U$.
Thus, \cref{theorem: disconnectedness in families} applies over $U$, with $\Sigma_b=\Stab^\dagger_{H_b}(\XC_b)$, and proves both assertions.
\end{proof}

We now verify the wall condition in \cref{theorem: family geometricity} for the evaluation bundles associated with lines.
Then we show the disconnectedness of $\Stabgeo(X)$.
For this purpose, let $X\subset\PP{4}$ be a smooth quintic and put $H\coloneqq c_1(\OC_X(1))$.
For a torsion-free sheaf $F$, we set its slope by $\mu_H(F)=H^2\ch_1(F)/(H^3\rk F)$. When $F$ is a torsion sheaf, we set $\mu_H(F) = \infty$.

First, we review the geometric stability conditions on quintic threefolds, and determine the stability of suitable bundles for applying the results in \cref{section: A local criterion for geometricity}. 
Then, we construct suitable bundles from lines on quintic threefolds.
Recall Li's double-tilt construction on the quintic threefolds in a slightly simplified form.

\begin{dfn}\label{dfn:double-tilt}
For $\beta\in\RB$, put $\ch^\beta=e^{-\beta H}\ch$. 
\begin{enumerate}
\item The first tilt $\Coh^\beta(X)$ is defined to be
\[
\Coh^\beta(X)=\langle\FC_\beta[1],\TC_\beta\rangle_{\mathrm{ex}},
\]
where $\FC_\beta\coloneqq \{ E\in\Coh(X) \mid \mu_H^+(E)\leq\beta\}$ and 
$\TC_\beta\coloneqq \{ E\in\Coh(X) \mid \mu_H^-(E)>\beta\}$.
\item For $\alpha>0$, the tilt slope on $\Coh^\beta(X)$ is defined to be
\[
\nu_{\alpha,\beta}(F)=
\begin{cases}
\frac{H\ch_2^\beta(F)-\frac{\alpha^2}{2}H^3\rk F}
{H^2\ch_1^\beta(F)} & \text{when }H^2\ch_1^\beta(F)>0, \\
+\infty & \text{when }H^2\ch_1^\beta(F)=0.
\end{cases}
\]
A nonzero object $F$ is \emph{tilt-semistable} if $\nu_{\alpha,\beta}(A)\leq\nu_{\alpha,\beta}(F/A)$ for every nonzero proper subobject $A\subset F$ in $\Coh^\beta(X)$.
We denote the largest and smallest slopes of its Harder--Narasimhan factors by $\nu_{\alpha,\beta}^+(F)$ and $\nu_{\alpha,\beta}^-(F)$.
\item The second tilt $\AC_{\alpha,\beta}$ is defined to be
\[
\AC_{\alpha,\beta}
=\langle\FC'_{\alpha,\beta}[1],\TC'_{\alpha,\beta}\rangle_{\mathrm{ex}}, 
\]
where $\FC'_{\alpha,\beta}\coloneqq \{ E\in\Coh^\beta(X) \mid \nu_{\alpha,\beta}^+(E)\leq0\}$ and 
$\TC'_{\alpha, \beta}\coloneqq \{ E\in\Coh^\beta(X) \mid \nu_{\alpha,\beta}^-(E)>0\}$.
\end{enumerate}
\end{dfn}

\begin{thm}[{\cite[Theorem~1.3]{li_2019_on_stability_conditions_for_the_quintic_threefold}}]\label{thm:li}
There is a continuous family of stability conditions $\sigma_{\alpha,\beta}^a=(Z_{\alpha,\beta}^a,\PC_{\alpha,\beta}^a)$ for $\alpha>1/2$, $\beta\in\RB$, and $a>\alpha^2/6$, with $\PC_{\alpha,\beta}^a((0,1])=\AC_{\alpha,\beta}$ and
\begin{equation}\label{equation: open double tilt charge}
Z_{\alpha,\beta}^a(F)
=-\ch_3^\beta(F)+aH^2\ch_1^\beta(F)
+\sqrt{-1}\left(H\ch_2^\beta(F)-\frac{\alpha^2}{2}H^3\rk F\right).
\end{equation}
Moreover, all skyscraper sheaves are $\sigma^a_{\alpha, \beta}$-stable of phase one.
\end{thm}

Before proving that $\Stabgeo(X)$ is disconnected, we check that the stability conditions constructed in \cref{thm:li} lie in the same connected component, which will be used also in \cref{subsection: products and finite Albanese morphisms}.

\begin{lem}\label{lemma: comparison of induced and double tilt stability conditions}
Let $j\colon X\hookrightarrow\PP{N}$ be a closed immersion with $j^*\OC_{\PP{N}}(1)\cong\OC_X(H)$, and let $\tau^{\PP{N}}_{\xi,\theta}$ be the stability conditions on $\PP{N}$ in \cref{thm:projective_family_stability}.
Put $\kappa\coloneqq(N+21)/24$. For sufficiently large $\xi>0$ and every $\theta\in\RB$, set $\beta\coloneqq\theta+(N+1)/2$, $\alpha\coloneqq\sqrt{\xi^2/3-2\kappa}$, and $a\coloneqq\xi^2/2-\kappa$.
Then $j^\sharp\tau^{\PP{N}}_{\xi,\theta}$ has heart $\AC_{\alpha,\beta}$ and central charge
\[
\Re Z_{\alpha,\beta}^a+\sqrt{-1}\xi\Im Z_{\alpha,\beta}^a.
\]
In particular, $\sigma_{\alpha,\beta}^a\in\Stab^\dagger_H(X)$ for every $\alpha>1/2$, $\beta\in\RB$, and $a>\alpha^2/6$.
\end{lem}
\begin{proof}
Write $j^\sharp\tau^{\PP{N}}_{\xi,\theta}=(W,\PC)$.
Since $c_1(X)=0$ and $c_2(X)=10H^2$, we have
\[
\td(N_{X/\PP{N}})^{-1}
=\frac{\td(X)}{j^*\td(\PP{N})}
=e^{-(N+1)H/2}(1+\kappa H^2).
\]
For $F\in D^b(X)$, Grothendieck--Riemann--Roch gives
\[
\begin{aligned}
W(F)
&=-\int_X e^{-(\beta+\sqrt{-1}\xi)H}(1+\kappa H^2)\ch(F)\\
&=-\ch_3^\beta(F)
  +\left(\frac{\xi^2}{2}-\kappa\right)H^2\ch_1^\beta(F)+\sqrt{-1}\xi\left(H\ch_2^\beta(F)
  -\left(\frac{\xi^2}{6}-\kappa\right)H^3\rk F\right)\\
&=\Re Z_{\alpha,\beta}^a(F)+\sqrt{-1}\xi\Im Z_{\alpha,\beta}^a(F).
\end{aligned}
\]

Assume first that $\xi,\theta\in\QB$.
By \cite[Theorem~1.2]{preprint_yiran_soheyla_2026_stability_conditions_on_threefolds}, the heart $\PC((0,1])$ is obtained by double tilting with respect to the twisted Chern character $e^{-\theta H}\td(N_{X/\PP{N}})^{-1}\ch=(1+\kappa H^2)\ch^\beta$.
The first tilt is $\Coh^\beta(X)$, and the slope for the second tilt is
\[
\frac{H\ch_2^\beta(F)
  -\left(\frac{\xi^2}{6}-\kappa\right)H^3\rk F}
{H^2\ch_1^\beta(F)}
=\nu_{\alpha,\beta}(F).
\]
Thus, $\PC((0,1])=\AC_{\alpha,\beta}$.

For sufficiently large $\xi$, we have $\alpha>1/2$ and $a=3\alpha^2/2+2\kappa>\alpha^2/6$.
The $\GLt$-action gives stability conditions with heart $\AC_{\alpha,\beta}$ and central charge $\Re Z_{\alpha,\beta}^a+\sqrt{-1}t\Im Z_{\alpha,\beta}^a$ for $t>0$.
These stability conditions depend continuously on $(t,\alpha,\beta,a)$ by the continuity of the action and \cref{thm:li}.
Taking $t=\xi$ and the parameters in the statement, this family agrees with $j^\sharp\tau^{\PP{N}}_{\xi,\theta}$ for rational $\xi,\theta$.
Since the latter family is also continuous by \cref{thm:projective_family_stability}, they agree for real $\xi,\theta$ as well.

For fixed $\xi,\theta$, letting $t$ vary from $\xi$ to $1$ gives a path from $j^\sharp\tau^{\PP{N}}_{\xi,\theta}$ to $\sigma_{\alpha,\beta}^a$.
Since $j^\sharp\tau^{\PP{N}}_{\xi,\theta}\in\Stab^\dagger_H(X)$, we obtain $\sigma_{\alpha,\beta}^a\in\Stab^\dagger_H(X)$ for these parameters.
The domain defined by $\alpha>1/2$, $\beta\in\RB$, and $a>\alpha^2/6$ is convex.
Therefore, the continuity in \cref{thm:li} shows that $\sigma_{\alpha,\beta}^a\in\Stab^\dagger_H(X)$ for all parameters in this domain.
\end{proof}

We next prove stability of the evaluation bundles at a wall.

\begin{lem}\label{lemma: rank bound for tilt subobjects}
Let $E$ be a $\mu_H$-stable vector bundle of rank two with $c_1(E)=-H$.
Suppose that $\nu_{\alpha,\beta}(E)=0$ and $\alpha, \beta$ satisfy
$\alpha>0$, $\beta<-1/2$, $\alpha+\beta>-1$, and $3\alpha+2\beta>-1$.

Let $A\hookrightarrow E$ be a nonzero $\nu_{\alpha,\beta}$-semistable subobject in $\Coh^\beta(X)$ with $\nu_{\alpha,\beta}(A)\geq0$.
Then $A$ is a coherent sheaf such that $E/A$ is a zero-dimensional sheaf and $\nu_{\alpha,\beta}(A)=0$.
\end{lem}
\begin{proof}
By assumption, the bundle $E$ is $\mu_H$-stable with $\mu_H(E)=-1/2>\beta$, so $E\in\TC_\beta$.
Let $0\to A\to E\to Q\to0$ be the given exact sequence in $\Coh^\beta(X)$.
By taking the ordinary cohomology $\HC^i(-)$, we obtain $\HC^{-1}(A)=0$ since the objects $A$ and $Q$ have cohomology only in degrees $-1$ and $0$, and $E$ is a sheaf.
Thus, $A\simeq\HC^0(A)\in\TC_\beta$, and the remaining terms give an exact sequence
\[
0\longrightarrow\HC^{-1}(Q)\longrightarrow A
\longrightarrow E\longrightarrow\HC^0(Q)\longrightarrow0
\]
in $\Coh(X)$, with $\HC^{-1}(Q)\in\FC_\beta$.

Set $I=\Image(A\to E)$ in $\Coh(X)$.
The exact sequence $0\to\HC^{-1}(Q)\to A\to I\to0$ shows that $A$ is torsion-free, since both $\HC^{-1}(Q)$ and $I\subset E$ are torsion-free.
If $I=0$, then $A\simeq\HC^{-1}(Q)\in\TC_\beta\cap\FC_\beta=\{0\}$, which is a contradiction to $A\ne0$.
Therefore, we have $\rk A\geq\rk I>0$.
Since the torsion-free sheaf $A$ belongs to $\TC_\beta$, we have $\mu_H(A)>\beta$ and hence $H^2\ch_1^\beta(A)>0$.

The assumption $\nu_{\alpha,\beta}(A)\geq0$ gives
\begin{equation}\label{equation: tilt subobject second Chern character bound}
H\ch_2^\beta(A)\geq(5\alpha^2/2)\rk A
\end{equation}
and $\HC^{-1}(Q)\in\FC_\beta$ implies $H^2\ch_1^\beta(\HC^{-1}(Q))\leq0$.
Therefore, the Bogomolov--Gieseker inequality for tilt-semistable objects \cite[Corollary~7.3.2]{bayer_macri_toda_2014_bridgeland_stability_conditions_on_threefolds_i_bogomolovgieseker_type_inequalities} yields
\[
\bigl(H^2\ch_1^\beta(A)\bigr)^2
\geq10\rk A\cdot H\ch_2^\beta(A)
\geq25\alpha^2(\rk A)^2.
\]
This inequality and the short exact sequence $0\to\HC^{-1}(Q)\to A\to I\to0$ imply that 
\begin{equation}\label{equation: basic rank bound}
\alpha\rk A
\leq\frac{H^2\ch_1^\beta(A)}5
\leq\frac{H^2\ch_1^\beta(I)}5.
\end{equation}

Since $0\ne I\subset E$ is torsion-free, we have $\rk I\in\{1,2\}$.
Suppose first that $\rk I=1$.
The $\mu_H$-stability of $E$ gives $H^2\ch_1(I)/5=\mu_H(I)<-1/2$. 
Thus, $H^2\ch_1(I)/5\leq-1$ since it is an integer by $\Pic(X)=\ZZ H$.
Therefore, by \labelcref{equation: basic rank bound}, we obtain 
\[
\alpha\rk A\leq H^2\ch_1^\beta(I)/5 = H^2\ch_1(I)/5-\beta \leq -1-\beta, 
\]
which contradicts the assumption $\alpha+\beta > -1$.
Therefore, $\rk I=2$.

Since the quotient $E/I$ is a torsion sheaf whose first Chern class is effective, we have 
\begin{equation}\label{equation: tilt subobject image bound}
\frac{H^2\ch_1^\beta(I)}5
=\frac{H^2\ch_1(I)}5-2\beta
\leq \frac{H^2\ch_1(E)}5-2\beta
=-1-2\beta.
\end{equation}
Combining \labelcref{equation: basic rank bound,equation: tilt subobject image bound} and $3\alpha+2\beta>-1$, we obtain $2\leq\rk A\leq(-1-2\beta)/\alpha<3$, which shows $\rk A=2$.
Consequently, the torsion-free sheaf $\HC^{-1}(Q)$ has rank $\rk A-\rk I=0$, so $\HC^{-1}(Q)=0$.
This proves that $A\to E$ is an injection of coherent sheaves.
Similarly, using \labelcref{equation: basic rank bound,equation: tilt subobject image bound}, and $\rk A=2$, we obtain $H^2\ch_1(A)/5\in[2(\alpha+\beta),-1]\subset(-2,-1]$, so $H^2\ch_1(A)/5 = -1$.
Thus, $c_1(A)=-H=c_1(E)$.

Since the quotient $E/A$ satisfies $\rk(E/A) = 0$ and $c_1(E/A)=0$, $\dim\Supp(E/A)\leq 1$ and $H\ch_2(E/A)\geq0$ hold.
On the other hand, the assumption $\nu_{\alpha, \beta}(E)=0$ implies $H\ch_2^\beta(E)=\frac{\alpha^2}{2}H^3\rk E=5\alpha^2$.
Therefore, by \labelcref{equation: tilt subobject second Chern character bound}, we obtain
\begin{equation}\label{equation: tilt subobject quotient second Chern character}
0\leq H\ch_2(E/A)
=H\ch_2^\beta(E)-H\ch_2^\beta(A)
\leq0.
\end{equation}
This gives $\nu_{\alpha,\beta}(A)=0$ and $\dim\Supp(E/A) = 0$.
\end{proof}

We next prove the $\sigma_{\alpha,\beta}^a$-stability of $E[1]$.
We first recall the description of the phase-one objects in the double-tilt heart.

\begin{lem}[{\cite[proof of Theorem~8.6]{bayer_macri_stellari_2016_the_space_of_stability_conditions_on_abelian_threefolds_and_on_some_calabiyau_threefolds}}]\label{lemma: objects of phase one in the double tilt heart}
Let $\sigma=(Z,\PC)$ be a stability condition in \cref{thm:li}.
For $D\in\PC(1)$, its cohomology with respect to $\Coh^\beta(X)$ gives an exact sequence in $\AC_{\alpha,\beta}$
\[
0\longrightarrow F[1]\longrightarrow D\longrightarrow T\longrightarrow0,
\]
where $T$ is a zero-dimensional sheaf and $F$ is either zero or $\nu_{\alpha,\beta}$-semistable of slope zero.
\end{lem}

\begin{prop}\label{proposition: stability at the wall}
Under the assumptions on $X,E,\alpha,\beta$ in \cref{lemma: rank bound for tilt subobjects}, $E$ is $\nu_{\alpha,\beta}$-semistable.
If $\alpha>1/2$ and $a>\alpha^2/6$, then $E[1]$ is $\sigma_{\alpha,\beta}^a$-stable of phase one.
\end{prop}
\begin{proof}
Let $G_1,\ldots,G_m$ be $\nu_{\alpha,\beta}$-Harder--Narasimhan factors of $E$ such that $\nu_{\alpha,\beta}(G_j) < \nu_{\alpha,\beta}(G_i)$ for $i<j$.
Since $G_1$ is a semistable subobject of $E$, \cref{lemma: rank bound for tilt subobjects} gives $\nu_{\alpha,\beta}(G_j) \leq \nu_{\alpha,\beta}(G_1)\leq0$ for every $j$.
In particular, $H^2\ch_1^\beta(G_j)>0$.
As $\nu_{\alpha, \beta}(E)=0$, we have 
\[
0=H\ch_2^\beta(E)-\frac{\alpha^2}{2}H^3\rk E
=\sum_{j=1}^m H^2\ch_1^\beta(G_j)\nu_{\alpha,\beta}(G_j).
\]
Since each summand is nonpositive, we obtain $\nu_{\alpha,\beta}(G_j)=0$ and therefore, $m=1$.
This shows that $E$ is $\nu_{\alpha,\beta}$-semistable of slope zero.

Assume now that $\alpha>1/2$ and $a>\alpha^2/6$, and write $\sigma_{\alpha,\beta}^a=(Z_{\alpha,\beta}^a,\PC)$.
It is enough to show that $E[1]$ is simple in $\PC(1)$.
Assume that there exists a nonzero subobject $A\subset E[1]$ in $\PC(1)$, and put $B\coloneqq E[1]/A$.
By \cref{lemma: objects of phase one in the double tilt heart}, there is an exact sequence $0\to F_A[1]\to A\to T_A\to0$ in $\AC_{\alpha,\beta}$, where $T_A$ is a zero-dimensional sheaf and $F_A$ is either zero or $\nu_{\alpha,\beta}$-semistable of slope zero.
Taking cohomology with respect to $\Coh^\beta(X)$, we have $\HC_{\Coh^\beta(X)}^{-1}(A) \cong F_A$ and $\HC_{\Coh^\beta(X)}^{0}(A) \cong T_A$. Similarly for $B$, $\HC_{\Coh^\beta(X)}^{-1}(B) \cong F_B$ and $\HC_{\Coh^\beta(X)}^{0}(B) \cong T_B$ also hold.
Moreover, taking cohomology of the short exact sequence defining $B$ gives $T_B = 0$, i.e. $B\simeq F_B[1]$, and an exact sequence in $\Coh^\beta(X)$
\[
0\longrightarrow F_A\longrightarrow E\longrightarrow F_B\longrightarrow T_A\longrightarrow0.
\]
Here, $F_A\neq0$. 
Indeed, if $F_A=0$, then $A\simeq T_A$; since $\dim\Supp(T_A)=0$, 
\[
\Hom(T_A,E[1])
\simeq\Ext^2(E,T_A)^*
\simeq H^2(X,E^\vee\otimes T_A)^*
=0,
\]
which is a contradiction.

By \cref{lemma: rank bound for tilt subobjects}, the map $F_A\to E$ is an injection of coherent sheaves with zero-dimensional cokernel $E/F_A$.
The exact sequence $0\to E/F_A\to F_B\to T_A\to0$ in $\Coh^\beta(X)$ shows that $F_B$ is a zero-dimensional sheaf.
If $F_B\ne0$, then $\nu_{\alpha,\beta}(F_B)=+\infty$, which contradicts $\nu_{\alpha,\beta}(F_B)=0$.
Thus, $F_B=0$ and $B=0$, that is, $E[1]$ is simple.
\end{proof}

\begin{lem}\label{lemma: evaluation bundles associated with lines}
Let $\ell\subset X$ be a line and $E_\ell \coloneqq \Ker(H^0(X,\IC_\ell(H))\otimes\OC_X\stackrel{\ev}{\twoheadlongrightarrow}\IC_\ell(H))$.
Then, $E_\ell$ is a  $\mu_H$-stable vector bundle with
\begin{equation}\label{equation: Chern character of evaluation bundles}
\bigl(\rk E_\ell,H^2\ch_1(E_\ell),H\ch_2(E_\ell),\ch_3(E_\ell)\bigr)
=\left(2,-5,-\frac32,\frac76\right).
\end{equation}
Moreover, if $N_{\ell/X}\cong\OC_{\PP{1}}(-1)^{\oplus2}$, then $E_\ell$ is spherical.
\end{lem}
\begin{proof}
The ideal of a line in $\PP{4}$ is generated by three independent linear forms.
Consequently, $\IC_{\ell}(1)$ is globally generated and $H^0(X,\IC_{\ell}(1))\cong\CB^3$.
The exact sequence defining the ideal sheaf, together with Kodaira vanishing, gives $H^j(X,\IC_{\ell}(1))=0$ for $j>0$.
Since $\ell\subset X$ is a regular immersion of codimension two, its ideal sheaf has local projective dimension at most one.
Therefore, the kernel $E_\ell$ is locally free of rank two.

Next, we show the slope stability of $E_\ell$.
The defining exact sequence of $E_\ell$ gives $H^0(X,E_\ell)=0$ and $c_1(E_\ell)=-H$, so $\mu_H(E_\ell)=-1/2$.
The inclusion $E_\ell\hookrightarrow\OC_X^{\oplus3}$ also gives $H^0(X,E_\ell(-k))=0$ for every integer $k>0$.
If a rank-one subsheaf $A\subset E_\ell$ destabilized $E_\ell$, then $A^{**}\cong\OC_X(k)$ for an integer $k\geq-1/2$, so $k\geq0$.
Taking double duals would produce a nonzero map $\OC_X(k)\to E_\ell$, which is a contradiction to the vanishing.

Since $T_{\OC_X(-1)}(\OC_{\ell}(-1))\cong\IC_{\ell}(-1)[1]$ holds by $\RHom(\OC_X(-1), \OC_{\ell}(-1))\cong \CB$ and $T_{\OC_X(-2)}(\IC_{\ell}(-1))\cong E_\ell(-2)[1]$ holds by $\RHom(\OC_X(-2), \IC_{\ell}(-1)[1])\cong H^0(X, \IC_{\ell}(1))[1]$, we have
\[
(-\otimes \OC_X(2))\circ T_{\OC_X(-2)}\circ T_{\OC_X(-1)}(\OC_{\ell}(-1))  \cong E_\ell[2].
\]
If $N_{\ell/X}\cong\OC_{\PP{1}}(-1)^{\oplus2}$, $\OC_\ell(-1)$ is spherical.
Therefore, $E_\ell$ is isomorphic to the image of a spherical object by an autoequivalence, and is spherical.
\end{proof}

\begin{thm}\label{theorem: geometric disconnectedness for general quintics}
For every $b\in U(\CB)$, the locus $\Stabgeo(\XC_b)$ meets distinct connected components of $\Stab(\XC_b)$.
In particular, $\Stabgeo(\XC_b)$ is disconnected.
\end{thm}
\begin{proof}
Let $T\to U$ be the cover in \cref{remark: lines on general quintic Calabi--Yau threefolds}.
Let $\LC\subset\XC_T$ be the universal line, $f_T\colon\XC_T\to T$ be the projection,
and $\EC \coloneqq \Ker(f_T^*(f_T)_*\IC_\LC(\HC) \longrightarrow\IC_\LC(\HC))$.
The proof of \cref{lemma: evaluation bundles associated with lines}, together with cohomology and base change, shows that $(f_T)_*\IC_\LC(\HC)$ is locally free of rank three and the evaluation map is surjective.
Since $\LC\subset\XC_T$ is a regular immersion of codimension two, $\EC$ is a vector bundle whose restriction at a point corresponding to $\ell$ is $E_\ell$.
Since the lines on a geometric fiber over $U$ are pairwise disjoint and have normal bundle $\OC_{\PP{1}}(-1)^{\oplus2}$, \cref{lemma: evaluation bundles associated with lines} shows that the fiber restrictions of $\EC$ are mutually orthogonal $3$-spherical bundles.

Choose $\sigma_\bullet\in\Stab(\XC_U/U)$ as in the proof of \cref{theorem: disconnectedness for general quintics}.
Therefore, the connected component $\Sigma_b=\Stab^\dagger_{H_b}(\XC_b)$ containing $\sigma_b$, the cover $T\to U$, and the bundle $\EC$ satisfy the assumptions of \cref{theorem: disconnectedness in families}.

Fix $b\in U(\CB)$.
Put $\alpha_0=\sqrt{17/10}$ and choose $a>\alpha_0^2/6$.
By \cref{lemma: comparison of induced and double tilt stability conditions}, the stability condition $\sigma_0\coloneqq\sigma_{\alpha_0,-2}^a$ on $\XC_b$ belongs to $\Sigma_b$.
All skyscraper sheaves on $\XC_b$ are $\sigma_0$-stable of phase one by \cref{thm:li}.
To apply \cref{theorem: family geometricity} with $I=T_b$, it remains to check $\sigma_0$-stability of $E_\ell[1]$.
Since each $E_\ell$ satisfies $\nu_{\alpha_0,-2}(E_\ell)=0$ by \eqref{equation: Chern character of evaluation bundles}, \cref{proposition: stability at the wall} shows that every $E_\ell[1]$ is $\sigma_0$-stable of phase one.
This completes the proof.
\end{proof}

\begin{rem}\label{remark: number of geometric components for quintics}
A general quintic contains $2875$ lines; see \cref{remark: lines on general quintic Calabi--Yau threefolds}.
Consequently, at least $2^{2875}-1$ distinct components of $\Stab(X)$ meet $\Stabgeo(X)$.
\end{rem}

\subsection{Further Calabi--Yau threefolds}\label{subsection: further calabi-yau 3-folds}

\subsubsection{Complete intersection Calabi--Yau threefolds}
For the remaining complete intersection Calabi--Yau threefolds, we record only the disconnectedness of the stability space.

Fix a multidegree $\mathbf d=(d_1,\ldots,d_c)\in \{(2,4),(3,3),(2,2,3),(2,2,2,2)\}$, and put $n=c+3$.
Let
\[
B_{\mathbf d}\subset
\prod_{j=1}^c\PB H^0(\PP{n},\OC_{\PP{n}}(d_j))
\]
be the open subset parametrizing smooth complete intersection threefolds of multidegree $\mathbf d$.
Let $f_{\mathbf d}\colon\XC_{\mathbf d}\to B_{\mathbf d}$ be the universal family, equipped with the relatively very ample line bundle
$\HC_{\mathbf d}\coloneqq
(\OC_{\PP{n}}(1)\boxtimes\OC_{B_{\mathbf d}})|_{\XC_{\mathbf d}}$.
Write $\XC_{\mathbf d,b}$ for the fiber over $b\in B_{\mathbf d}$, and put
$H_b\coloneqq c_1(\HC_{\mathbf d}|_{\XC_{\mathbf d,b}})$.
Consider the incidence variety
\[
\IC_{\mathbf d}\coloneqq
\{(\ell,b)\in\Gr(2,n+1)\times B_{\mathbf d}
\mid\ell\subset\XC_{\mathbf d,b}\},
\]
and let $\pi_{\mathbf d}\colon\IC_{\mathbf d}\to B_{\mathbf d}$ be the second projection.
The fiber $\pi_{\mathbf d}^{-1}(b)$ is the scheme of lines on $\XC_{\mathbf d,b}$.
The variety $\IC_{\mathbf d}$ is a nonempty open subset of the fiber product of projective bundles over $\Gr(2,n+1)$ whose fiber over $[\ell]$ is
$\prod_{j=1}^c\PB H^0(\PP{n},\IC_\ell(d_j))$.
Thus $\IC_{\mathbf d}$ is irreducible.

\begin{rem}\label{remark: lines on general complete intersections of multidegree d}
The conclusions of \cref{remark: lines on general quintic Calabi--Yau threefolds} also hold for these complete intersections.
By \cite[Theorem~2.1(b)]{debarre_manivel_1998_sur_la_variete_des_espaces_lineaires_contenus_dans_une_intersection_complete}, the fiber $\pi_{\mathbf d}^{-1}(b)$ is nonempty, reduced, and zero-dimensional for general $b\in B_{\mathbf d}$.
Its length is greater than one by the degree formula in
\cite[Theorem~4.3]{debarre_manivel_1998_sur_la_variete_des_espaces_lineaires_contenus_dans_une_intersection_complete}.
Consequently, $\pi_{\mathbf d}$ is finite and \'etale of degree greater than one over a nonempty open subset of $B_{\mathbf d}$.

Distinct lines on $\XC_{\mathbf d,b}$ are disjoint for general $b\in B_{\mathbf d}$.
Indeed, consider the incidence variety of triples $(b,\ell,\ell')$ such that $\ell,\ell'$ are intersecting distinct lines contained in $\XC_{\mathbf d,b}$.
The pairs $(\ell,\ell')$ form a family of dimension $n+2(n-1)=3n-2$.
For each fixed pair, requiring the defining equations to vanish on $\ell\cup\ell'$ imposes $\sum_{j=1}^c(2d_j+1)=3n-1$ independent linear conditions.
Thus, this incidence variety has dimension at most $\dim B_{\mathbf d}-1$, so the closure of its image in $B_{\mathbf d}$ is a proper closed subset.

We may therefore choose a nonempty open subset $U_{\mathbf d}\subset B_{\mathbf d}$ such that the restriction
$\pi_{\mathbf d}|_{T_{\mathbf d}}\colon T_{\mathbf d}\to U_{\mathbf d}$, where
$T_{\mathbf d}\coloneqq\pi_{\mathbf d}^{-1}(U_{\mathbf d})$, is finite and \'etale of degree greater than one, and distinct lines on every geometric fiber over $U_{\mathbf d}$ are disjoint.
Since $T_{\mathbf d}$ is a nonempty open subset of $\IC_{\mathbf d}$, it is irreducible and hence connected.
\end{rem}

\begin{thm}\label{theorem: other complete intersection Calabi--Yau threefolds}
Let $b\in U_{\mathbf d}(\CB)$, and put $X=\XC_{\mathbf d,b}$ and $H=H_b$.
Then $\Stab(X)$ has infinitely many connected components, and every central charge occurring in $\Stab^\dagger_H(X)$ has lifts in infinitely many distinct connected components.
\end{thm}
\begin{proof}
For every geometric point $\bar b$ of $U_{\mathbf d}$ and every line $\ell\subset\XC_{\mathbf d,\bar b}$, the argument in the proof of \cref{theorem: disconnectedness for general quintics} gives
$N_{\ell/\XC_{\mathbf d,\bar b}}\cong\OC_{\PP{1}}(-1)^{\oplus2}$.
Thus, the sheaves $\OC_\ell(-1)$ on each geometric fiber are mutually orthogonal $3$-spherical objects.

Let $j\colon\LC_{\mathbf d}\hookrightarrow
\XC_{\mathbf d}\times_{B_{\mathbf d}}T_{\mathbf d}$ be the universal line, whose fiber over $(\ell,b)\in T_{\mathbf d}$ is $\ell\subset\XC_{\mathbf d,b}$.
The complex $j_*\OC_{\LC_{\mathbf d}}(-1)$ is perfect, and its fiber restrictions are the sheaves $\OC_\ell(-1)$.
The classes of $\HC_{\mathbf d}^{\otimes i}|_{\XC_{\mathbf d,\bar b}}$ for $0\leq i\leq3$ form a basis of $\Knum(\XC_{\mathbf d,\bar b})_\QB$ for every geometric point $\bar b$ of $U_{\mathbf d}$.
By \cref{thm:projective_family_stability}, choose a stability condition $\sigma_\bullet$ over $U_{\mathbf d}$ induced by the embedding $\XC_{\mathbf d}\times_{B_{\mathbf d}}U_{\mathbf d}\hookrightarrow\PP{n}\times U_{\mathbf d}$.
Thus, \cref{theorem: disconnectedness in families} applies over $U_{\mathbf d}$, with
$\Sigma_b=\Stab^\dagger_{H_b}(\XC_{\mathbf d,b})$, and proves both assertions.
\end{proof}

\subsubsection{Fermat quintic and Dwork pencil}
In \cref{subsection: rational curves on quintic threefolds}, we proved that $\Stab(X)$ has infinitely many connected components for every quintic threefold parametrized by $U$.
The open subset $U$ does not exhaust the locus with this property.
Indeed, we show below that the same conclusion holds for the smooth members of the Dwork pencil, none of which is parametrized by $U$.

\begin{dfn}\label{definition: Dwork pencil}
The \emph{Dwork pencil} is the family of quintic hypersurfaces
\begin{equation}\label{equation: Dwork pencil}
X_\psi\coloneqq
\{x_0^5+\cdots+x_4^5-5\psi x_0x_1x_2x_3x_4=0\}
\subset\PP{4},
\quad \psi\in\CB.
\end{equation}
The hypersurface $X_\psi$ is smooth if and only if $\psi^5\ne1$.
The member $X_0$ is called \emph{Fermat quintic}.
\end{dfn}

Each $X_\psi$ contains a positive-dimensional family of lines by \cite[Theorem~5.1]{book_mustata_2003_degree_1_curves_on_the_dwork_pencil_and_the_mirror_quintic}, so it is not represented by a point of $U$.
We first treat $\psi=0$, namely the Fermat quintic $X_0$.
In this case, we use rigid conics in place of lines.

\begin{lem}\label{lemma: disjoint rigid conics on the Fermat quintic}
Put $F\coloneqq x_0^5+\cdots+x_4^5$.
The Fermat quintic $X_0=\{F=0\}\subset\PP{4}$ contains smooth conics $C_1,C_2$ such that $C_1\cap C_2=\varnothing$ and $N_{C_i/X_0}\cong\OC_{\PP{1}}(-1)^{\oplus2}$ for $i=1, 2$, 
and admits an automorphism $g$ induced by a projective transformation such that $g(C_1)= C_2$.
\end{lem}
\begin{proof}
Choose $a,b\in\CB$ with $a^5=-16/3$ and $b^2=2$, and consider a plane
\[
P_a=\{x_0+x_1+ax_4=x_2+x_3+ax_4=0\}\subset\PP{4}.
\]
In the coordinates $u=x_0+ax_4/2$, $v=x_2+ax_4/2$, and $w=ax_4$ on $P_a$, the conics
\[
C_a^\pm=\{Q_\pm\coloneqq u^2+v^2\pm buv+w^2/4=0\}\subset P_a
\]
lie on $X_0$; see \cite[(3.18)--(3.19)]{walcher_2012_on_the_arithmetic_of_dbrane_superpotentials_lines_and_conics_on_the_mirror_quintic}, specialized at $\psi=0$. 
Note that they are smooth because the determinants of the quadratic forms $Q_\pm$ are $1/8$.

Put $C_1=C_a^+$ and write
$\OC_{C_1}(k)=\OC_{\PP{4}}(k)|_{C_1}$.
Since $C_1\subset\PP{4}$ is a complete intersection of type $(1,1,2)$,
the normal sequence for $C_1\subset X_0\subset\PP{4}$ takes the form
\[
0\longrightarrow N_{C_1/X_0}
\longrightarrow
\OC_{C_1}(1)^{\oplus2}\oplus\OC_{C_1}(2)
\xrightarrow{dF}
\OC_{C_1}(5)
\longrightarrow0.
\]
Taking global sections, we obtain
\[
H^0(C_1,N_{C_1/X_0})
\cong
\ker\left(
H^0(C_1,\OC_{C_1}(1))^{\oplus2}
\oplus H^0(C_1,\OC_{C_1}(2))
\longrightarrow H^0(C_1,\OC_{C_1}(5))
\right).
\]

The homogeneous coordinate ring of the plane conic
$C_1=\{Q_+=0\}\subset P_a$ is
$R=\CB[u,v,w]/(Q_+)$, with its standard grading.
For $k\geq0$, the restriction sequence on $P_a\cong\PP{2}$
identifies $R_k$ with $H^0(C_1,\OC_{C_1}(k))$.
Using the defining equations
$x_0+x_1+ax_4$, $x_2+x_3+ax_4$, and $Q_+$
to identify $N_{C_1/\PP{4}}$, the map induced by $dF$ is,
after division by $5$,
\[
R_1^{\oplus2}\oplus R_2\longrightarrow R_5,\quad
(L,M,A)\longmapsto (u+w/2)^4L+(v+w/2)^4M+2buvwA.
\]
Thus, it suffices to prove that this map is injective;
cf. \cite[Appendix~B]{katz_1986_on_the_finiteness_of_rational_curves_on_quintic_threefolds}.

Suppose that $(L,M,A)$ belongs to this kernel.
Reduction modulo $u$ and $v$ gives
$L-4M=0$ in $R/(u)$ and $M-4L=0$ in $R/(v)$, respectively.
Consequently,
\[
L=\lambda u+4\mu v,\qquad M=4\lambda u+\mu v
\]
for some $\lambda,\mu\in\CB$.
Since $u^4+v^4=0$ in $R/(w)$, reduction modulo $w$ gives
\[
3v^4(\lambda u-\mu v)=0
\quad\text{in }R/(w)=\CB[u,v]/(u^2+buv+v^2).
\]
As $v$ is a non-zero-divisor in this ring and its defining ideal
contains no nonzero linear form, we obtain $\lambda=\mu=0$.
Thus, $L=M=0$, and $2buvwA=0$ in the domain $R$ gives $A=0$.
Therefore, $H^0(C_1,N_{C_1/X_0})=0$.
Together with $\deg N_{C_1/X_0}=-2$, this proves
$N_{C_1/X_0}\cong\OC_{\PP{1}}(-1)^{\oplus2}$.

Choose a fifth root of unity $\zeta\ne1$ and define
\[
g(x_0:x_1:x_2:x_3:x_4)
=(x_1:x_0:x_2:x_3:\zeta x_4).
\]
Put $C_2\coloneqq g(C_1)=C_{a/\zeta}^-$.
This automorphism of $X_0$ sends $C_a^+$ to $C_2$.

Suppose that a point lies in $C_a^+\cap C_2$.
Comparing the defining equations of $P_a$ and $P_{a/\zeta}$ gives $x_4=0$.
The two conic equations then give
$u^2+v^2+buv=u^2+v^2-buv=0$.
Hence $u=v=0$, so $x_0=\cdots=x_4=0$.
This is a contradiction.
\end{proof}

Before proving the next theorem, we record the following special case of \cite[Corollary~3.5.2]{polishchuk_2007_constant_families_of_tstructures_on_derived_categories_of_coherent_sheaves} and \cite[Theorem 2.23 (2)]{preprint_chunyi_emanuele_alexander_paolo_xiaolei_2025_stability_conditions_on_products_of_curves_and_hilbert_schemes_of_surfaces}.
By \cref{theorem: Bridgeland's deformation theorem}, each connected component of $\Stab(\PP{N})$ is locally parametrized by $\Hom(\Knum(\PP{N}),\CB)$.
Hence the corollary applies to the action of the connected group $\PGL_{N+1}(\CB)$.

\begin{lem}\label{lemma: projective automorphisms and stability conditions}
For $N\geq1$, $g\in\PGL_{N+1}(\CB)$, and $\tau\in\Stab(\PP{N})$, we have $(g^{-1})^*\tau=\tau$.
\end{lem}

\begin{thm}\label{theorem: disconnectedness for the Dwork pencil}
Let $X=X_\psi$ be a smooth member of the Dwork pencil \labelcref{equation: Dwork pencil}, and put $H\coloneqq c_1(\OC_X(1))$.
Then $\Stab(X)$ has infinitely many connected components, and every central charge occurring in $\Stab^\dagger_H(X)$ has lifts in infinitely many distinct connected components.
\end{thm}
\begin{proof}
If $\psi=0$, take $C_1,C_2$, and $g$ as in
\cref{lemma: disjoint rigid conics on the Fermat quintic},
and put $F_i=\OC_{C_i}(-H)$.
If $\psi\ne0$, choose a fifth root of unity $\zeta\ne1$ and let
\[
\begin{aligned}
C_1&=\{(u:-u:v:-v:0)\mid(u:v)\in\PP{1}\},\\
C_2&=\{(u:-\zeta u:v:-\zeta v:0)\mid(u:v)\in\PP{1}\}.
\end{aligned}
\]
These lines in $X$ are disjoint.
The projective transformation $g=\operatorname{diag}(1,\zeta,1,\zeta,\zeta^{-2})$ preserves $X$ and sends $C_1$ to $C_2$.
By \cite[Lemma~1.2]{book_mustata_2003_degree_1_curves_on_the_dwork_pencil_and_the_mirror_quintic}, both normal bundles are $\OC_{\PP{1}}(-1)^{\oplus2}$.
Therefore,  $F_i=\OC_{C_i}(-H)$ for $i=1,2$ are orthogonal spherical objects.

In both cases, the autoequivalence $(g^{-1})^*$ sends $F_1$ to $F_2$.
Since $g$ preserves $H$, $(g^{-1})^*$ acts trivially on $\Knum(X)$.

Let $h\colon X\hookrightarrow\PP{4}$ be the inclusion.
By \cref{thm:projective_family_stability}, we can choose $\tau\in\Stab(\PP{4})$ such that $\sigma_0\coloneqq h^\sharp\tau\in\Stab^\dagger_H(X)$.
By \cref{lemma: projective automorphisms and stability conditions}, $\PGL_{5}(\CB)$ fixes $\tau$.
As $g$ is induced by a projective transformation, the equivariance of $h$ and the definition of the induced stability condition give $(g^{-1})^*\sigma_0=\sigma_0$.
Thus, \cref{corollary: disconnectedness from a semilinear equivalence} applies with $\Sigma=\Stab^\dagger_H(X)$ and proves both assertions.
\end{proof}

\subsection{Examples in any dimension}\label{subsection: products and finite Albanese morphisms}

Let $Y$ and $Z$ be smooth projective varieties over $\CB$, and assume that the Albanese morphism of $Z$ is finite.
Put $X=Y\times Z$, and write $p\colon X\to Y$ and $q\colon X\to Z$ for the projections.
For $z\in Z(\CB)$, let $i_z\colon Y\hookrightarrow X$ be the inclusion of the fiber over $z$.
We use the notation for induced slicings and stability conditions from \cref{subsection:Stability conditions over a base}.

\begin{lem}\label{lemma: product restriction}
For $z\in Z(\CB)$, the map
\[
R_z\colon\Stab(X)\longrightarrow\Stab(Y),
\quad \tau\longmapsto i_z^\sharp\tau
\]
is well-defined and continuous.
For $\tau\in\Stab(X)$, $0\ne F\in D^b(Y)$, and $\phi\in\RB$, the object $F$ is $R_z(\tau)$-stable of phase $\phi$ if and only if $(i_z)_*F$ is $\tau$-stable of phase $\phi$.
\end{lem}
\begin{proof}
Let $z\in Z(\CB)$ and $\tau=(Z_\tau,\PC_\tau)$.
Since $p_*(i_z)_*=\id$, we identify $\Knum(Y)$ with the image of $(i_z)_*$ in $\Knum(X)$.
Therefore, we have 
\[
i_z^\sharp\tau=(Z_\tau\circ(i_z)_*,\PC_z)\in\Stab(Y),
\]
where $\PC_z(\phi)=\{F\in D^b(Y)\mid(i_z)_*F\in\PC_\tau(\phi)\}$ by the identification and \cite[Theorem~4.22]{preprint_yiran_2025_bridgeland_stability_conditions_on_some_higherdimensional_calabiyau_manifolds_and_generalized_kummer_varieties}. 
Let $F\in D^b(Y)$ be a nonzero object and $F_1,\ldots,F_l$ be the $R_z(\tau)$-Harder--Narasimhan factors of $F$ with phases $\phi_1>\cdots>\phi_l$.
By the definition of $\PC_z$, each $(i_z)_*F_j$ is $\tau$-semistable of phase $\phi_j$.
Thus, pushing forward the filtration gives the $\tau$-Harder--Narasimhan filtration of $(i_z)_*F$.
Since $Z_{R_z(\tau)}(F_j)=Z_\tau((i_z)_*F_j)$, we obtain
\[
\phi^\pm_{R_z(\tau)}(F)=\phi^\pm_\tau((i_z)_*F), \text{ and }
m_{R_z(\tau)}(F)=m_\tau((i_z)_*F).\]
Applying these equalities to $\tau, \tau'\in\Stab(X)$, each term in the definition of $d(R_z(\tau),R_z(\tau'))$ is bounded by $d(\tau,\tau')$.
Taking the supremum over $0\ne F\in D^b(Y)$ gives
\[
d(R_z(\tau),R_z(\tau'))\leq d(\tau,\tau').
\]
Hence $R_z$ is continuous.

For every $\phi\in\RB$, the functor $(i_z)_*\colon\PC_z(\phi)\to\PC_\tau(\phi)$ is exact and faithful.
Consequently, if $(i_z)_*F$ is $\tau$-stable of phase $\phi$, then $F$ is $R_z(\tau)$-stable of phase $\phi$.

Conversely, suppose that $F$ is $R_z(\tau)$-stable of phase $\phi$.
Choose a stable subobject $A\subset(i_z)_*F$ in $\PC_\tau(\phi)$.
First, we show that $q(\Supp(A))$ is finite.
By \cite[Lemma~3.3]{preprint_chunyi_emanuele_alexander_paolo_xiaolei_2025_stability_conditions_on_products_of_curves_and_hilbert_schemes_of_surfaces} applied to the composition $X\xrightarrow{q}Z\to\Alb(Z)$ and the finiteness of the Albanese morphism of $Z$, it is enough to show that $A\otimes q^*L\cong A$ for all $L\in\Pic^0(Z)$.

For $L\in\Pic^0(Z)$, tensoring by $q^*L$ fixes $\tau$ by \cite[Theorem~2.23(2)]{preprint_chunyi_emanuele_alexander_paolo_xiaolei_2025_stability_conditions_on_products_of_curves_and_hilbert_schemes_of_surfaces} and fixes the isomorphism class of $(i_z)_*F$ by the projection formula.
Therefore, $\Pic^0(Z)(\CB)$ permutes the finitely many isomorphism classes of the Jordan--H\"older factors of $(i_z)_*F$.
Since $\Pic^0(Z)(\CB)$ is divisible, the permutation action is trivial.
Indeed, if there are $r$ isomorphism classes of Jordan--H\"older factors, we can write $L\cong M^{\otimes r!}$ with $M\in\Pic^0(Z)$, and the permutation induced by $M$ has order dividing $r!$.
In particular, $A\otimes q^*L\cong A$ for all $L\in\Pic^0(Z)$.
Thus, $q(\Supp(A))$ is finite.
Since $A$ is indecomposable, $\Supp(A)$ is contained in a single fiber of $q$.
The nonzero morphism $A\to(i_z)_*F$ gives $\Supp(A)\subset Y\times\{z\}$.

Let $t_1,\ldots,t_{\dim Z}$ be local parameters at $z$.
On the formal completion of $X$ along $Y\times\{z\}$, multiplication by $q^*t_j$ on $A$ is $\lambda_j\id_A$ for some $\lambda_j\in\CB$, since $\End(A)=\CB$.
Since $\Supp(A)\subset Y\times\{z\}$, multiplication by $q^*t_j$ is nilpotent on each cohomology sheaf of $A$.
As $A\ne0$, we obtain $\lambda_j=0$ for every $j$.
By \cite[Lemmas~4.16 and~4.24]{preprint_yiran_2025_bridgeland_stability_conditions_on_some_higherdimensional_calabiyau_manifolds_and_generalized_kummer_varieties}, there exists $G\in D^b(Y)$ such that $A\cong(i_z)_*G$.
Since $(i_z)_*G\cong A\in\PC_\tau(\phi)$, we have $G\in\PC_z(\phi)$.
As $p\circ i_z=\id_Y$ and $G,F\in\PC_z(\phi)$, \cite[Lemma~4.21(ii)]{preprint_yiran_2025_bridgeland_stability_conditions_on_some_higherdimensional_calabiyau_manifolds_and_generalized_kummer_varieties} implies that the inclusion $A\hookrightarrow(i_z)_*F$ lifts to a morphism $f\colon G\to F$.
The exactness and faithfulness of $(i_z)_*$ imply that $f$ is a nonzero monomorphism in $\PC_z(\phi)$.
Since $F$ is $R_z(\tau)$-stable, $f$ is an isomorphism.
Thus, $(i_z)_*F\cong A$ is $\tau$-stable of phase $\phi$.
\end{proof}

Let $\Phi\in\Aut D^b(Y)$, and $K\in D^b(Y\times Y)$ be its Fourier--Mukai kernel.
Under the identification $X\times X\cong(Y\times Y)\times(Z\times Z)$, the kernel $K\boxtimes\OC_{\Delta_Z}$ defines an autoequivalence $\widetilde\Phi$ of $D^b(X)$.
Base change and the definition of $R_z$ give
\begin{equation}\label{equation: product autoequivalence restriction}
\widetilde\Phi\circ(i_z)_*\cong(i_z)_*\circ\Phi, \text{ and } R_z(\widetilde\Phi\tau)=\Phi R_z(\tau).
\end{equation}
Note that the maps $R_z$ also commute with the $\GLt$-action by the definition of the induced slicing.

\begin{thm}\label{theorem: product criterion for distinguished components}
Let $L_Y$ and $L_Z$ be very ample line bundles on $Y$ and $Z$, respectively, and put
$L=L_Y\boxtimes L_Z$, $\Sigma_Y=\Stab^\dagger_{L_Y}(Y)$, and $\Sigma_X=\Stab^\dagger_L(X)$.
The map
\[
R\colon\Sigma_X\longrightarrow\Sigma_Y,
\quad \tau\longmapsto R_z(\tau)
\]
is well-defined, continuous, and open.
Moreover, the following hold:
\begin{enumerate}
\item\label{item: distinguished product stability components}
Let $\{\Phi_i\}_{i\in I}\subset\Aut D^b(Y)$.
If the connected components $\Phi_i\Sigma_Y$ are pairwise distinct, then the connected components $\widetilde\Phi_i\Sigma_X$ are pairwise distinct.
If, in addition, every $\widetilde\Phi_i$ acts trivially on $\Knum(X)$, then every central charge occurring in $\Sigma_X$ has a lift in each $\widetilde\Phi_i\Sigma_X$.
\item\label{item: distinguished product geometric components}
Suppose that $\sigma_0\in R(\Sigma_X)$, $\Phi\in\Aut D^b(Y)$, and $S\in D^b(Y)$ satisfy the assumptions of \cref{setup: geometric wall} on $Y$.
Then $\Sigma_X$ and $\widetilde\Phi^{-1}\Sigma_X$ are distinct connected components of $\Stab(X)$, and both contain geometric stability conditions.
\end{enumerate}
\end{thm}
\begin{proof}
Let $h\colon X\hookrightarrow\PP{N}$ be the Segre embedding determined by $L_Y$ and $L_Z$, and put $j_z=h\circ i_z$.
Fix $z_0\in Z(\CB)$.
For sufficiently large $\xi>0$, \cref{thm:projective_family_stability} gives
\[
\tau_0\coloneqq h^\sharp\tau^{\PP{N}}_{\xi,0}\in\Sigma_X
\text{ and }
R_{z_0}(\tau_0)=j_{z_0}^\sharp\tau^{\PP{N}}_{\xi,0}\in\Sigma_Y,
\]
where the second inclusion uses $j_{z_0}^*\OC_{\PP{N}}(1)\cong L_Y$.
By the construction of the Segre embedding, for each $z\in Z(\CB)$, 
there exists $g\in\PGL_{N+1}(\CB)$ such that $g\circ j_{z_0}=j_z$.
By \cref{lemma: projective automorphisms and stability conditions}, $g$ fixes $\tau^{\PP{N}}_{\xi,0}$.
Hence $R_z(\tau_0)=R_{z_0}(\tau_0)$.
Since $\Sigma_X$ is connected and $R_z$ is continuous by \cref{lemma: product restriction}, we obtain $R_z(\Sigma_X)\subset\Sigma_Y$ for every $z$.

Since $\chi_X(E,(i_z)_*F)=\chi_Y(i_z^*E,F)$ is independent of $z$ for all $E\in D^b(X)$ and $F\in D^b(Y)$, both the homomorphism $(i_z)_*\colon\Knum(Y)\to\Knum(X)$ and the central charge $Z_\tau\circ(i_z)_*$ are independent of $z$ for every $\tau\in\Sigma_X$.
Fix $z\in Z(\CB)$, and put
\[
\Sigma_z\coloneqq
\{\tau\in\Sigma_X\mid R_z(\tau)=R_{z_0}(\tau)\}.
\]
Since $R_z$ and $R_{z_0}$ are continuous and $\Stab(Y)$ is Hausdorff, $\Sigma_z$ is closed in $\Sigma_X$.

Let $\tau\in\Sigma_z$.
By \cref{theorem: Bridgeland's deformation theorem}, we may take an open neighborhood $V$ of $R_z(\tau)$ on which the central charge map is injective.
Since $R_z$ and $R_{z_0}$ have the same central charges, the open neighborhood $R_z^{-1}(V)\cap R_{z_0}^{-1}(V)$ of $\tau$ is contained in $\Sigma_z$.
Thus, $\Sigma_z$ is open in $\Sigma_X$.
As $\tau_0\in\Sigma_z$ and $\Sigma_X$ is connected, we obtain $\Sigma_z=\Sigma_X$.
We denote the common continuous map $R_z$ by $R$.

In the local coordinates of \cref{theorem: Bridgeland's deformation theorem}, the map $R$ is given by
\[
\Hom(\Knum(X),\CB)\longrightarrow\Hom(\Knum(Y),\CB),
\quad W\longmapsto W\circ(i_z)_*.
\]
The identity $p_*(i_z)_*=\id$ gives a right inverse $W_Y\mapsto W_Y\circ p_*$.
Hence this linear map is surjective, and $R$ is open.

\labelcref{item: distinguished product stability components}:
By \eqref{equation: product autoequivalence restriction}, we have
$R_z(\widetilde\Phi_i\Sigma_X)\subset\Phi_i\Sigma_Y$.
Since the components $\Phi_i\Sigma_Y$ are pairwise disjoint, the components $\widetilde\Phi_i\Sigma_X$ are also pairwise disjoint.
If every $\widetilde\Phi_i$ acts trivially on $\Knum(X)$, then for $\tau=(W,\PC)\in\Sigma_X$ we have
$\widetilde\Phi_i\tau=(W,\widetilde\Phi_i\PC)\in\widetilde\Phi_i\Sigma_X$.

\labelcref{item: distinguished product geometric components}:
Since $R$ is open, we may choose the neighborhood $U\subset R(\Sigma_X)$ of $\sigma_0$ in \cref{proposition: extension stability}.
By \cref{theorem: Bridgeland's deformation theorem} and the linear independence of $[\OC_y]$ and $[S]$, there exists $\sigma\in U$ such that
$\Delta(\sigma)>0$, so $\sigma$ and $\Phi^{-1}\sigma$ are geometric by \cref{proposition: extension stability}.

Take $\tau\in\Sigma_X$ with $R(\tau)=\sigma$.
For every $(y,z)\in X$, we have
$\OC_{(y,z)}\cong(i_z)_*\OC_y$, and $\widetilde\Phi(\OC_{(y,z)})\cong(i_z)_*\Phi(\OC_y)$.
Since $R_z(\tau)=\sigma$ for every $z$, \cref{lemma: product restriction} shows that all $\OC_{(y,z)}$ are $\tau$-stable of the same phase, and all $\widetilde\Phi(\OC_{(y,z)})$ are $\tau$-stable of the same phase.
Thus, $\tau$ and $\widetilde\Phi^{-1}\tau$ are geometric.
\end{proof}

We apply \cref{theorem: product criterion for distinguished components} to quintic threefolds parametrized by $U$, and show that the disconnected phenomenon of $\Stab(X)$ and $\Stabgeo(X)$ occurs in any dimension greater than or equal to $3$.

\begin{cor}\label{corollary: disconnectedness of products}
Let $Y\subset\PP{4}$ be a quintic parametrized by $U$ in \cref{remark: lines on general quintic Calabi--Yau threefolds}, and $Z$ be a smooth projective variety with finite Albanese morphism.
Put $X=Y\times Z$ and $L=\OC_Y(1)\boxtimes L_Z$, where $L_Z$ is a very ample line bundle on $Z$.
Then the following hold:
\begin{enumerate}
\item\label{item: products of quintics stability components}
$\Stab(X)$ has infinitely many connected components, and every central charge occurring in $\Stab^\dagger_L(X)$ has lifts in infinitely many distinct connected components.
\item\label{item: products of quintics geometric components}
$\Stabgeo(X)$ meets at least $2^{2875}-1$ distinct connected components of $\Stab(X)$.
\end{enumerate}
\end{cor}

\begin{proof}
Put $\Sigma_X\coloneqq \Stab^\dagger_L(X)$.

\labelcref{item: products of quintics stability components}:
Take distinct $F_1,F_2$ from the family in the proof of \cref{theorem: disconnectedness for general quintics}, and put $\Psi=T_{F_1}T_{F_2}^{-1}$.
Since $[F_1]=[F_2]$ in $\Knum(Y)$, the Fourier--Mukai kernels of $\widetilde{T_{F_1}} = T_{F_1}\boxtimes\id_{D^b(Z)}$ and $\widetilde{T_{F_2}}= T_{F_2}\boxtimes\id_{D^b(Z)}$ are numerically equivalent, so $\widetilde\Psi$ acts trivially on $\Knum(X)$.
Apply \labelcref{item: distinguished product stability components} of \cref{theorem: product criterion for distinguished components} to the powers of $\Psi$.

\labelcref{item: products of quintics geometric components}:
Let $E_\ell$ be an evaluation bundle associated with a line $\ell$ on $Y$ defined in \cref{lemma: evaluation bundles associated with lines}, and let $h\colon X\hookrightarrow\PP{N}$ be the Segre embedding determined by $L$.
Apply \cref{thm:projective_family_stability} to $h$ and \cref{lemma: comparison of induced and double tilt stability conditions} to $h\circ i_z$, taking the parameters so that $\beta\ll0$ and $\alpha^2=\beta^2+\beta-3/10$.
Since $R$ commutes with the $\GLt$-action, we obtain $\sigma_0\coloneqq\sigma_{\alpha,\beta}^a\in R(\Sigma_X)$.
By \cref{proposition: stability at the wall} and the proof of \cref{theorem: geometric disconnectedness for general quintics}, the data $(\sigma_0,T_{E_{\ell}},E_\ell[1])$ satisfy \cref{setup: geometric wall} for every line $\ell$ on $Y$.

Let $I$ be the set of the $2875$ lines on $Y$.
For $J\subseteq I$, put $\Phi_J\coloneqq\prod_{\ell\in J}T_{E_\ell}$, with $\Phi_\varnothing=\id$.
Since $R(\Sigma_X)$ is open, the proof of \cref{theorem: family geometricity} applied to $\sigma_0$ gives $\sigma\in R(\Sigma_X)$ such that $\Phi_J^{-1}\sigma$ is geometric for every $J\subseteq I$.
Take $\tau\in\Sigma_X$ with $R(\tau)=\sigma$.
By \cref{lemma: product restriction} and \eqref{equation: product autoequivalence restriction}, the stability condition $\widetilde\Phi_J^{-1}\tau$ is geometric for every $J\subseteq I$.
By \cref{theorem: disconnectedness in families} and \labelcref{item: distinguished product stability components} of \cref{theorem: product criterion for distinguished components}, the components $\widetilde\Phi_J^{-1}\Sigma_X$ for $J\subsetneq I$ are pairwise distinct.
There are $2^{2875}-1$ proper subsets of $I$, proving the claim.
\end{proof}

The assertions for the remaining products in
\labelcref{item in intro: products of quintics} of
\cref{theorem in intro: disconnectedness examples}
follow from
\labelcref{item: distinguished product stability components} of
\cref{theorem: product criterion for distinguished components}
by the same argument as in the proof of
\labelcref{item: products of quintics stability components} of
\cref{corollary: disconnectedness of products}.

\section{Examples from spinor bundles}\label{section: examples from Spinor bundles}

Let $Q\subset \PP{n+1}$ be a smooth quadric hypersurface.
Recall that, by Kapranov \cite{kapranov_1986_derived_category_of_coherent_bundles_on_a_quadric}, if $Q$ is even-dimensional $D^b(Q)$ admits the following full exceptional collection:
\[
D^b(Q) = \langle \Spbdl_+, \Spbdl_-, \OO, \OO(1), \ldots, \OO(n-1) \rangle, 
\]
where $\Spbdl_{\pm}$ are the spinor bundles on $Q$. 
Moreover, the exceptional pair $\{\Spbdl_+, \Spbdl_-\}$ is orthogonal. 
We prepare a lemma for constructing orthogonal spherical objects.

\begin{lem}\label{lemma: exceptional spherical constructions}
Let $X$ be a smooth projective variety of $\dim X\geq2$ and $i\colon Z\hookrightarrow X$ be a smooth divisor.
\begin{enumerate}
\item\label{item: anticanonical restriction}
If $Z\in|-K_X|$ and $E_1,E_2$ are orthogonal exceptional bundles on $X$, then $i^*E_1,i^*E_2$ are orthogonal spherical bundles on $Z$.
\item\label{item: canonical normal pushforward}
If $\omega_X|_Z\cong\OC_Z$ and $F_1,F_2$ are orthogonal exceptional bundles on $Z$, then $i_*F_1,i_*F_2$ are orthogonal spherical objects on $X$.
\end{enumerate}
\end{lem}
\begin{proof}
The sphericality assertions are \cite[Propositions~3.13 and 3.15]{seidel_thomas_2001_braid_group_actions_on_derived_categories_of_coherent_sheaves}.
We prove the orthogonality assertions.

(\ref{item: anticanonical restriction}): Applying $\RHom_X(E_1,{-})$ to the triangle
\[
E_2\otimes \omega_X\longrightarrow E_2\longrightarrow i_*i^*E_2, 
\]
the orthogonality and Serre duality give
$\RHom_Z(i^*E_1,i^*E_2)\cong \RHom_X(E_1,i_*i^*E_2) =0.$

(\ref{item: canonical normal pushforward}): The adjunction formula and $\omega_X|_Z\cong\OC_Z$ give $N_{Z/X}\cong \omega_Z$.
Consider the local-to-global Ext spectral sequence
\[
E_2^{p,q}
=H^p\bigl(X,\mathcal{E}xt_X^q(i_*F_1,i_*F_2)\bigr)
\Longrightarrow
\Ext_X^{p+q}(i_*F_1,i_*F_2).
\]
Since $Z$ is a smooth divisor and $F_1,F_2$ are locally free, a local resolution gives
\[
\mathcal{E}xt_X^q(i_*F_1,i_*F_2)
\cong
\begin{cases}
i_*(F_1^\vee\otimes F_2)&q=0,\\
i_*(F_1^\vee\otimes F_2\otimes N_{Z/X})&q=1,\\
0&q\geq2,
\end{cases}
\]
and it shows $E_2^{p,q} = 0$ if $q\neq 0, 1$.
On the other hand, orthogonality and Serre duality give 
$E_2^{p,0}=\Ext_Z^p(F_1,F_2)=0$ and 
$E_2^{p,1}=\Ext_Z^p(F_1,F_2\otimes\omega_Z)
\cong\Ext_Z^{\dim Z-p}(F_2,F_1)^*=0$.
Thus, the $E_2$-page vanishes, and hence $\RHom_X(i_*F_1,i_*F_2)=0$.
\end{proof}

\begin{lem}\label{lemma: connected double cover for Spinor bundles}
Let $n\geq2$ be an even integer and $B\subset\mgrpsch$ a nonempty open subset.
Consider the family
\begin{equation}\label{equation: quadric one parameter family}
\QC=\{x_0^2+\cdots+x_n^2+t x_{n+1}^2=0\}
\subset\PB^{n+1}\times B
\end{equation}
and the double cover of $B$
\begin{equation}\label{equation: connected double cover for Spinor bundles}
\pi\colon T=\{s\in\mgrpsch\mid s^2\in B\}\longrightarrow B,
\quad s\longmapsto s^2.
\end{equation}
Then, there exists a vector bundle $\SC$ on $\QC_T$ such that, for every geometric point $\bar b$ of $B$ with $T_{\bar b}=\{s,-s\}$, the restrictions $\SC|_s,\SC|_{-s}$ are the two spinor bundles on $Q_{\bar b}$.
\end{lem}
\begin{proof}
Let $Q\coloneqq\{x_0^2+\cdots+x_{n+1}^2=0\}\subset\PB^{n+1}$, and let $\Spbdl_+,\Spbdl_-$ be the two spinor bundles on $Q$.
The relation $t=s^2$ gives an isomorphism
\[
\varphi\colon\QC_T\longrightarrow Q\times T,\quad
((x_0:\cdots:x_{n+1}),s)
\longmapsto((x_0:\cdots:x_n:sx_{n+1}),s).
\]
Let $\pi_Q\colon Q\times T\to Q$ be the projection onto the first factor, and set $\SC=\varphi^*\pi_Q^*\Spbdl_+$.

Consider the reflection $r$ of $Q$ given by
$(x_0:\cdots:x_n:x_{n+1})\longmapsto(x_0:\cdots:x_n:-x_{n+1})$.
Since $r$ exchanges the two families of maximal linear subspaces of $Q$, $r^*\Spbdl_+\cong\Spbdl_-$
(cf. \cite[Definition~1.3]{ottaviani_1988_spinor_bundles_on_quadrics}).
For a geometric point $\bar b$ with $T_{\bar b}=\{s,-s\}$, let $\varphi_s,\varphi_{-s}$ be the isomorphisms on the fibers induced by $\varphi$.
We have $\varphi_{-s}=r\circ\varphi_s$ by definition and hence
\[
\SC|_s\cong\varphi_s^*\Spbdl_+, \text{ and }
\SC|_{-s}\cong\varphi_{-s}^*\Spbdl_+
\cong\varphi_s^*\Spbdl_-.
\]
These are the two spinor bundles on $Q_{\bar b}$.
\end{proof}

The first examples in this setting are Calabi--Yau complete intersections.

\begin{prop}\label{proposition: disconnectedness for anticanonical divisors in quadrics}
Let $n\geq3$ be an odd integer and $Y\subset\PP{n+2}$ be a smooth complete intersection of type $(2,n+1)$ whose defining quadric is smooth.
Put $H=\OC_Y(1)$.
Then, $\Stabnum(Y)$ has infinitely many connected components, and every central charge occurring in $\Stab^\dagger_H(Y)$ has lifts in infinitely many distinct components.
\end{prop}
\begin{proof}
Let $\QC\to\mgrpsch$ be the family of $(n+1)$-dimensional quadrics defined in \cref{lemma: connected double cover for Spinor bundles}, with fiber $Q_b$ over $b\in\mgrpsch$.
After a change of coordinates, we may write $Y=Q_1\cap\{f=0\}$ for a homogeneous polynomial $f$ of degree $n+1$.
Put $Y_b=Q_b\cap\{f=0\}$ and let $B\subset\mgrpsch$ be the open locus where $Y_b$ is smooth.
Restrict $\QC$ to $B$ and set $\YC=\QC\cap(\{f=0\}\times B)$.
Then $\YC\to B$ is a smooth projective family with $Y_1=Y$.

Let $T\to B$ be the connected finite \'etale double cover and $\SC$ the vector bundle on $\QC_T$ in \cref{lemma: connected double cover for Spinor bundles}.
Let $k_T\colon\YC_T\hookrightarrow\QC_T$ be the inclusion, and put $\EC\coloneqq k_T^*\SC$.
For a geometric point $\bar b$ of $B$, write $T_{\bar b}=\{s,-s\}$.
The bundles $\SC|_s$ and $\SC|_{-s}$ are the two spinor bundles on $Q_{\bar b}$.
Since $Y_{\bar b}$ is an anticanonical divisor in $Q_{\bar b}$, \labelcref{item: anticanonical restriction} of \cref{lemma: exceptional spherical constructions} shows that $\EC_s$ and $\EC_{-s}$ are orthogonal $n$-spherical bundles.
Since the variety $Y_{\bar b}$ is a smooth complete intersection of odd dimension, the classes of $\OC_\YC(i)|_{Y_{\bar b}}$ for $0\leq i\leq n$ form a basis of $\Knum(Y_{\bar b})_\QB$.
By \cref{thm:projective_family_stability}, choose $\sigma_\bullet\in\Stab(\YC/B)$ induced by the embedding $\YC\hookrightarrow\PP{n+2}\times B$.
Therefore, these data satisfy \cref{setup: disconnectedness in families}, with $\Sigma_1=\Stab^\dagger_H(Y)$.
Applying \cref{theorem: disconnectedness in families} to $\YC\to B$, the double cover $T\to B$, and the vector bundle $\EC$ proves both assertions.
\end{proof}

\subsection{Rational weak Fano varieties}\label{subsection: rational weak Fano varieties}
The same spinor family gives examples without a Calabi--Yau hypothesis.

\begin{setup}\label{setup: example rational weak fano}
Let $n\geq3$ be an odd integer and $C\subset\PP{n}$ be a smooth complete intersection of type $(2,n+1)$ whose defining quadric $Q$ is smooth.
Let $\pi\colon X=\Bl_C\PP{n}\to\PP{n}$ be the blow-up, write $\iota\colon R\hookrightarrow X$ and $p\colon R\to C$ for the exceptional divisor and its projection, and put $H=\pi^*\OC_{\PP{n}}(1)$.
Write $j\colon D\hookrightarrow X$ for the strict transform of $Q$.
Since $C$ is a Cartier divisor on $Q$, the restriction $\pi|_D\colon D\to Q$ is an isomorphism.
Let $\Spbdl_+,\Spbdl_-$ be the two spinor bundles on $D\cong Q$, and set $F_\pm=j_*\Spbdl_\pm$.

Put $E_a=\OC_X(aH)$ for $a\in\ZZ$ and
$\BC=\iota_*\bigl(p^*D^b(C)\otimes\iota^*\OC_X(R)\bigr)\subset D^b(X)$.
By Orlov's blow-up formula
\cite[Theorem~4.3]{orlov_1992_projective_bundles_monoidal_transformations_and_derived_categories_of_coherent_sheaves}, $D^b(X)$ admits a semiorthogonal decomposition $\DB_0\coloneqq\langle\BC,E_0,\ldots,E_n\rangle$.
As shown in \cref{lemma: BlCPn is rational weak Fano}, the objects $F_+,F_-$ are orthogonal $n$-spherical objects.
Put $\Psi=T_{F_+}T_{F_-}^{-1}$ and $\DB_m=\Psi^m(\DB_0)$ for $m\in\ZZ$.
\end{setup}

\begin{lem}\label{lemma: BlCPn is rational weak Fano}
In \cref{setup: example rational weak fano}, $X$ is weak Fano and the objects $F_+$ and $F_-$ are orthogonal $n$-spherical objects.
\end{lem}
\begin{proof}
The blow-up formula gives $-K_X=(n+1)H-R = D+(n-1)H$.
Since $\IC_C(n+1)$ is globally generated, the natural surjection $\pi^*\IC_C\to\OC_X(-R)$ shows that $\OC_X(-K_X)$ is globally generated. In particular, $-K_X$ is nef.
Since $H$ is big and $D$ is effective, $-K_X$ is big.
Thus, $X$ is weak Fano.

To prove the latter, note that $R|_D$ corresponds to $C$ under $D\cong Q$.
Thus we have $\omega_X|_D \cong\OC_Q(-n-1)\otimes\OC_Q(C) \cong\OC_D$.
Since $\Spbdl_+,\Spbdl_-$ are orthogonal exceptional bundles, \labelcref{item: canonical normal pushforward} of \cref{lemma: exceptional spherical constructions} shows that $F_+,F_-$ are orthogonal $n$-spherical objects.
\end{proof}
\begin{prop}\label{proposition: disconnectedness for weak Fano blowups}
In \cref{setup: example rational weak fano}, the following hold:
\begin{enumerate}
\item\label{item: weak Fano stability components}
$\Stab(X)$ has infinitely many connected components.
\item\label{item: weak Fano central charges}
For any ample line bundle $L$ on $X$, every central charge occurring in $\Stab^\dagger_L(X)$ has lifts in infinitely many distinct connected components of $\Stab(X)$.
\item\label{item: weak Fano braid orbits}
The semiorthogonal decompositions $\DB_m$ belong to pairwise distinct $B_{n+2}$-orbits, and corresponding components have images in $\Knum(X)$ independent of $m$.
\end{enumerate}
\end{prop}
\begin{proof}
\labelcref{item: weak Fano stability components} and \labelcref{item: weak Fano central charges}:
Let $\QC\to\mgrpsch$ be the family of $(n-1)$-dimensional quadrics defined in \cref{lemma: connected double cover for Spinor bundles}.
After a change of coordinates, we may assume that $C=Q_1\cap\{f=0\}$ for a homogeneous polynomial $f$ of degree $n+1$.
Put $C_b=Q_b\cap\{f=0\}$ and let $B\subset\mgrpsch$ be the open locus where $C_b$ is a smooth complete intersection.
We continue to write $\QC$ for its restriction to $B$ and set $\CC=\QC\cap(\{f=0\}\times B)$.
Let
\[
\pi_B\colon\XC=\Bl_\CC(\PP{n}\times B)
\longrightarrow\PP{n}\times B
\]
be the blow-up, with exceptional divisor $\iota_B\colon\RC\hookrightarrow\XC$ and projection $p_B\colon\RC\to\CC$.
The induced morphism $\XC\to B$ is a smooth projective family with $\XC_1=X$.
Write $j_B\colon\DC\hookrightarrow\XC$ for the strict transform of $\QC$.
Since $\CC$ is a Cartier divisor on $\QC$, we have $\DC\cong\QC$ over $B$.

Fix an ample line bundle $L$ on $X$.
Since $\Pic(X)=\ZZ H\oplus\ZZ R$, the line bundle $L$ extends to a line bundle $\LC$ on $\XC$ using $\pi_B^*\OC_{\PP{n}\times B}(1)$ and $\OC_\XC(\RC)$.
After shrinking $B$ around $1$, we may assume that $\LC$ is relatively ample.

Let $T\to B$ and $\SC$ be the double cover and vector bundle on $\QC_T$ given by \cref{lemma: connected double cover for Spinor bundles}.
Using $\DC_T\cong\QC_T$, regard $\SC$ as a vector bundle on $\DC_T$.
Let $j_T\colon\DC_T\hookrightarrow\XC_T$ be the base change of $j_B$ and set $\FC_T=j_{T*}\SC\in\Perf(\XC_T)$.
For every geometric point $\bar b$ of $B$, write $T_{\bar b}=\{s,-s\}$ and let $j_{\bar b}\colon\DC_{\bar b}\hookrightarrow\XC_{\bar b}$ be the base change of $j_B$.
The bundles $\SC|_s$ and $\SC|_{-s}$ are the two spinor bundles on $\DC_{\bar b}\cong Q_{\bar b}$.
By \cref{lemma: BlCPn is rational weak Fano}, the objects
$\FC_T|_s\cong j_{\bar b*}(\SC|_s)$ and
$\FC_T|_{-s}\cong j_{\bar b*}(\SC|_{-s})$
are orthogonal $n$-spherical objects.

Since $n-2$ is odd, the classes of the perfect complexes
\begin{equation}\label{equation: weak fano relative basis}
\begin{gathered}
\pi_B^*\OC_{\PP{n}\times B}(a),
\quad 0\leq a\leq n,\\
\iota_{B*}\bigl(p_B^*\OC_\CC(a)\otimes\iota_B^*\OC_\XC(\RC)\bigr),
\quad 0\leq a\leq n-2,
\end{gathered}
\end{equation}
restrict to a basis of $\Knum(\XC_{\bar b})_{\QB}$ for every geometric point $\bar b$ of $B$.
By \cref{thm:projective_family_stability}, applied to a relative projective embedding defined by a sufficiently large tensor power of $\LC$, we obtain a stability condition $\sigma_\bullet\in\Stab(\XC/B)$.
Together with \eqref{equation: weak fano relative basis}, this gives the data of \cref{setup: disconnectedness in families}, with $\Sigma_1=\Stab^\dagger_L(X)$.
Applying \cref{theorem: disconnectedness in families} to $\XC\to B$, the double cover $T\to B$, and $\FC_T$ proves \labelcref{item: weak Fano stability components} and \labelcref{item: weak Fano central charges}.

\labelcref{item: weak Fano braid orbits}: 
The relative blow-up formula extends $\DB_0$ to a $B$-linear semiorthogonal decomposition of $\Perf(\XC)$, with exceptional components generated relatively by $\pi_B^*\OC_{\PP{n}\times B}(a)$ for $0\leq a\leq n$.
Let $\bar\eta$ be a geometric generic point of $B$ and write $\pi_{\bar\eta}\colon\XC_{\bar\eta}\to\PP{n}_{\bar\eta}$ for the base change of $\pi_B$.
Put $E_{a,\bar\eta}=\pi_{\bar\eta}^*\OC_{\PP{n}_{\bar\eta}}(a)$.
Fix $s\in T_{\bar\eta}$ and put $F_s=\FC_T|_s$.

There is an index $0\leq a\leq n$ such that $\RHom(E_{a,\bar\eta},F_s)\ne0$.
Indeed, if all these complexes vanished, adjunction and generation by
$\OC_{\PP{n}_{\bar\eta}},\ldots,\OC_{\PP{n}_{\bar\eta}}(n)$
would give $\rderived\pi_{\bar\eta*}F_s=0$.
This contradicts the fact that $\rderived\pi_{\bar\eta*}F_s$ is the pushforward of the nonzero spinor bundle $\SC|_s$ under the inclusion $Q_{\bar\eta}\hookrightarrow\PP{n}_{\bar\eta}$.

Therefore, \cref{theorem: separation of braid orbits in families}, applied at $1\in B$ to the vectors $(m,-m)$, shows that the decompositions $\DB_m$ belong to pairwise distinct $B_{n+2}$-orbits.
Finally, \cref{theorem: disconnectedness in families} also shows that $\Psi$ acts trivially on $\Knum(X)$.
Hence corresponding components of $\DB_m$ have images in $\Knum(X)$ independent of $m$.
\end{proof}

\begin{rem}
The weak Fano threefold constructed by Nordskova \cite[Section~3.1]{preprint_anya_2026_a_smooth_projective_counterexample_to_bondalpolishchuks_conjecture} is the blow-up of $\PP{3}$ along a smooth rational curve of degree $6$ on a smooth cubic surface.
When $n=3$, the curve $C$ in \cref{setup: example rational weak fano} has degree $8$ and genus $9$.
Thus, \cref{proposition: disconnectedness for weak Fano blowups} does \emph{not} include Nordskova's example as a special case.
\end{rem}

\subsection{Further examples}\label{subsection: further examples}
Let $Y$ be a del Pezzo Fano threefold of degree $5$ over $\CB$, which has the description of smooth section of $\Gr(2, 5)\subset \PP{9}$ by a linear subspace $\PP{6}\subset \PP{9}$.
In this subsection, we consider the blow-up of $Y$ at a general point.
In this example, the autoequivalences permuting the spherical objects are induced by automorphisms of the blow-up, so the separation criteria apply directly without using their versions for families.
This also gives a counterexample to the converse of \cref{theorem: separation of braid orbits}.

First, we recall some definitions and results concerning $Y$.
Denote $H_Y=\OC_Y(1)$; then we have $\Pic(Y)=\ZZ H_Y$, $H_Y^3=5$, and $-K_Y=2H_Y$.

The construction of Mukai--Umemura, whose details we omit, realizes $Y$ as the closure of the $\PGL_2$-orbit of $[uv(u^4-v^4)]$ in the projective space of binary sextics, with stabilizer the octahedral group $S_4$
\cite{mukai_umemura_1983_minimal_rational_threefolds}.
Moreover, this action identifies $\Aut(Y)$ with $\PGL_2$ by
\cite[Proposition~4.4]{mukai_1988_curves_k3_surfaces_and_fano_3folds_of_genus_leq_10}.
Write $U\cong\PGL_2/S_4$ for the open orbit, fix $p\in U(\CB)$, and put $G=\Aut(Y)_p$.
There are exactly three lines $\ell_1,\ell_2,\ell_3$ through $p$.
Under the equivariant description of the Hilbert scheme of lines, they correspond to the three axes of the octahedron.
Thus, their permutation action is the natural quotient $G\cong S_4\twoheadrightarrow S_3$; see
\cite[Corollary~5.1.5]{kuznetsov_prokhorov_shramov_2018_hilbert_schemes_of_lines_and_conics_and_automorphism_groups_of_fano_threefolds}.

Let $\UC_Y$ be the restriction of the tautological rank-two bundle on $\Gr(2,5)$.
The Grassmannian realization of $Y$ is $\Aut(Y)$-equivariant by
\cite[Lemma~4.5]{mukai_1988_curves_k3_surfaces_and_fano_3folds_of_genus_leq_10}, so $g^*\UC_Y\cong\UC_Y$ for $g\in\Aut(Y)$.
Orlov \cite{orlov_1991_exceptional_set_of_vector_bundles_on_the_variety_v5}
constructed the following full exceptional collection:
\begin{equation*}
D^b(Y)
=\langle\OC_Y,\UC_Y^\vee,\OC_Y(H_Y),\UC_Y^\vee(H_Y)\rangle.
\end{equation*}

\begin{setup}\label{setup: point blowups of del Pezzo threefolds}
Let $\pi\colon X=\Bl_pY\to Y$ be the blow-up with exceptional divisor $\iota\colon R\hookrightarrow X$, and put $H=\pi^*H_Y$.
Then, $D^b(X)$ admits the full exceptional collection
\begin{equation}\label{equation: standard decompositions of point blowups}
\begin{aligned}
\DB\coloneqq\bigl(
\iota_*\iota^*\OC_X(2R),
\iota_*\iota^*\OC_X(R),
\OC_X,
\pi^*\UC_Y^\vee,
\OC_X(H),
\pi^*\UC_Y^\vee\otimes\OC_X(H)
\bigr).
\end{aligned}
\end{equation}

Write $j_i\colon C_i\hookrightarrow X$ for the strict transform of $\ell_i$
and $F_i\coloneqq j_{i*}\OC_{C_i}(-H)$ for $1\leq i\leq3$.
As shown in \cref{lemma: point blowups of del Pezzo threefolds are weak Fano}, $F_1,F_2,F_3$ are mutually orthogonal $3$-spherical objects.
For $\mathbf a=(a_1,a_2,a_3)\in\ZZ^3$ with $a_1+a_2+a_3=0$, put
$\Phi_{\mathbf a}=T_{F_1}^{a_1}T_{F_2}^{a_2}T_{F_3}^{a_3}$ and $\DB_{\mathbf a}=\Phi_{\mathbf a}(\DB)$.

Every element of $G$ lifts to an automorphism of $X$.
We also regard $G$ as a subgroup of $\Aut D^b(X)$ via $g\mapsto(g^{-1})^*$.
Since the lifted automorphisms preserve $H$ and $R$, and $\UC_Y$ is preserved by $\Aut(Y)$, $G$ fixes the isomorphism class of each object in $\DB$.
In particular, $G$ acts trivially on $\Knum(X)$.
Moreover, the action of $G$ on the lines $\ell_i$ induces a transitive action on $\{F_1,F_2,F_3\}$.
Consequently, $[F_1]=[F_2]=[F_3]$ in $\Knum(X)$.
\end{setup}

\begin{lem}\label{lemma: point blowups of del Pezzo threefolds are weak Fano}
In \cref{setup: point blowups of del Pezzo threefolds}, $X$ is weak Fano and the objects $F_1,F_2,F_3$ are mutually orthogonal $3$-spherical objects.
\end{lem}
\begin{proof}
Since $H-R$ is globally generated, $-K_X=2(H-R)$ is nef.
Moreover, $(-K_X)^3=32>0$, so $X$ is weak Fano.

Since $p\in U$, the lines $\ell_i$ have normal bundles
$N_{\ell_i/Y}\cong\OC_{\PP{1}}^{\oplus2}$ by
\cite[Corollary~5.1.5]{kuznetsov_prokhorov_shramov_2018_hilbert_schemes_of_lines_and_conics_and_automorphism_groups_of_fano_threefolds}.
Therefore, we obtain
\[
N_{C_i/X}
\cong N_{\ell_i/Y}\otimes\OC_{\ell_i}(-p)
\cong\OC_{\PP{1}}(-1)^{\oplus2}.
\]
Also, $\omega_X|_{C_i}\cong\OC_{C_i}$ because $(H-R)\cdot C_i=0$.
Hence $F_i$ is $3$-spherical by
\cite[Lemma~3.4]{seidel_thomas_2001_braid_group_actions_on_derived_categories_of_coherent_sheaves}.
Distinct lines through $p$ meet only at $p$ and have distinct tangent directions, so their strict transforms $C_i$ are disjoint.
This gives the orthogonality of $F_1,F_2$, and $F_3$.
\end{proof}

\begin{prop}\label{proposition: disconnectedness for point blowups of del Pezzo threefolds}
In \cref{setup: point blowups of del Pezzo threefolds}, the following hold:
\begin{enumerate}
\item\label{item: point blowup stability components}
$\Stab(X)$ has infinitely many connected components.
\item\label{item: point blowup central charges}
For any ample line bundle $L$ on $X$, every central charge occurring in $\Stab^\dagger_L(X)$ has lifts in infinitely many distinct connected components of $\Stab(X)$.
\item\label{item: point blowup braid orbits}
The full exceptional collections $\DB_{\mathbf a}$ belong to pairwise distinct $B_6$-orbits.
Moreover, the images in $\Knum(X)$ of the corresponding components are independent of $\mathbf a$.
\end{enumerate}
\end{prop}
\begin{proof}
\labelcref{item: point blowup stability components,item: point blowup central charges}:
Fix an ample line bundle $L$ on $X$.
We verify the assumptions of \cref{corollary: disconnectedness from a semilinear equivalence} for the component $\Stab^\dagger_L(X)$.
The objects $F_1,F_2$ are orthogonal $3$-spherical objects by \cref{lemma: point blowups of del Pezzo threefolds are weak Fano}.
By \cref{setup: point blowups of del Pezzo threefolds}, there exists $g\in G$ such that $g(F_1)\cong F_2$, and $g$ acts trivially on $\Knum(X)$.
It remains to find $\sigma_0\in\Stab^\dagger_L(X)$ fixed by $g$.

Since $G$ preserves $H$ and $R$, and $\Pic(X)=\ZZ H\oplus\ZZ R$, it preserves $L$.
For a sufficiently large integer $q$, the complete linear system $|L^{\otimes q}|$ therefore defines a $G$-equivariant closed embedding $h\colon X\hookrightarrow\PP{N}$, where $G$ acts on $\PP{N}$
by projective transformations.
By the construction in
\cref{thm:projective_family_stability},
we can choose $\tau\in\Stab(\PP{N})$ such that
$\sigma_0\coloneqq h^\sharp\tau\in\Stab^\dagger_L(X)$.
By \cref{lemma: projective automorphisms and stability conditions},
$\tau$ is fixed by $\PGL_{N+1}(\CB)$.
The equivariance of $h$ and the definition of the induced stability
condition then imply $g(\sigma_0)=\sigma_0$.
Thus, \cref{corollary: disconnectedness from a semilinear equivalence}
applies and proves
\labelcref{item: point blowup stability components,item: point blowup central charges}.

\labelcref{item: point blowup braid orbits}:
We only need to check \labelcref{item: non-rothogonality condition of exceptional object} of \cref{setup: separation of braid orbits} since the group $G$ preserves each component of $\DB$ and acts transitively on $\{F_1,F_2,F_3\}$.
The line bundle $\OC_X(H)$ satisfies this condition. Indeed, we have 
\[
\RHom_X(\OC_X(H),F_i)
\cong\rderived\Gamma(\PP{1},\OC_{\PP{1}}(-2))
\cong\CB[-1]\ne0.
\]
Thus, the data in \cref{setup: point blowups of del Pezzo threefolds} satisfy the assumptions of \cref{setup: separation of braid orbits}.
By \cref{theorem: separation of braid orbits}, if $\DB_{\mathbf a}$ and $\DB_{\mathbf b}$ belong to the same $B_6$-orbit, then
$\mathbf a-\mathbf b\in\ZZ(1,1,1)$.
Since both vectors have sum zero, this forces $\mathbf a=\mathbf b$.
Finally, 
since $G$ acts transitively on $\{F_1,F_2,F_3\}$, we also have
$[F_1]=[F_2]=[F_3]$ in $\Knum(X)$.
Thus, \cref{remark: numerical images of the separated decompositions}
gives the equality of the corresponding numerical images.
\end{proof}

\begin{rem}\label{remark: counter example of the converse of separation theorem}
The example in this subsection shows that the converse of \cref{theorem: separation of braid orbits} does not hold in general.

Let $G'\subset G$ be the inverse image of the stabilizer of $3$ under the surjection $G\to S_3$.
Then $G'$ acts transitively on $\{F_1,F_2\}$ and preserves each component of $\DB$.
The nonvanishing verified in the proof of
\cref{proposition: disconnectedness for point blowups of del Pezzo threefolds}
shows that $\DB$, $F_1,F_2$, and $G'$ satisfy
\cref{setup: separation of braid orbits} with $d=2$.
Consider the decompositions
$\DB$ and $(T_{F_1}T_{F_2})(\DB)$.
For the pair $F_1,F_2$, their exponent vectors are $(0,0)$ and $(1,1)$, whose difference is diagonal.
Nevertheless, these decompositions belong to distinct $B_6$-orbits.
Indeed, this follows from \cref{theorem: separation of braid orbits} applied to the triple
$F_1,F_2,F_3$ and the group $G$, since $(1,1,0)\notin\ZZ(1,1,1)$.
\end{rem}

\bibliographystyle{amsalpha}
\bibliography{converted_bibtex_tyoshida}
\end{document}